\documentclass{article}%
\usepackage{amssymb}%
\usepackage{amsmath, bbm, mathtools, mathrsfs}%
\usepackage{amsfonts}%
\usepackage{authblk}
\usepackage{graphicx}
\usepackage[a4paper, left=2.5cm, right=2.5cm, top=3cm, bottom=3cm]{geometry}
\usepackage{shuffle}
\usepackage{csquotes}
\usepackage{comment}
\usepackage{multirow}
\usepackage{subcaption}
\usepackage{float}

\providecommand{\U}[1]{\protect\rule{.1in}{.1in}}
\newtheorem{theorem}{Theorem}

\newtheorem{corollary}[theorem]{Corollary}

\newtheorem{definition}[theorem]{Definition}
\newtheorem{example}[theorem]{Example}

\newtheorem{lemma}[theorem]{Lemma}

\newtheorem{proposition}[theorem]{Proposition}
\newtheorem{remark}[theorem]{Remark}

\newenvironment{proof}[1][Proof]{\noindent\textbf{#1.} }{\ \rule{0.5em}{0.5em}}

\newcommand{\floor}[1]{\left \lfloor {#1}\right \rfloor}

\allowdisplaybreaks

\title{Signature Kernel and Schwinger-Dyson Kernel Equations as Two-Parameter Rough Differential Equations}
\author{Thomas Cass, Dan Crisan, Andrea Iannucci\thanks{Corresponding author: \texttt{andrea.iannucci22@imperial.ac.uk}}  , William F. Turner}

\affil{Department of Mathematics, Imperial College London}

\begin{document}
\maketitle
\begin{abstract}
We develop a rough-path framework for two-parameter rough differential equations on rectangular and simplicial domains, motivated by the signature kernel and Schwinger--Dyson kernel equations. The theory is formulated in spaces of jointly controlled rough paths and is based on a robust two-parameter rough integration framework. In particular, we introduce a notion of rough integration over two-dimensional simplices at low regularity extending the results of \cite{cass2023fubini} and \cite{gerasimovivcs2019hormander}.

Within this setting, we show that the signature kernel equation arises naturally as a two-parameter rough differential equation and establish well-posedness and stability. We also extend the Schwinger--Dyson kernel equation, previously formulated for bounded-variation paths, to rough driving signals, proving existence and uniqueness in appropriate controlled rough path spaces.

In the smooth rough path regime, we relate the resulting equations to PDE and integro-differential formulations. Finally, we derive and analyse a numerical scheme for the rough Schwinger--Dyson equation, including runtime and memory complexity estimates, and illustrate its performance with numerical experiments.
\end{abstract}


\medskip


\section{Introduction}

The theory of rough paths provides a deterministic extension of classical differential equation theory, allowing one to give meaning to equations
\[
dZ = V(Z)\, dX, \qquad Z_0 = z,
\]
driven by highly irregular signals, well beyond the Young integration regime, by enriching the driving signal $X$ with additional algebraic structure.

In several recent developments, equations of an analogous type have appeared in genuinely two-parameter form. A first motivating example is the signature kernel equation, introduced in \cite{salvi2021signature} and used in learning problems for sequential data,
\begin{equation}\label{sig_ker_intro}
K(t,v) = 1 + \int_a^t \int_c^v K(s,u)\, \langle dX_s, dY_u \rangle.
\end{equation}
A second motivating example is the Schwinger--Dyson kernel equation, which arises in the study of kernels obtained as limits of random developments of a path in the unitary group \cite{cass2024free},
\begin{equation}\label{sd_ker_intro}
K(s,t) = 1 - \int_s^t \int_s^r K(s,u)K(u,r)\, \langle dX_u, dX_r \rangle,
\end{equation}
and has also been connected to questions in quantum field theory; see \cite{crew2025quantum}. 
In both cases, the driving signals are naturally too irregular to be treated by classical two-variable integration once one leaves the bounded-variation setting.
Beyond providing motivating examples of genuinely two-parameter equations, these kernels also arise in the construction of estimators for the maximum mean discrepancy between distributions on path space; see \cite[Sections 7--8]{chevyrev2022signature} and \cite{chevyrev2016characteristic}. This observation has been used in \cite{issa2023non} and, in the bounded-variation setting, in \cite{lou2023pcf} to show that such kernels can be employed to generate time-series samples from certain target distributions. Their study in the rough setting is therefore particularly relevant, since many driving signals of interest in applications are not of bounded variation but nevertheless admit a natural enhancement to a geometric rough path. Important examples include Brownian motion, more general continuous semimartingales, and sample paths of solutions to stochastic differential equations. The connection we establish between these kernels and two-parameter rough differential equations, and, under additional regularity assumptions, partial differential equations, provides a rough-path framework in which these kernel identities can be interpreted, analysed, and approximated numerically.
Our viewpoint is that \eqref{sig_ker_intro} and \eqref{sd_ker_intro} are instances of a broader class of two-parameter rough differential equations driven by a pair of rough paths $\mathbf{X}$ and $\mathbf{Y}$.
More precisely, we study equations of the form
\begin{equation}\label{skeleton_equation_intro}
Z_{t,v}
= z
+ \int_{\mathcal{D}_{t,v}}
\Theta^g\left(Z\right)_{s,u}\,
d(\mathbf{X}_s,\mathbf{Y}_u),
\end{equation}
posed either on rectangular or simplicial domain $\mathcal{D}_{t,v}$ on which the parameters mapping $\Theta^g$ depends.
We introduce a notion of solution to \eqref{skeleton_equation_intro} as a fixed point in suitable spaces of jointly controlled paths and establish well-posedness and stability estimates for our motivating examples, yielding in particular continuous dependence on the driving rough paths.
This structural well-posedness also leads naturally to numerical approximation schemes.

On the analytical side, our development relies on recent progress in two-parameter rough integration.
Two largely independent approaches have been proposed.
On the one hand, \cite{chouk2014rough} develops an integration theory for rough sheets when the driving signal has finite $p$-variation with $2\le p<3$.
On the other hand, \cite{gerasimovivcs2019hormander} introduces the class of two-parameter controlled paths in the same $p$-variation regime and, under a smooth approximation assumption, establishes a Fubini-type theorem for rough integration over rectangles and two-dimensional simplices.
In this work we adopt the controlled-path viewpoint, as subsequently extended in \cite{cass2023fubini} to a general $p$-variation setting, including the construction of a joint rough integral and a rough Fubini theorem on rectangles that does not rely on smooth approximations.
Motivated by the structure of the Schwinger--Dyson kernel equation, we supplement this theory with a notion of rough integration over simplices at low regularity, thereby bridging the gap between the approaches of \cite{cass2023fubini} and \cite{gerasimovivcs2019hormander}.
In addition, in order to make \eqref{skeleton_equation_intro} well defined, we show that the class of two-parameter controlled paths is stable under composition with smooth functions.

Additionally, when the driving signals are smooth rough paths, introduced in \cite{bellingeri2022smooth}, the two-parameter rough differential equation viewpoint connects naturally with partial differential equations.
In particular, interpreting \eqref{sig_ker_intro} within our framework recovers the PDE derived in \cite{lemercier2024log} in a straightforward manner, and the corresponding uniqueness and approximation statements for the kernel follow from the general two-parameter RDE theory developed here.

Similarly, when $\mathbf{X}$ is a smooth rough path, the Schwinger--Dyson kernel equation admits an equivalent reformulation at the level of an integro-differential equation.
More precisely, exploiting combinatorial identities, derived  from the Schwinger--Dyson equations, we show that is possible to connect the trace component of the kernel to the remaining components of the associated two-parameter controlled path leading to a system of integro-differential equation.
This provides an alternative analytic description of the kernel in the smooth regime and may serve as a starting point for further qualitative analysis.

These combinatorial identities also lead to a numerical scheme, for which we study runtime and memory complexity.
We conclude with numerical experiments illustrating the performance of the proposed approximation schemes: the experiments compare several variants of our method and show that their agreement deteriorates as the driving signal becomes less regular, in line with the theoretical picture.
We also report runtime studies that empirically validate the complexity estimates derived above.

\begin{comment}
\paragraph{Outline of this paper}
The paper is organized as follows. In Section 2, we provide a brief review of the theory of two-parameter controlled paths. Section 3 is dedicated to the construction of the two-parameter rough integral over two-dimensional simplices and its connection to the two-parameter rough integral introduced in \cite{cass2023fubini}. In Section 4, we study two-dimensional rough differential equations. We begin by analysing the composition of two-parameter controlled rough paths with smooth functions and formalize the notion of solution to a two-parameter rough differential equation. Using this framework, we establish uniqueness for the signature kernel equation \eqref{sig_ker_intro} and examine its connection to the corresponding PDE when the driving signals $\mathbf{X}$ and $\mathbf{Y}$ are smooth rough paths.

We then derive the rough version of the Schwinger–Dyson kernel equation and prove well-posedness for the associated two-parameter rough differential equation.

In the last  section a numerical scheme is presented for this equation, including numerical experiments and comparisons with the schemes introduced in \cite{cass2024free}.
Finally, we conclude by deriving an integro-differential equation satisfied by the Schwinger–Dyson kernel in the case where $\mathbf{X}$ is a smooth rough path.

\end{comment}

\paragraph{Outline of the paper}
The paper is organized as follows. In Section~2, we briefly review and relate the theories of one- and two-parameter controlled paths, with an emphasis on the structural analogies that underlie our later constructions. We also collect the main ingredients of two-parameter rough integration used throughout the paper, including the controlled-path formulation and the joint rough integral on rectangles.

In Section~3, motivated by the Schwinger--Dyson kernel equation in the rough regime, we introduce a notion of two-parameter rough integration over two-dimensional simplices at low regularity. We then show that this simplex integral is consistent with the joint rough integral on rectangles introduced in \cite{cass2023fubini}, thereby placing both constructions within a common framework.

Section~4 is devoted to two-parameter rough differential equations. We first prove that the class of two-parameter controlled paths is stable under composition with smooth functions, which is a fundamental ingredient in defining nonlinear equations of the form \eqref{skeleton_equation_intro}. We then formalize the notion of solution to a two-parameter rough differential equation as a fixed point in a suitable jointly controlled path space and develop the corresponding well-posedness theory.

We next apply this framework to our motivating examples. For the signature kernel equation \eqref{sig_ker_intro}, we show that it arises naturally as a two-parameter rough differential equation, and we establish well-posedness and stability, including continuous dependence on the driving rough paths. In the smooth rough path regime, we also recover its associated PDE formulation. We then derive the rough Schwinger--Dyson kernel equation and prove existence and uniqueness for the corresponding two-parameter rough differential equation. In the same smooth regime, we further derive an integro-differential reformulation of the Schwinger--Dyson kernel based on the combinatorial identities developed in the paper.

In the final section, we turn to numerical approximation. We present a numerical scheme for the rough Schwinger--Dyson equation, analyse its runtime and memory complexity, and report numerical experiments illustrating the performance of the method across different regularity regimes. We also include comparisons with the schemes introduced in \cite{cass2024free}, together with runtime experiments that support the theoretical complexity estimates.

\section{Preliminaries for two-parameter controlled paths}
We use this section to develop the theory of two-parameter (jointly controlled) rough paths. To make the definitions accessible it is helpful to keep in mind the analogy with classical, one-parameter rough path theory. The primary achievement of the latter theory is well-posedness and stability for a class of equations of the form
\[
Z_t = Z_0 + \int_{0}^t V(Z_s)\, d\mathbf{X}_s ,
\]
where $\mathbf{X}$ denotes a driving rough path, allowing one to go beyond classical integration theory, which would require bounded variation of the driving signal. Interpreting these equations requires a definition of the rough path $\mathbf{X}$, its properties, and the class of paths that can be integrated against it. The structure of the equation which describe the kernel trick for the signature kernel in the setting of continuous bounded variation paths 
\begin{equation}\label{sig_ker_intro__}
K(t,v) = 1 + \int_a^t \int_c^v K(s,u)\, \langle dX_s, dY_u \rangle,
\end{equation} 
and the Schwinger–Dyson kernel trick \eqref{sd_ker_intro} suggests that the extension of these concepts to the rough path setting is to view them as particular examples of two-parameter rough differential equations of the form
\begin{equation}\label{skeleton_equation}
Z_{t,v}
= z
+ \int_{\mathcal{D}_{t,v}}
\Theta^g\left(Z\right)_{s,u}
\, d(\mathbf{X}_s, \mathbf{Y}_u),
\end{equation}
where our two core examples are obtained from the following choices:
\[
\begin{aligned}
\text{Signature kernel:}\quad
&\begin{aligned}
\mathcal{D}_{t,v} &= [a,t]\times[c,v], \\
\Theta^g(Z)(s,u) &= g(Z)_{s,u}, \\
g(z_1,z_2)[\eta,\xi] &= z_1 \langle \eta,\xi\rangle ;
\end{aligned}
\\[8pt]
\text{Schwinger-Dyson  kernel:}\quad
&\begin{aligned}
\mathcal{D}_{t,v} &= \Delta_{[t,v]}, \\
\Theta^g(Z) &= g(Z_{t, \cdot},Z)_{s,u}, \\
g(z_1,z_2)[\eta,\xi] &= -z_1 z_2 \langle \eta,\xi\rangle .
\end{aligned}
\end{aligned}
\]

The aim of the next three sections is to develop a framework for two-parameter rough differential equations that is sufficiently robust to accommodate our core examples. We will show that both examples fit this framework in a precise sense. As a consequence, the general results of rough path theory become available in this setting, including composition and stability estimates, and the structure of the solutions naturally leads to numerical approximation schemes.

Two-parameter controlled rough paths were first introduced in \cite{gerasimovivcs2019hormander} and further developed in \cite{cass2023fubini}. Building on this theory, we introduce the additional tools needed for our programme: two-parameter controlled paths, the joint rough integral, and a rough Fubini theorem.

For background on classical rough path theory we refer to \cite{cass2022combinatorial, friz2014course}, and for a detailed treatment of smooth rough paths to \cite{bellingeri2022smooth}. The results on two-parameter controlled paths used here are taken primarily from \cite{cass2023fubini}.

\subsection{Review of one-parameter rough path theory}\label{section_review}
Let $U, V$, $W$, and $F$ be real Banach spaces, with $U, V$ and $W$ finite-dimensional.  
We denote by $\mathrm{Hom}(U,F)$ the space of bounded linear maps from $U$ to $F$ equipped with the operator norm, and by $C([a,b],F)$ the Banach space of continuous functions $[a,b] \to F$ equipped with the topology of uniform convergence.  

For a function $f:[a,b]\to F$, we write the increment between $s$ and $t$ (with $a \le s \le t \le b$) as
\[
f_{s,t} := f_t - f_s.
\]

Finally, we denote by $\Delta_{[a,b]}$ the 2-simplex
\[
\Delta_{[a,b]} := \{ (s,t) \in [a,b]^2 : s \le t \}.
\]

\begin{definition}[Control and $\omega$-$p$-variation]
A control on $[a,b]$ is a continuous, super-additive map
$\omega:\Delta_{[a,b]}\to[0,\infty)$ vanishing on the diagonal.

Given $p\ge 1$ and $f:\Delta_{[a,b]}\to F$ set
\[
\|f\|_{p,\omega}
:=
\inf\Big\{C>0:\ \|f_{s,t}\|_F\le C\,\omega(s,t)^{1/p}
\ \text{for all }(s,t)\in\Delta_{[a,b]}\Big\},
\]
with the convention $\inf\varnothing=\infty$, and define
\[
C_\omega^p(\Delta_{[a,b]},F)
:=
\{f:\Delta_{[a,b]}\to F:\ \|f\|_{p,\omega}<\infty\}.
\]

For a path $g:[a,b]\to F$, we write $g\in C_\omega^p([a,b],F)$ if its increment map
\(
(s,t)\to g_{s,t}
\)
belongs to $C_\omega^p(\Delta_{[a,b]},F)$.
\end{definition}

When the control $\omega$ is clear from context, we may omit it from the notation. We also use $C^{0,p}_\omega([a,b],V)$ to denote the closure of $C^1_\omega([a,b],V)$ in $C^p_\omega([a,b],V)$.\newline  

Denote by $U^{\otimes k}$ the 
k-fold tensor product of 
$U$, with the convention that $U^{\otimes 0} = \mathbb{R}$. We endow this spaces with the projective tensor norm.
Let $T((U))$ be the tensor algebra over $V$ with unit $\mathbf{1}$ and let $T^{\leq N}(U)$ its truncation at level $N$, we endow this last space with the $\ell^1$ norm. Let $G^N(U)$ denote the step-$N$ nilpotent group over $U$ equipped with the Carnot–Carathéodory metric, and let $\mathcal{L}^N(U)$ denote the corresponding Lie algebra.

\begin{definition}[Geometric rough paths]
Let $p\ge 1$ and let $\omega$ be a control. The set of geometric $p$-rough paths
$\Omega G_{p,\omega}([a,b],V)$ consists of all maps
\(
\mathbf{X}\in C^{0,p}_\omega\!\left(\Delta_{[a,b]},G^{\lfloor p\rfloor}(V)\right)
\)
that are multiplicative, i.e.
\[
\mathbf{X}_{s,s}=\mathbf{1}
\quad\text{and}\quad
\mathbf{X}_{s,u}=\mathbf{X}_{s,t}\otimes \mathbf{X}_{t,u},
\qquad \text{for all } a\le s\le t\le u\le b.
\]
\end{definition}

A special subclass of geometric rough paths is given by the smooth rough paths. 
These correspond to rough paths whose coordinate functionals are smooth in time, 
so that rough integrals reduce to classical integrals. This class will allow us 
to identify the signature kernel equation with a PDE and the Schwinger–Dyson 
kernel equation with an integro–differential equation.

\begin{definition}[Smooth rough paths]
We say that a rough path $\mathbf{X} \in \Omega G_{p,\omega}
([a,b],V)$ is a smooth rough path if for every
$f \in \mathrm{Hom}(T^{\le \lfloor p \rfloor}(V), \mathbb{R})$, the function
\(
t \to f(\mathbf{X}_{a,t})
\)
belongs to $C^{1}([a,b], \mathbb{R})$.
\end{definition}

The key property linking smooth rough paths with PDE-type equations is that they admit a time derivative along the diagonal. More precisely, by Proposition~2.6 in \cite{bellingeri2022smooth}, for a smooth rough path $\mathbf{X}$ there exists an $\mathcal{L}^{\lfloor p \rfloor}(V)$-valued path $\boldsymbol{x}$ such that for every $s \in [a,b]$
\[
\boldsymbol{x}_s = \partial_t\big|_{t=s}\, \mathbf{X}_{s,t}.
\]
We refer to $\boldsymbol{x}$ as the diagonal derivate of $\mathbf{X}$.

For each, let \(\pi_k : T((U)) \to U^{\otimes k}\) denote the canonical projection onto the \(k\)-th level of the tensor algebra, and write \(\eta^{(k)} := \pi_k(\eta)\) for \(\eta \in T((U))\).  
We also denote by \(\eta^{\le k} :=\pi_{\le k} (\eta)\) (resp $\eta^{\ge k} = \pi_{\ge k}(\eta)$) the projection onto the direct sum of levels up to \(k\) (resp. levels greater than $k$). We use a similar notation for $\eta^{>k}$ (resp. $\eta^{<k}$).\newline 
Fix a basis $(e_i)_{i=1}^d$ of $U$ and let $(e^i)_{i=1}^d$ be the dual basis of $U^*$.
For a word $I=(i_1,\dots,i_n)$ with $0\le n\le N$, set
\[
e_I := e_{i_1}\otimes\cdots\otimes e_{i_n}\in U^{\otimes n},
\qquad
e^I := e^{i_1}\otimes\cdots\otimes e^{i_n}\in (U^*)^{\otimes n}.
\]
Let $\pi_n:T^{\le N}(U)\to U^{\otimes n}$ denote the grade-$n$ projection.
For $\eta\in T^{\le N}(U)$ we define its $(I)$-coordinate by
\[
\eta^{(I)} := \langle e^I,\ \pi_n \eta\rangle
= (e^{i_1}\otimes\cdots\otimes e^{i_n})(\pi_n\eta),
\qquad |I|=n.
\]
For the empty word $\emptyset$, set $\eta^{(\emptyset)}:=\pi_0\eta\in\mathbb{R}$.
We also define $\operatorname{rev}(I)$ to denote the reversal of the
word $I$ with the convention $\operatorname{rev}(\emptyset)=\emptyset$.

\begin{definition}[$\mathbf{X}$-controlled paths]
Let $p \in [1,\infty)$ with $\kappa = \lfloor p \rfloor$, and let $\mathbf{X} \in \Omega G_{p, \omega}([a,b], V)$.  

Let
\[
Z : [a,b] \to \mathrm{Hom}(T^{<\kappa}(V), F),
\]
and suppose there exists 
\[
R : \Delta_{[a,b]} \to \mathrm{Hom}(T^{<\kappa}(V), F)
\]
such that
\begin{equation}\label{1D_controlled}
Z_{t} = Z_{s} \circ L_{\mathbf{X}_{s,t}^{<\kappa}} + R_{s,t}, \quad a \le s \le t \le b,
\end{equation}
where, for every $\eta \in T^{< \kappa}(V)$, 
\(
L_\eta : T((V)) \to T((V))
\) is 
the left tensor multiplication operator, defined by
\(
L_\eta(\xi) := \eta \otimes \xi.
\) 
Then we say that $Z$ is a  $\mathbf{X}$-controlled path if there exists a finite constant $C$ such that for all $a \le s \le t \le b$, any $\eta_j \in V^{\otimes j}$, $j \in [0:\kappa)$,
\begin{equation} \label{regularuty_constraint_remainder_1d}
\| R_{s,t}[\eta_j] \| \le C \, \omega(s,t)^{(\kappa - j)/p} \, \| \eta_j \|.
\end{equation}
\end{definition}
Equipped with the norm 
\[
\|Z\|_{\mathcal{D}^p_{\mathbf{X}}} = \|Z_a\| + \sum_{j=0}^{\kappa-1}\|R^{(j)}\|_{\frac{p}{\floor{p}-j}, \omega_{\mathbf{X}}},
\]
this space is a Banach space, denoted by $\mathcal{D}^p_{\mathbf{X}}([a,b], F)$.\newline

In what follows we will refer to $Z^{(0)}$ as the trace process and $R$ as the remainder process.\newline

Having introduced rough paths and controlled paths, we now turn to the construction of the rough integral.  
Given a rough path $\mathbf{X}$ defined as above and a path $Z$ controlled by $\mathbf{X}$, the rough integral of $Z$ against $\mathbf{X}$ generalizes the classical Young integral to allow for low-regularity integrators.  

A central tool in the construction is the map
\[
\Omega^Z : \Delta_{[a,b]} \to F, 
\qquad 
\Omega^Z(s,t) := Z_s\big[\mathbf{X}_{s,t}\big],
\]
associated with \(Z \in \mathcal{D}^p_{\mathbf{X}}([a,b], \mathrm{Hom}(V,F))\). This map captures the local increment structure of the integral. 

In fact, the rough integral is defined as the limit over
partitions $\mathcal P=\{a=s_0<\cdots<s_n=b\}$ whose
mesh converges to zero:
\[
\int_a^b Z_s\,d\mathbf X_s
:=
\lim_{|\mathcal P|\to0}
\sum_{i=0}^{n-1}\Omega^Z(s_i,s_{i+1}).
\]
More precisely, the rough sewing argument (see
Section~2 of \cite{cass2022combinatorial}) yields
\[
\left\|
\int_a^b Z_s\,d\mathbf X_s
-
\sum_{i=0}^{n-1}\Omega^Z(s_i,s_{i+1})
\right\|
\leq
C\|Z\|_{\mathcal D^p_{\mathbf X}}
\sum_{i=0}^{n-1}
\omega(s_i,s_{i+1})^{(\kappa+1)/p},
\]
which converges to zero as $|\mathcal P|\to0$.
Moreover, rough integration is locally Lipschitz
continuous with respect to both the controlled
integrand and the driving rough path. In particular,
for uniformly bounded controlled pairs,
\[
\left\|
\int_a^b Z\,d\mathbf X
-
\int_a^b \widetilde Z\,d\widetilde{\mathbf X}
\right\|
\leq
C_M\left(d^{\mathbf X,\widetilde{\mathbf X}}_{[a,b]}
(Z,\widetilde Z)
+
\|\mathbf X-\widetilde{\mathbf X}\|_p
\right).
\]

Here $\|\mathbf X\|_p + \|\tilde{\mathbf X}\|_p + \|Z\|_{\mathcal D^p_{\mathbf X}} + \|Z\|_{\mathcal D^p_{\mathbf X}} <M$ and $d^{\mathbf X,\widetilde{\mathbf X}}_{[a,b]}$ denotes the one-parameter controlled-path
distance
\[
d^{\mathbf X,\widetilde{\mathbf X}}_{[a,b]}
(Z,\widetilde Z)
:=
\|Z_a-\widetilde Z_a\|
+
\sum_{j=0}^{\kappa-1}
\left\|
R^{Z,(j)}-R^{\widetilde Z,(j)}
\right\|_{\frac{p}{\kappa-j},\omega}.
\]

The rough integral itself admits a canonical controlled lift
\(
I_\mathbf{X}(Z)\in \mathcal D_X^p([a,b],F).
\)
Its degree-zero component is given by
\[
I_\mathbf{X}(Z)^{(0)}_t
:=
\int_a^t Z_s\,dX_s,
\]
whereas, for every \(1\leq k<\kappa\),
\(
I_\mathbf{X}(Z)^{(k)}=Z^{(k-1)}.
\)
Here we use the canonical identification
\[
\operatorname{Hom}\!\left(
V^{\otimes(k-1)},
\operatorname{Hom}(V,F)
\right)
\simeq
\operatorname{Hom}\!\left(
V^{\otimes k},F
\right),
\]
given by
\[
f
\to
\left[
v_1\otimes\cdots\otimes v_k
\to
f(v_1\otimes\cdots\otimes v_{k-1})(v_k)
\right].
\]
Thus, this identification concerns the positive tensor degrees of the
controlled lift, while the degree-zero component is the rough integral
itself. We use the analogous convention separately in each parameter in
the two-parameter setting below.
\subsection{Review of two-parameter rough path theory}

From the one-parameter rough path theory reviewed above, it is clear that, in order to define a notion of two-parameter (or joint) rough integration suitable for our rough kernel equations, one must first introduce an appropriate class of two-parameter controlled rough paths.

Intuitively, a two-parameter controlled rough path $Z$ is jointly controlled by the rough paths $\mathbf X$ and $\mathbf Y$ if it is controlled in one parameter by $\mathbf X$ and in the other by $\mathbf Y$. The next definition, which adapts the framework of \cite{cass2023fubini}, makes this intuition precise.

\begin{definition}[Two-parameter $(\mathbf X,\mathbf Y)$-controlled paths]
Let $\omega_{\mathbf X}$ and $\omega_{\mathbf Y}$ be controls defined on $[a,b]$ and $[c,d]$, respectively. Let $p,q\in[1,\infty)$, with $\kappa=\lfloor p\rfloor$ and $\lambda=\lfloor q\rfloor$, and assume that
\(
\mathbf X\in\Omega G_{p,\omega_{\mathbf X}}([a,b],V),
\mathbf Y\in\Omega G_{q,\omega_{\mathbf Y}}([c,d],W).
\)
Denote by $\boxtimes_{\mathrm{ext}}$ the external tensor product, and write $\boxtimes$ for $\boxtimes_{\mathrm{ext}}$ in the sequel.

Let
\[
Z:[a,b]\times[c,d]\to
\operatorname{Hom}
\left(
T^{<\kappa}(V)\boxtimes T^{<\lambda}(W),F
\right),
\]
and define the first-order remainders
\begin{align*}
R^1 &: \Delta_{[a,b]}\times[c,d]
\to
\operatorname{Hom}
\left(T^{<\kappa}(V)\boxtimes T^{<\lambda}(W),F\right),\\
R^2 &: [a,b]\times\Delta_{[c,d]}
\to
\operatorname{Hom}
\left(T^{<\kappa}(V)\boxtimes T^{<\lambda}(W),F\right)
\end{align*}
through the identities
\begin{align}
Z_{t,u}
&=
Z_{s,u}\circ
\left[
L_{\mathbf X_{s,t}^{<\kappa}}
\boxtimes id_{T^{<\lambda}(W)}
\right]
+R^1_{s,t}(u),
\nonumber\\
Z_{s,v}
&=
Z_{s,u}\circ
\left[
id_{T^{<\kappa}(V)}
\boxtimes L_{\mathbf Y_{u,v}^{<\lambda}}
\right]
+R^2_{u,v}(s).
\label{constraints}
\end{align}

We say that $Z$ is a two-parameter $(\mathbf X,\mathbf Y)$-controlled path if there exist finite constants $C_1,C_2,C_3$ such that, for every $a\le s\le t\le b$ and $c\le u\le v\le d$, the following conditions hold.

For every $\eta_j\in V^{\otimes j}$, $j\in[0:\kappa)$, and $\xi\in T^{<\lambda}(W)$,
\begin{equation}\label{regularity_condition_r1}
\left\|
R^1_{s,t}(u)[\eta_j\boxtimes\xi]
\right\|
\le
C_1\omega_{\mathbf X}(s,t)^{(\kappa-j)/p}
\|\eta_j\|\|\xi\|.
\end{equation}
Similarly, for every $\eta\in T^{<\kappa}(V)$ and $\xi_k\in W^{\otimes k}$, $k\in[0:\lambda)$,
\begin{equation}\label{regularity_condition_r2}
\left\|
R^2_{u,v}(s)[\eta\boxtimes\xi_k]
\right\|
\le
C_2\omega_{\mathbf Y}(u,v)^{(\lambda-k)/q}
\|\eta\|\|\xi_k\|.
\end{equation}

Moreover, define the second-order remainder
\[
\mathbf R:
\Delta_{[a,b]}\times\Delta_{[c,d]}
\to
\operatorname{Hom}
\left(
T^{<\kappa}(V)\boxtimes T^{<\lambda}(W),F
\right)
\]
by
\begin{align}
\mathbf R\begin{pmatrix}s,t\\u,v\end{pmatrix}
&:=
R^1_{s,t}(v)
-
R^1_{s,t}(u)\circ
\left[
id_{T^{<\kappa}(V)}
\boxtimes L_{\mathbf Y_{u,v}^{<\lambda}}
\right]
\nonumber\\
&=
R^2_{u,v}(t)
-
R^2_{u,v}(s)\circ
\left[
L_{\mathbf X_{s,t}^{<\kappa}}
\boxtimes id_{T^{<\lambda}(W)}
\right].
\label{eq:second_order_remainder}
\end{align}
The equality between the two representations follows algebraically from \eqref{constraints}, since the transport operators acting on the two tensor factors commute.

We additionally require that, for all $\eta_j\in V^{\otimes j}$, $j\in[0:\kappa)$, and $\xi_k\in W^{\otimes k}$, $k\in[0:\lambda)$,
\begin{equation}\label{remainder_remainder_norm}
\left\|
\mathbf R
\begin{pmatrix}s,t\\u,v\end{pmatrix}
[\eta_j\boxtimes\xi_k]
\right\|
\le
C_3
\omega_{\mathbf X}(s,t)^{(\kappa-j)/p}
\omega_{\mathbf Y}(u,v)^{(\lambda-k)/q}
\|\eta_j\|\|\xi_k\|.
\end{equation}
\end{definition}

The external tensor product
\(
T^{<\kappa}(V)\boxtimes T^{<\lambda}(W)
\)
is used mainly as a bookkeeping device, keeping the $V$- and $W$-directions distinct. This is particularly convenient since the two truncated tensor algebras already carry their own internal algebraic structures. Moreover, after choosing bases in $T^{<\kappa}(V)$ and $T^{<\lambda}(W)$, every element admits a unique expansion in the induced tensor-product basis, so that components may be indexed separately in the two directions.

In what follows, for any $\eta\in T((V))\boxtimes T((W))$, we write
\[
\eta^{(j,k)}
:=
(\pi_j\boxtimes\pi_k)(\eta),
\]
where $\pi_j$ and $\pi_k$ denote the canonical projections defined above.

The notion of jointly controlled paths was introduced in \cite{gerasimovivcs2019hormander} for \(p\in[2,3)\) and extended in \cite{cass2023fubini} to arbitrary roughness and to two potentially different driving rough paths. Our formulation packages the compatible mixed controlled coefficients into a single external-tensor-product-valued field and imposes the mixed remainder estimate (12) explicitly. This compatibility is precisely the structure used below in the rough Fubini theorem and in joint rough integration. The following remark illustrates this structure in the low-regularity regime.

\begin{remark}
Consider the case $\mathbf X=\mathbf Y$ and $p\in[2,3)$, so that $\kappa=\lambda=2$. The first identity in \eqref{constraints} becomes
\begin{align*}
Z_{t,u}^{(0,0)}
&=
Z_{s,u}^{(0,0)}
+
Z_{s,u}^{(1,0)}
[\mathbf X_{s,t}^{(1)}\boxtimes\cdot]
+
R_{s,t}^{1;(0,0)}(u),\\
Z_{t,u}^{(1,0)}
&=
Z_{s,u}^{(1,0)}
+
R_{s,t}^{1;(1,0)}(u),\\
Z_{t,u}^{(0,1)}
&=
Z_{s,u}^{(0,1)}
+
Z_{s,u}^{(1,1)}
[\mathbf X_{s,t}^{(1)}\boxtimes\cdot]
+
R_{s,t}^{1;(0,1)}(u),
\end{align*}
while the second identity gives
\begin{align*}
Z_{s,v}^{(0,0)}
&=
Z_{s,u}^{(0,0)}
+
Z_{s,u}^{(0,1)}
[\cdot\boxtimes\mathbf X_{u,v}^{(1)}]
+
R_{u,v}^{2;(0,0)}(s),\\
Z_{s,v}^{(0,1)}
&=
Z_{s,u}^{(0,1)}
+
R_{u,v}^{2;(0,1)}(s),\\
Z_{s,v}^{(1,0)}
&=
Z_{s,u}^{(1,0)}
+
Z_{s,u}^{(1,1)}
[\cdot\boxtimes\mathbf X_{u,v}^{(1)}]
+
R_{u,v}^{2;(1,0)}(s).
\end{align*}

Our definition therefore ensures that \(Z^{(1,1)}\) is the same mixed controlled coefficient in the expansions with respect to both parameters, independently of the order in which these expansions are performed. In the framework of \cite{gerasimovivcs2019hormander}, these operations are initially distinguished, with their compatibility subsequently imposed as a condition for the rough Fubini theorem. Since our applications rely on this property, it is convenient to incorporate it directly into the controlled-path structure.
\end{remark}

As in Lemma~3.3 of \cite{cass2023fubini}, the identities \eqref{constraints} show that $R^1$ and $R^2$ are themselves controlled paths in the second and first parameters, respectively, with common controlled remainder $\mathbf R$ given by \eqref{eq:second_order_remainder}. We refer to $\mathbf R$ also as remainder of the remainder. Its level-wise mixed regularity is captured by the following norm.

\begin{definition}[Joint $(p,q)$-variation norm]
Let $\omega_1,\omega_2$ be controls on $[a,b]$ and $[c,d]$, respectively, and let $p,q\ge1$. For a map
\[
f:\Delta_{[a,b]}\times\Delta_{[c,d]}\to F,
\]
define
\[
\|f\|_{(p,q),(\omega_1,\omega_2)}
:=
\inf
\left\{
C>0:
\left\|
f\begin{pmatrix}s,t\\u,v\end{pmatrix}
\right\|_F
\le
C\,\omega_1(s,t)^{1/p}\omega_2(u,v)^{1/q}
\right\},
\]
where the inequality is required to hold for every
$a\le s\le t\le b$ and $c\le u\le v\le d$, with the convention $\inf\varnothing=\infty$.
\end{definition}

Within this framework, we introduce the following norm on the space of two-parameter $(\mathbf{X},\mathbf{Y})$-controlled paths, defined by 
\[
\begin{aligned}
\|Z\|_{\mathcal{D}_{\mathbf{X}, \mathbf{Y}}}
:= \big| Z_{a,c}  \big| + \sum_{j = 0}^{\kappa -1} \sum_{k = 0}^{\lambda - 1} \Big[
&\;
  \big\| R^{1; (j,k)}(c)
       \big\|_{\frac{p}{\kappa - j}, \omega_\mathbf{X}} + \big\| R^{2; (j,k)}(a)   \big\|_{ \frac{q}{\lambda - k}, \omega_{\mathbf{Y}}}
 + \big\| \mathbf{R}^{(j,k)} \big\|_{(\frac{p}{\kappa - j},\frac{q}{\lambda - k}), (\omega_\mathbf{X}, \omega_\mathbf{Y})} 
\Big],
\end{aligned}
\]
Additionally, for fixed $\tilde{\mathbf{X}} \in  \Omega G_{p, \omega_{\mathbf{X}}}([a,b], V)$, $\tilde{\mathbf{Y}} \in  \Omega G_{q, \omega_{\mathbf{Y}}}([c,d], W)$ and a $(\tilde{\mathbf{X}}, \tilde{\mathbf{Y}})$-controlled path $\tilde{Z}$, it is possible to define the \enquote{distance}
\[
\begin{aligned}
d^{(\mathbf{X}, \mathbf{Y}), (\tilde{\mathbf{X}}, \tilde{\mathbf{Y}})}_{[a,b] \times [c,d]}(Z, \tilde{Z})
:= 
\;\big| Z_{a,c} - \tilde{Z}_{a,c} \big| + \sum_{j = 0}^{\kappa -1} &\sum_{k = 0}^{\lambda - 1} \bigg[  \big\| R^{Z, 1; (j,k)}(c)
      - R^{\tilde{Z}, 1; (j,k)}(c)
       \big\|_{\frac{p}{\kappa - j}, \omega_\mathbf{X}} \\
&\quad
 + \big\| R^{Z, 2; (j,k)}(a)
      - R^{\tilde{Z}, 2; (j,k)}(a)
      \big\|_{\frac{q}{\lambda - k}, \omega_\mathbf{Y}} \\
&\quad
 + \big\| \mathbf{R}^{Z, (j,k)} - \mathbf{R}^{\tilde{Z}, (j,k)} \big\|_{(\frac{p}{\kappa - j},\,\frac{q}{\lambda - k}), (\omega_\mathbf{X}, \omega_\mathbf{Y})}
\bigg].
\end{aligned}
\]
We remark that the above is not a true metric as $Z$ and $\tilde{Z}$ can possibly belong to different spaces. \newline
Just like we mentioned above, when the choice of controls is clear from the context, to improve readability, we will omit the control from the norm and remove the sub and superscripts from the \enquote{distance}.\newline
In the two-parameter setting, this norm necessarily involves the second-order remainder.
Taking increments in both parameters reveals cross-parameter dependence of the first-level remainders, which produces a second-order term.
Hence controlling the increments of the two-parameter integral will require to control of the second-order remainder.
The space of two-parameters controlled rough paths equipped with the norm will be denoted as $\mathcal{D}^{p, q}_{\mathbf{X}, \mathbf{Y}}([a,b] \times [c,d], F)$ throughout this work. Whenever $\mathbf{X} = \mathbf{Y}$ we will simplify the notation by writing $\mathcal{D}^p_{\mathbf{X}}([a,b] \times [a,b], F)$.\newline

As we anticipated before, this class of two-parameter paths is defined so that their interactions with the rough paths $\mathbf{X}$ and $\mathbf{Y}$ are sufficiently regular to allow one to define a rough integral over the rectangle $[a,b]\times[c,d]$. We recall that this joint integral is precisely the one that appears in the rough formulation of the signature-kernel trick. The construction of the integral, and the main properties we use, are summarised below from \cite{cass2023fubini}.\newline
To construct such an integral, one needs an analogue of the map $\Omega$ now familiar from the one-parameter setting. 
In this case, the role is played by the map  
\[
\Omega^Z: \Delta_{[a,b]} \times \Delta_{[c,d]} \to F, \qquad \Omega^Z \begin{pmatrix} s,t \\ u,v \end{pmatrix} := 
Z_{s,u} \big[ \mathbf{X}^{\geq 1}_{s,t} \boxtimes \mathbf{Y}^{\geq 1}_{u,v} \big],
\]  
which encodes how the integrand \(Z\) interacts with the structure of the driving rough paths \(\mathbf{X}\) and \(\mathbf{Y}\). 
Notice that if one of the increments $s,t$ or $u,v$ is fixed, we recover the classical, one-parameter, approximating map acting on either \(\mathbf{X}_{s,t}\) or \(\mathbf{Y}_{u,v}\).

Just like in the one-parameter setting, the next step is to define an approximation of the two-parameter rough integral. 
To this end, fix the two partitions:  
\[
\mathcal{P} = \{ a = s_1 \leq s_2 \leq \dots \leq s_n = b \}, 
\qquad  
\mathcal{D} = \{ c = u_1 \leq u_2 \leq \dots \leq u_m = d \},
\]  
and define the discrete two-parameter integral  
\begin{equation}\label{two_param_sewing}
\sum_{\mathcal{P}\times \mathcal{D}} \Omega^Z 
   := \sum_{i=1}^{n-1} \sum_{j=1}^{m-1} 
   \Omega^Z \begin{pmatrix} s_i, s_{i+1} \\ u_{j}, u_{j+1} \end{pmatrix}.
\end{equation}  
The double sum represents the accumulation of the elementary contributions of \(\Omega^Z\) over each rectangle \([s_i,s_{i+1}]\times[u_j,u_{j+1}]\). 

One of the main results in \cite{cass2023fubini} shows that, as the mesh of the rectangular partition tends to zero, these sums converge to a continuous and additive map on the rectangle $[a,b] \times [c,d]$, which we call two-parameter (or joint) rough integral
\[
\int_{[a,b] \times [c,d]} Z_{s, u} \, d(\mathbf{X}, \mathbf{Y})  
   :=  \lim_{\left|\mathcal{P} \times \mathcal{D}\right| \to 0} 
   \sum_{\mathcal{P}\times \mathcal{D}} \Omega^Z.
\]
Moreover, this integral depends continuously on the pair of rough paths and on the controlled integrand. For convenience, we record in Appendix~\ref{appendix_technical} the precise existence estimate and the corresponding stability bound in the notation used here; see Theorems~\ref{existence_2D_theorem} and~\ref{stability_theorem_2D}.
Additionally, the structure of the two-parameter controlled rough paths allows the the two-parameter rough integral to agree with the iterated one-parameter integrals. Concretely,
using the fact that two-parameter controlled paths are controlled in both parameters to define the maps 
\begin{align*}
&Z \in \mathcal{D}^q_{\mathbf{Y}}([c,d], \mathcal{D}^p_{\mathbf{X}}([a,b], F)) \to \int_a^\cdot Z_{s,u} \mathbf{Y}_u \in \mathcal{D}^p_{\mathbf{X}}([a,b], F) \text{ and } \\
&Z \in \mathcal{D}^p_{\mathbf{X}}([a,b], \mathcal{D}^q_{\mathbf{Y}}([c,d], F)) \to \int_c^\cdot Z_{s,u} \mathbf{X}_s \in \mathcal{D}^q_{\mathbf{Y}}([c,d], F)
\end{align*}
the following identity, derived in Theorem 5.2 in \cite{cass2023fubini}, holds:
\[
\int_{[a,b] \times [c,d]} Z_{s, u} \, d(\mathbf{X}_s, \mathbf{Y}_u)  
= \int_{a}^b \int_{c}^d Z_{s, u} \, d\mathbf{Y}_u \, d\mathbf{X}_s 
= \int_c^d \int_a^b Z_{s, u} \, d\mathbf{X}_s \, d\mathbf{Y}_u.
\]
This property, which is referred to as rough Fubini theorem, allows to define a controlled path  $I_{\mathbf{X}, \mathbf{Y}}(Z) \in \mathcal{D}^{p, q}_{\mathbf{X}, \mathbf{Y}}\big([a, b] \times [c,d],  F\big)$ such that  for $j \in [1:\kappa), k\in [1:\lambda)$ and $\eta \in V^{\otimes j}$, $\xi \in W^{\otimes k}$
\begin{equation}\label{definition_I_X_Y}
\begin{aligned}
    I_{\mathbf{X}, \mathbf{Y}}(Z)_{t,v}[\mathbf{1}, \mathbf{1}] &=  \int_{[a,t] \times [c, v]} Z_{s,u} d(\mathbf{X}_s, \mathbf{Y}_u),\\
    I_{\mathbf{X}, \mathbf{Y}}(Z)_{t,v} [\eta,  \mathbf{1}] &= \int_c^v Z_{t,u} d\mathbf{Y}_u [\eta],  \\
    I_{\mathbf{X}, \mathbf{Y}}(Z)_{t,v} [ \mathbf{1}, \xi] &= \int_a^t Z_{s,v} d\mathbf{X}_s [\xi],\\
    I_{\mathbf{X}, \mathbf{Y}}(Z)_{t,v} [\eta, \xi] &=  Z_{t,v} [\eta, \xi].
\end{aligned}
\end{equation}
That two-parameter rough integrals produce two-parameter controlled paths is consistent with the fact that, as we anticipated, solutions to two-parameter rough differential equations are themselves two-parameter controlled paths. This mirrors the familiar one-parameter theory, where rough integration maps controlled paths to controlled paths, and solutions of rough differential equations are (one-parameter) controlled paths.

\section{Rough integrals on two-dimensional simplices} \label{Simplex_section}

In the previous section, we introduced a notion of joint rough integration
that makes sense of the integral appearing in the signature-kernel trick.
However, this construction does not directly yield a rough integral over
a triangular domain, as required for the Schwinger--Dyson kernel. Indeed,
the partitions in the two directions of the approximating map
\eqref{two_param_sewing} are chosen independently, and the resulting
rectangular cells need not respect the geometry of the simplex,
particularly along its diagonal.

To carry out the programme outlined above, we extend the two-parameter
rough integral from rectangles to two-dimensional simplices for general
two-parameter controlled rough paths. In the setting of
\cite{gerasimovivcs2019hormander}, a joint integral over a triangular
domain is constructed for two-parameter controlled paths admitting
smooth approximations and controlled by a geometric $p$-rough path
with $p\in[2,3)$. Here we broaden this construction to arbitrary
$p\geq 1$ and general two-parameter controlled paths, using a direct
approach based on rough Fubini and sewing estimates.

More precisely, we first establish a rough Fubini theorem showing
that the two orders of integration over the simplex coincide, once
the appropriate diagonal corrections are taken into account. We then
construct a joint approximating map based on two independent partitions,
retaining the usual rectangular contributions away from the diagonal
and introducing local corrections for cells intersecting it. We prove
that these sums converge to the common value of the iterated rough
integrals, thereby defining the two-parameter rough integral over
the simplex. Finally, we establish its stability and show how the
construction relates to the classical joint rough integral on
rectangular domains.

Throughout this section we work with a \(p\)-geometric rough path \(\mathbf{X}\), controlled by \(\omega_{\mathbf{X}}\), and a two-parameter controlled path
\[
Z \in \mathcal{D}^{p}_{\mathbf{X}}\!\left( \Delta_{[a,b]} , \operatorname{Hom}(V \boxtimes V, F)\right).
\]
For notational convenience, set \(\kappa := \lfloor p\rfloor\).

Our first objective is to study the iterated rough integrals 
\begin{equation}\label{iterated_1_d_integrals}
\int_a^b \int_a^u Z_{s, u}\, d\mathbf{X}_s \, d\mathbf{X}_u,
\qquad 
\int_a^b \int_s^b Z_{s, u}\, d\mathbf{X}_u \, d\mathbf{X}_s, 
\end{equation}
with the aim of constructing the corresponding outer integrals. For this, it is necessary to ensure that the inner integrals can be made into 
one-parameter controlled rough paths. To this end, it is convenient to introduce the auxiliary maps
\begin{align*}
Z^L &: [a,b] \to 
\operatorname{Hom}\!\left(
T^{< \kappa}(V),
\operatorname{Hom}(V, F)
\right), \\
Z^U &: [a,b] \to 
\operatorname{Hom}\!\left(
T^{< \kappa}(V),
\operatorname{Hom}(V, F)
\right),
\end{align*}
defined for \(t \in [a,b]\) by
\begin{align*}
Z^L_t[\eta](\cdot) &= Z_{t,t} \left[\Delta \eta\ \otimes ( \mathbf{1} \boxtimes \cdot)\right], \\
Z^U_t[\eta](\cdot) &= Z_{t,t} \left[\Delta \eta\ \otimes (\cdot \boxtimes  \mathbf{1})\right].
\end{align*}
Here $\Delta$ denotes the shuffle coproduct on $T((V))$,
namely the unique continuous algebra morphism into the
completed external tensor product satisfying
\[
\Delta(\mathbf 1)=\mathbf 1\boxtimes \mathbf 1,
\qquad
\Delta(v)=v\boxtimes \mathbf 1+\mathbf 1\boxtimes v,
\qquad v\in V.
\]
Completion and continuity are understood degreewise.
Equivalently, for every non-empty word $I=i_1\cdots i_n$,
\[
\Delta(e_I)
=
\prod_{m=1}^{n}
\bigl(e_{i_m}\boxtimes 1+1\boxtimes e_{i_m}\bigr).
\]
Since $\Delta$ preserves total tensor degree, it restricts
to a linear map
\[
\Delta:T^{<\kappa}(V)
\to
T^{<\kappa}(V)\boxtimes T^{<\kappa}(V),
\]
which we denote by the same symbol and use in the definitions
of $Z^L$ and $Z^U$.
The multiplicativity of $\Delta$ is applied in the full
tensor algebra before taking the required degree projections.

The next proposition shows how these maps arise naturally when identifying the one-parameter controlled path above the inner integrals in \eqref{iterated_1_d_integrals}.

\begin{lemma}\label{proposition_L_U_controlled}
Let $\mathbf{X}$, $Z$, $Z^L$ and $Z^U$ be defined as above. Then the following paths belong to
$\mathcal{D}^p_{\mathbf{X}}([a,b],\operatorname{Hom}(V,F))$:
\[
L_s
:=
\int_a^s Z_{u,s}\,d\mathbf{X}_u+Z_s^L,
\qquad
U_u
:=
\int_u^b Z_{u,s}\,d\mathbf{X}_s-Z_u^U.
\]
\end{lemma}

\begin{proof}
Fix $a\leq s\leq t\leq b$. We first consider
\(
L_t=\int_a^t Z_{u,t}\,d\mathbf{X}_u+Z_t^L.
\)
Splitting the integral at $s$, we obtain
\[
\int_a^t Z_{u,t}\,d\mathbf{X}_u
=
\int_a^s Z_{u,t}\,d\mathbf{X}_u
+
\int_s^t Z_{u,t}\,d\mathbf{X}_u.
\]

On $[a,s]$, we use the controlled expansion in the second
parameter:
\[
Z_{u,t}
=
Z_{u,s}\circ
\left[
id_{T^{<\kappa}(V)}
\boxtimes L_{\mathbf{X}_{s,t}^{<\kappa}}
\right]
+
R^2_{s,t}(u).
\]
On $[s,t]$, we instead expand in the first parameter:
\[
Z_{u,t}
=
Z_{s,t}\circ
\left[
L_{\mathbf{X}_{s,u}^{<\kappa}}
\boxtimes id_{T^{<\kappa}(V)}
\right]
+
R^1_{s,u}(t).
\]
Consequently,
\begin{align*}
\int_a^t Z_{u,t}\,d\mathbf{X}_u
&=
\int_a^s Z_{u,s}\circ
\left[
id_{T^{<\kappa}(V)}
\boxtimes L_{\mathbf{X}_{s,t}^{<\kappa}}
\right]d\mathbf{X}_u
+
\int_a^s R^2_{s,t}(u)\,d\mathbf{X}_u
\\
&\quad+
\int_s^t Z_{s,t}\circ
\left[
L_{\mathbf{X}_{s,u}^{<\kappa}}
\boxtimes id_{T^{<\kappa}(V)}
\right]d\mathbf{X}_u
+
\int_s^t R^1_{s,u}(t)\,d\mathbf{X}_u.
\end{align*}

By the canonical rough integration identity,
\[
\int_s^t Z_{s,t}\circ
\left[
L_{\mathbf{X}_{s,u}^{<\kappa}}
\boxtimes id_{T^{<\kappa}(V)}
\right]d\mathbf{X}_u
=
Z_{s,t}\circ
\left[
\mathbf{X}_{s,t}^{\geq 1}
\boxtimes id_{T^{<\kappa}(V)}
\right].
\]
Expanding $Z_{s,t}$ in its second parameter then gives
\begin{align*}
\int_a^t Z_{u,t}\,d\mathbf{X}_u
&=
\int_a^s Z_{u,s}\circ
\left[
id_{T^{<\kappa}(V)}
\boxtimes L_{\mathbf{X}_{s,t}^{<\kappa}}
\right]d\mathbf{X}_u
\\
&\quad+
Z_{s,s}\circ
\left[
\mathbf{X}_{s,t}^{\geq 1}
\boxtimes L_{\mathbf{X}_{s,t}^{<\kappa}}
\right]
+
\int_a^s R^2_{s,t}(u)\,d\mathbf{X}_u
\\
&\quad+
R^2_{s,t}(s)\circ
\left[
\mathbf{X}_{s,t}^{\geq 1}
\boxtimes id_{T^{<\kappa}(V)}
\right]
+
\int_s^t R^1_{s,u}(t)\,d\mathbf{X}_u.
\end{align*}

Adding $Z_t^L$ and subtracting
$L_s\circ L_{\mathbf{X}_{s,t}^{<\kappa}}$, we obtain
\begin{align}
L_t-L_s\circ L_{\mathbf{X}_{s,t}^{<\kappa}}
&=
Z_{s,s}\circ
\left[
\mathbf{X}_{s,t}^{\geq 1}
\boxtimes L_{\mathbf{X}_{s,t}^{<\kappa}}
\right]
+
Z_t^L-Z_s^L\circ L_{\mathbf{X}_{s,t}^{<\kappa}}
\nonumber\\
&\quad+
\int_a^s R^2_{s,t}(u)\,d\mathbf{X}_u
+
R^2_{s,t}(s)\circ
\left[
\mathbf{X}_{s,t}^{\geq 1}
\boxtimes id_{T^{<\kappa}(V)}
\right]
\nonumber\\
&\quad+
\int_s^t R^1_{s,u}(t)\,d\mathbf{X}_u.
\label{eq:L_corrected_remainder}
\end{align}

We now verify that the right-hand side satisfies
\eqref{regularuty_constraint_remainder_1d}.

First, recall that
\[
Z_t^L[\eta]
=
Z_{t,t}
\left[
\Delta\eta\otimes(\mathbf{1}\boxtimes\cdot)
\right].
\]
The diagonal expansion gives
\begin{align*}
Z_{t,t}
&=
Z_{s,s}\circ
\left[
L_{\mathbf{X}_{s,t}^{<\kappa}}
\boxtimes L_{\mathbf{X}_{s,t}^{<\kappa}}
\right] +
R^2_{s,t}(s)\circ
\left[
L_{\mathbf{X}_{s,t}^{<\kappa}}
\boxtimes id_{T^{<\kappa}(V)}
\right]
+
R^1_{s,t}(t).
\end{align*}

Fix $j<\kappa$ and $\eta_j\in V^{\otimes j}$.
Using the multiplicativity of the shuffle coproduct and
the group-like identity
\[
\Delta\mathbf{X}^{\leq \kappa}_{s,t}
=
\pi_k(\mathbf{X}^{\leq \kappa}_{s,t}\boxtimes\mathbf{X}^{\leq \kappa}_{s,t}),
\]
we compare the expansion of $Z_t^L[\eta_j]$ with that
of
$\left(Z_s^L\circ
L_{\mathbf{X}_{s,t}^{<\kappa}}\right)[\eta_j]$.

In the definition of $Z_t^L$, the terms of
$\Delta\eta_j$ with an empty first tensor factor vanish.
The corresponding terms in
$\Delta(\mathbf{X}_{s,t}^{<\kappa}\eta_j)$ need not vanish,
since the first factor can receive a positive-degree
component of $\mathbf{X}_{s,t}$.
These additional terms are precisely those represented by
\[
Z_{s,s}\circ
\left[
\mathbf{X}_{s,t}^{\geq 1}
\boxtimes L_{\mathbf{X}_{s,t}^{<\kappa}}
\right][\eta_j].
\]
They therefore cancel in the first line of
\eqref{eq:L_corrected_remainder}.

The remaining terms from this comparison either contain
a controlled remainder $R^1$ or $R^2$, or have total
increment degree at least $\kappa-j$. In particular,
the controlled remainder estimates and the rough-path
level bounds yield
\begin{align}
&\left\|
\left(
Z_{s,s}\circ
\left[
\mathbf{X}_{s,t}^{\geq 1}
\boxtimes L_{\mathbf{X}_{s,t}^{<\kappa}}
\right]
+
Z_t^L-Z_s^L\circ L_{\mathbf{X}_{s,t}^{<\kappa}}
\right)[\eta_j]
\right\|
\leq
C\omega_{\mathbf{X}}(s,t)^{(\kappa-j)/p}
\|\eta_j\|.
\label{eq:ZL_diagonal_cancellation}
\end{align}
Here, as throughout, tensor products and compositions are
understood with the canonical degree projections.

Next, the mixed-remainder estimate implies that the path
$u\to R^2_{s,t}(u)$ is $\mathbf{X}$-controlled in its
first parameter, with level-wise controlled norm bounded by
$C\omega_{\mathbf{X}}(s,t)^{(\kappa-j)/p}$.
The continuity estimate for one-parameter rough integration presented in Section \ref{section_review}
therefore gives
\[
\left\|
\left[
\int_a^s R^2_{s,t}(u)\,d\mathbf{X}_u
\right][\eta_j]
\right\|
\leq
C\omega_{\mathbf{X}}(s,t)^{(\kappa-j)/p}
\|\eta_j\|.
\]

Similarly, the bound on $R^2_{s,t}(s)$, together with
the positive degree of $\mathbf{X}_{s,t}^{\geq 1}$,
implies
\[
\left\|
\left(
R^2_{s,t}(s)\circ
\left[
\mathbf{X}_{s,t}^{\geq 1}
\boxtimes id_{T^{<\kappa}(V)}
\right]
\right)[\eta_j]
\right\|
\leq
C\omega_{\mathbf{X}}(s,t)^{(\kappa-j+1)/p}
\|\eta_j\|.
\]

Finally, observe that
\begin{align*}
\int_s^t R^1_{s,u}(t)\,d\mathbf{X}_u
&=
\int_s^t Z_{u,t}\,d\mathbf{X}_u
-
Z_{s,t}\circ
\left[
\mathbf{X}_{s,t}^{\geq 1}
\boxtimes id_{T^{<\kappa}(V)}
\right].
\end{align*}
This is exactly the local sewing error associated with
the one-parameter controlled path $u\to Z_{u,t}$.
Consequently,
\[
\left\|
\left[
\int_s^t R^1_{s,u}(t)\,d\mathbf{X}_u
\right][\eta_j]
\right\|
\leq
C\omega_{\mathbf{X}}(s,t)^{(\kappa+1)/p}
\|\eta_j\|.
\]

Combining the preceding bounds in
\eqref{eq:L_corrected_remainder}, and enlarging $C$ if
necessary, we conclude that
\[
\left\|
\left(
L_t-L_s\circ L_{\mathbf{X}_{s,t}^{<\kappa}}
\right)[\eta_j]
\right\|
\leq
C\omega_{\mathbf{X}}(s,t)^{(\kappa-j)/p}
\|\eta_j\|.
\]
This proves that
$L\in\mathcal{D}^p_{\mathbf{X}}
([a,b],\operatorname{Hom}(V,F))$.

For $U$, we use the analogous decomposition
\[
\int_s^b Z_{s,u}\,d\mathbf{X}_u
=
\int_s^t Z_{s,u}\,d\mathbf{X}_u
+
\int_t^b Z_{s,u}\,d\mathbf{X}_u.
\]
On $[s,t]$ we expand in the second parameter, starting
from $(s,s)$, while on $[t,b]$ we expand in the first
parameter, from $s$ to $t$.

Proceeding as above yields
\begin{align*}
U_t-U_s\circ L_{\mathbf{X}_{s,t}^{<\kappa}}
&=
-Z_{s,s}\circ
\left[
L_{\mathbf{X}_{s,t}^{<\kappa}}
\boxtimes\mathbf{X}_{s,t}^{\geq 1}
\right]
-Z_t^U+Z_s^U\circ L_{\mathbf{X}_{s,t}^{<\kappa}}
\\
&\quad+
\int_t^b R^1_{s,t}(u)\,d\mathbf{X}_u
-
\left(
\int_s^t R^2_{s,u}(s)\,d\mathbf{X}_u
\right)
\circ L_{\mathbf{X}_{s,t}^{<\kappa}}.
\end{align*}

The first line satisfies the required remainder estimate
by the same coproduct argument, with the two tensor factors
interchanged. The first integral in the second line is
estimated using the mixed remainder, whereas the second
is a local one-parameter sewing error. Thus, for every
$j<\kappa$,
\[
\left\|
\left(
U_t-U_s\circ L_{\mathbf{X}_{s,t}^{<\kappa}}
\right)[\eta_j]
\right\|
\leq
C\omega_{\mathbf{X}}(s,t)^{(\kappa-j)/p}
\|\eta_j\|.
\]
Hence
$U\in\mathcal{D}^p_{\mathbf{X}}
([a,b],\operatorname{Hom}(V,F))$,
which completes the proof.
\end{proof}

Having established that the paths $L$ and $U$ are controlled by
$\mathbf X$, we now show that their rough integrals coincide. 
Throughout this section, $\theta := (\kappa + 1)/{p}$ and $C$ denotes a positive constant depending
on the controlled norm of $Z$, the rough-path norm of $\mathbf X$,
$p$, and the control on $[a,b]$. Moreover, for a partition $\mathcal P=\{a=t_0<\cdots<t_N=b\}$,
we denote its mesh with respect to the control $\omega_{\mathbf X}$ by
\(
\delta_\omega(\mathcal P)
:=
\max_{0\le i<N}\omega_{\mathbf X}(t_i,t_{i+1}).
\)

We begin with a uniform estimate for the local error arising when
an inner rough integral is replaced by its canonical approximating
map.

\begin{lemma}\label{lemma_simplex_local_sewing}
Let $Z$ and $\mathbf X$ be as above. For $a\le r<q\le b$, define
\begin{align*}
E^1_{r,q}(u)
&:=
\int_r^q Z_{v,u}\,d\mathbf X_v
-
Z_{r,u}\circ
\left[
\mathbf X_{r,q}^{\ge1}
\boxtimes id_{T^{<\kappa}(V)}
\right],
\qquad u\in[q,b],\\
E^2_{r,q}(s)
&:=
\int_r^q Z_{s,v}\,d\mathbf X_v
-
Z_{s,r}\circ
\left[
id_{T^{<\kappa}(V)}
\boxtimes\mathbf X_{r,q}^{\ge1}
\right],
\qquad s\in[a,r].
\end{align*}
Then $E^1_{r,q}$ and $E^2_{r,q}$ are
$\mathbf X$-controlled in their remaining parameter, and
\begin{equation}\label{eq:simplex_local_sewing}
\|E^1_{r,q}\|_{\mathcal D^p_{\mathbf X}([q,b])}
+
\|E^2_{r,q}\|_{\mathcal D^p_{\mathbf X}([a,r])}
\le C\omega_{\mathbf X}(r,q)^\theta.
\end{equation}
\end{lemma}

\begin{proof}
We prove the assertion for $E^1_{r,q}$; the proof for
$E^2_{r,q}$ follows by interchanging the two parameters.

Fix $q\le u\le v\le b$. The controlled expansion of $Z$
in its second parameter gives
\[
Z_{s,v}
=
Z_{s,u}\circ
\left[
id_{T^{<\kappa}(V)}
\boxtimes L_{\mathbf X_{u,v}^{<\kappa}}
\right]
+
R^2_{u,v}(s).
\]
Integrating over $[r,q]$ in the first parameter and subtracting
the corresponding expansion of the canonical approximating map $\Omega^{R^2}$,
we obtain
\[
E^1_{r,q}(v)
=
E^1_{r,q}(u)\circ L_{\mathbf X_{u,v}^{<\kappa}}
+
\int_r^q R^2_{u,v}(s)\,d\mathbf X_s
-
R^2_{u,v}(r)\circ
\left[
\mathbf X_{r,q}^{\ge1}
\boxtimes id_{T^{<\kappa}(V)}
\right].
\]

For $\xi_j\in V^{\otimes j}$, $j<\kappa$, the path
\[
s\to R^2_{u,v}(s)[\,\cdot\boxtimes\xi_j]
\]
is $\mathbf X$-controlled in its first parameter. Indeed, its
controlled remainder is precisely the mixed remainder $R$,
which satisfies
\[
\left\|
R\binom{s,t}{u,v}[\eta_\ell\boxtimes\xi_j]
\right\|
\le
C\omega_{\mathbf X}(s,t)^{(\kappa-\ell)/p}
\omega_{\mathbf X}(u,v)^{(\kappa-j)/p}
\|\eta_\ell\|\|\xi_j\|.
\]

Applying the one-parameter sewing estimate on $[r,q]$ (see Section \ref{section_review}), we get
\begin{align*}
&\left\|
\left(
\int_r^q R^2_{u,v}(s)\,d\mathbf X_s
-
R^2_{u,v}(r)\circ
\left[
\mathbf X_{r,q}^{\ge1}
\boxtimes id_{T^{<\kappa}(V)}
\right]
\right)[\xi_j]
\right\|\le
C\omega_{\mathbf X}(r,q)^\theta
\omega_{\mathbf X}(u,v)^{(\kappa-j)/p}
\|\xi_j\|.
\end{align*}

The one parameter sewing estimate  applied to
$s\to Z_{s,q}$ also yields
\[
\|E^1_{r,q}(q)\|
\le C\omega_{\mathbf X}(r,q)^\theta.
\]
Together, these estimates prove that $E^1_{r,q}$ is controlled
and satisfies the required norm bound. The second assertion is
obtained in the same way, using $R^1$ in place of $R^2$.
\end{proof}

We can now establish the rough Fubini identity directly on the
simplex.

\begin{proposition}\label{proposition_simplex_fubini}
Let $L$ and $U$ be defined as in
Lemma~\ref{proposition_L_U_controlled}. Then
\begin{equation}\label{identity_tringular_Fubini}
\int_a^b L_t\,d\mathbf X_t
=
\int_a^b U_s\,d\mathbf X_s.
\end{equation}
Moreover, both integrals admit the same triangular
Riemann-sum representation, as specified below.
\end{proposition}

\begin{proof}
Fix a partition
\(
\mathcal P=\{a=t_0<t_1<\cdots<t_N=b\}.
\)

Since $L$ is controlled, its rough integral is the limit of
\[
\sum_{i=0}^{N-1}
L_{t_i}[\mathbf X_{t_i,t_{i+1}}^{\ge1}].
\]
Recalling the definition of $L$, we may decompose the inner
integral according to $\mathcal P$:
\[
\int_a^{t_i}Z_{s,t_i}\,d\mathbf X_s
=
\sum_{j<i}
\int_{t_j}^{t_{j+1}}Z_{s,t_i}\,d\mathbf X_s.
\]

By additivity of the inner rough integral and
Lemma~\ref{lemma_simplex_local_sewing}, we obtain
\begin{align}
\sum_i L_{t_i}
[\mathbf X_{t_i,t_{i+1}}^{\ge1}]
={}&
\sum_{j<i}
Z_{t_j,t_i}
\left[
\mathbf X_{t_j,t_{j+1}}^{\ge1}
\boxtimes
\mathbf X_{t_i,t_{i+1}}^{\ge1}
\right]+
\sum_i Z^L_{t_i}
[\mathbf X_{t_i,t_{i+1}}^{\ge1}]
+
r_{\mathcal P}^L,
\label{eq:simplex_L_sum}
\end{align}
where
\[
r_{\mathcal P}^L
=
\sum_j\sum_{i>j}
E^1_{t_j,t_{j+1}}(t_i)
[\mathbf X_{t_i,t_{i+1}}^{\ge1}].
\]

For fixed $j$, the inner sum is a rough Riemann sum of the
controlled path $E^1_{t_j,t_{j+1}}$ over
$[t_{j+1},b]$. The uniform bound for one-parameter rough
Riemann sums and \eqref{eq:simplex_local_sewing} therefore give
\[
|r_{\mathcal P}^L|
\le
C\sum_j
\omega_{\mathbf X}(t_j,t_{j+1})^\theta.
\]

Similarly, applying the second estimate of
Lemma~\ref{lemma_simplex_local_sewing} to the inner integral
defining $U$, we obtain
\begin{align}
\sum_i U_{t_i}
[\mathbf X_{t_i,t_{i+1}}^{\ge1}]
={}&
\sum_{i<j}
Z_{t_i,t_j}
\left[
\mathbf X_{t_i,t_{i+1}}^{\ge1}
\boxtimes
\mathbf X_{t_j,t_{j+1}}^{\ge1}
\right]
\nonumber\\
&+
\sum_i
\left(
Z_{t_i,t_i}
\left[
\mathbf X_{t_i,t_{i+1}}^{\ge1}
\boxtimes
\mathbf X_{t_i,t_{i+1}}^{\ge1}
\right]
-
Z^U_{t_i}
[\mathbf X_{t_i,t_{i+1}}^{\ge1}]
\right)
+
r_{\mathcal P}^U,
\label{eq:simplex_U_sum}
\end{align}
where
\[
r_{\mathcal P}^U
=
\sum_j\sum_{i\le j}
E^2_{t_j,t_{j+1}}(t_i)
[\mathbf X_{t_i,t_{i+1}}^{\ge1}].
\]

For fixed $j$, the terms with $i<j$ form a rough Riemann
sum of $E^2_{t_j,t_{j+1}}$ over $[a,t_j]$, whereas the
remaining term $i=j$ is estimated directly by the local
sewing bound. Consequently,
\[
|r_{\mathcal P}^U|
\le
C\sum_j
\omega_{\mathbf X}(t_j,t_{j+1})^\theta.
\]

The off-diagonal sums in
\eqref{eq:simplex_L_sum} and \eqref{eq:simplex_U_sum}
are identical. It therefore remains to compare their
diagonal contributions.

Recall that
\[
Z_t^L[\eta](\cdot)
=
Z_{t,t}
[\Delta\eta\otimes(\mathbf 1\boxtimes\cdot)],
\qquad
Z_t^U[\eta](\cdot)
=
Z_{t,t}
[\Delta\eta\otimes(\cdot\boxtimes\mathbf 1)].
\]
By multiplicativity of the shuffle coproduct,
\(
\Delta(\eta\otimes v)
=
\Delta\eta\,
\bigl(v\boxtimes\mathbf 1+\mathbf 1\boxtimes v\bigr).
\)
The two terms correspond to assigning the final letter
of a word to either tensor factor. Thus, under the
canonical identification used for rough integration,
\begin{align}
&Z^L_{t_i}
[\mathbf X_{t_i,t_{i+1}}^{\ge1}]
+
Z^U_{t_i}
[\mathbf X_{t_i,t_{i+1}}^{\ge1}]=
Z_{t_i,t_i}
\left[
(\pi_{\ge1}\boxtimes\pi_{\ge1})
\Delta\mathbf X_{t_i,t_{i+1}}^{\le\kappa}
\right].
\label{eq:simplex_coproduct_identity}
\end{align}

By definition $\mathbf X$ is group-like.
Consequently, the right-hand side agrees with
\(
Z_{t_i,t_i}
\left[
\mathbf X_{t_i,t_{i+1}}^{\ge1}
\boxtimes
\mathbf X_{t_i,t_{i+1}}^{\ge1}
\right]
\)
up to terms whose total tensor degree is at least
$\kappa+1$. The rough-path level estimates therefore imply
\begin{align}
&\left\|
Z^L_{t_i}
[\mathbf X_{t_i,t_{i+1}}^{\ge1}]
+
Z^U_{t_i}
[\mathbf X_{t_i,t_{i+1}}^{\ge1}]
-
Z_{t_i,t_i}
\left[
\mathbf X_{t_i,t_{i+1}}^{\ge1}
\boxtimes
\mathbf X_{t_i,t_{i+1}}^{\ge1}
\right]
\right\|
\nonumber\\
&\hspace{3cm}
\le
C\omega_{\mathbf X}(t_i,t_{i+1})^\theta.
\label{eq:simplex_diagonal_estimate}
\end{align}

Combining
\eqref{eq:simplex_L_sum},
\eqref{eq:simplex_U_sum}, and
\eqref{eq:simplex_diagonal_estimate}, we obtain
\[
\left\|
\sum_i L_{t_i}
[\mathbf X_{t_i,t_{i+1}}^{\ge1}]
-
\sum_i U_{t_i}
[\mathbf X_{t_i,t_{i+1}}^{\ge1}]
\right\|
\le
C\sum_i
\omega_{\mathbf X}(t_i,t_{i+1})^\theta.
\]

By superadditivity of the control,
\[
\sum_i
\omega_{\mathbf X}(t_i,t_{i+1})^\theta
\le
\omega_{\mathbf X}(a,b)
\left(
\max_i\omega_{\mathbf X}(t_i,t_{i+1})
\right)^{\theta-1}.
\]
Since $\theta>1$, the right-hand side converges to zero
as the mesh of $\mathcal P$ tends to zero.

Both outer rough Riemann sums converge by
Lemma~\ref{proposition_L_U_controlled}. Passing to the
limit proves \eqref{identity_tringular_Fubini}.

Furthermore, \eqref{eq:simplex_L_sum} yields
\begin{align}
\int_a^b L_t\,d\mathbf X_t
=
\lim_{|\mathcal P|\to0}
\Bigg(
&\sum_{j<i}
Z_{t_j,t_i}
\left[
\mathbf X_{t_j,t_{j+1}}^{\ge1}
\boxtimes
\mathbf X_{t_i,t_{i+1}}^{\ge1}
\right]
+
\sum_i Z^L_{t_i}
[\mathbf X_{t_i,t_{i+1}}^{\ge1}]
\Bigg).
\label{eq:simplex_riemann_representation}
\end{align}
By \eqref{eq:simplex_diagonal_estimate}, the same limit
is obtained by replacing the diagonal contribution with
\[
Z_{t_i,t_i}
\left[
\mathbf X_{t_i,t_{i+1}}^{\ge1}
\boxtimes
\mathbf X_{t_i,t_{i+1}}^{\ge1}
\right]
-
Z^U_{t_i}
[\mathbf X_{t_i,t_{i+1}}^{\ge1}].
\]
This completes the proof.
\end{proof}

\begin{definition}[Two-parameter rough integral over the simplex]
\label{definition_simplex_integral}
Let
\[
Z\in
\mathcal D^p_{\mathbf X}
\left(
\Delta_{[a,b]},\operatorname{Hom}(V\boxtimes V,F)
\right).
\]
We define
\begin{equation}\label{defining_identity_2p_simplex}
\begin{aligned}
\int_{\Delta_{[a,b]}}
Z_{s,u}\,d(\mathbf X_s,\mathbf X_u)
&:=
\int_a^b L_t\,d\mathbf X_t =
\int_a^b U_s\,d\mathbf X_s,
\end{aligned}
\end{equation}
where the equality follows from
Proposition~\ref{proposition_simplex_fubini}.
\end{definition}

We next establish a partition-limit representation which retains
the independent discretisations in the two parameters.

For intervals $I=[r,q]$ and $J=[u,v]$ satisfying $q\le u$,
write
\[
\Omega^Z(I,J)
:=
Z_{r,u}\left[
\mathbf X_{r,q}^{\ge1}
\boxtimes
\mathbf X_{u,v}^{\ge1}
\right].
\]
For a partition $\mathcal P$ of $[a,b]$, define
\begin{equation}\label{eq:simplex_single_partition_sum}
S^\Delta_{\mathcal P}(Z)
:=
\sum_{i<j}
\Omega^Z([t_i,t_{i+1}],[t_j,t_{j+1}])
+
\sum_i Z^L_{t_i}
[\mathbf X_{t_i,t_{i+1}}^{\ge1}].
\end{equation}

By \eqref{eq:simplex_L_sum} and the ordinary
one-parameter sewing estimate, we have
\begin{equation}\label{estimate_triangular_approximation}
\begin{aligned}
&\left\|
S^\Delta_{\mathcal P}(Z)
-
\int_{\Delta_{[a,b]}}
Z_{s,u}\,d(\mathbf X_s,\mathbf X_u)
\right\|
\\
&\qquad\leq
C\sum_i\omega_{\mathbf X}(t_i,t_{i+1})^\theta
\\
&\qquad\leq
C\omega_{\mathbf X}(a,b)
\delta_\omega(\mathcal P)^{\theta-1}.
\end{aligned}
\end{equation}

We now consider two arbitrary, independently chosen partitions
\[
\mathcal P_1=\{a=s_0<\cdots<s_n=b\},
\qquad
\mathcal P_2=\{a=u_0<\cdots<u_m=b\},
\]
and write
\(
I_i=[s_i,s_{i+1}],
\) and \(J_j=[u_j,u_{j+1}].
\)
Denote their common refinement by
\(
\mathcal Q:=\mathcal P_1\vee\mathcal P_2.
\)
For an interval $I$ whose endpoints belong to $\mathcal Q$,
let $\mathcal Q|_I$ be its subdivision induced by $\mathcal Q$.

The cells of $\mathcal P_1\times\mathcal P_2$ are divided
into three cases:
\[
\begin{array}{ll}
s_{i+1}\le u_j:
& I_i\times J_j\text{ lies in the simplex},\\
u_{j+1}\le s_i:
& I_i\times J_j\text{ lies outside the simplex},\\
\max(s_i,u_j)<\min(s_{i+1},u_{j+1}):
& I_i\times J_j\text{ crosses the diagonal}.
\end{array}
\]

For a diagonal-crossing cell, define its contribution by
\begin{align}
&\sum_{\substack{
K\in\mathcal Q|_{I_i},\
H\in\mathcal Q|_{J_j}\\
\sup K\le\inf H
}}
\Omega^Z(K,H)
+
\sum_{K\in\mathcal Q|_{I_i\cap J_j}}
Z^L_{\inf K}[\mathbf X_K^{\ge1}].
\label{eq:simplex_crossing_cell}
\end{align}
Here \(\mathbf X_K^{\ge1}\) is the positive-degree rough-path increment over \(K\).

Thus, the independent-product-partition approximating sum is
\begin{align}
S^\Delta_{\mathcal P_1,\mathcal P_2}(Z)
:={}&
\sum_{\substack{i,j\\s_{i+1}\le u_j}}
\Omega^Z(I_i,J_j)
\nonumber\\
&+
\sum_{\substack{i,j\\
\max(s_i,u_j)<\min(s_{i+1},u_{j+1})}}
\Bigg\{
\sum_{\substack{
K\in\mathcal Q|_{I_i},\
H\in\mathcal Q|_{J_j}\\
\sup K\le\inf H
}}
\Omega^Z(K,H)
+
\sum_{K\in\mathcal Q|_{I_i\cap J_j}}
Z^L_{\inf K}[\mathbf X_K^{\ge1}]
\Bigg\}.
\label{eq:simplex_product_partition_sum}
\end{align}
In particular, every rectangle wholly contained in the
simplex retains its original, unrefined Cass--Pei
approximating map. The common refinement is used only to
resolve cells crossing the diagonal.

\begin{proposition}\label{proposition_simplex_product_partitions}
Let $Z$ and $\mathbf X$ be as above. Then
\begin{equation}\label{eq:simplex_independent_partition_limit}
\int_{\Delta_{[a,b]}}
Z_{s,u}\,d(\mathbf X_s,\mathbf X_u)
=
\lim_{\substack{
\delta_\omega(\mathcal P_1)\to0\\
\delta_\omega(\mathcal P_2)\to0
}}
S^\Delta_{\mathcal P_1,\mathcal P_2}(Z).
\end{equation}
More precisely,
\begin{align}
&\left\|
S^\Delta_{\mathcal P_1,\mathcal P_2}(Z)
-
\int_{\Delta_{[a,b]}}
Z_{s,u}\,d(\mathbf X_s,\mathbf X_u)
\right\|\le
C\omega_{\mathbf X}(a,b)
\left[
\delta_\omega(\mathcal P_1)^{\theta-1}
+
\delta_\omega(\mathcal P_2)^{\theta-1}
\right].
\label{eq:simplex_product_partition_estimate}
\end{align}
No restriction is imposed on the relative mesh sizes of
$\mathcal P_1$ and $\mathcal P_2$.
\end{proposition}

\begin{proof}
Set $\mathcal Q=\mathcal P_1\vee\mathcal P_2$.
By construction, the contribution of every cell crossing
the diagonal is exactly its contribution to
$S^\Delta_{\mathcal Q}(Z)$.

Consequently,
\begin{equation}\label{eq:simplex_refinement_difference}
S^\Delta_{\mathcal P_1,\mathcal P_2}(Z)
-
S^\Delta_{\mathcal Q}(Z)
=
\sum_{\substack{i,j\\s_{i+1}\le u_j}}
\left\{
\Omega^Z(I_i,J_j)
-
\sum_{\substack{
K\in\mathcal Q|_{I_i}\\
H\in\mathcal Q|_{J_j}
}}
\Omega^Z(K,H)
\right\}.
\end{equation}

We estimate this difference by refining first in the first
parameter and then in the second parameter.

For $I=[r,q]$ and a subdivision $\mathcal R$ of $I$, set
\[
\begin{aligned}
D^1_{I,\mathcal R}(u)
:={}&
Z_{r,u}\circ
[\mathbf X_I^{\ge1}\boxtimes id]
-
\sum_{K\in\mathcal R}
Z_{\inf K,u}\circ
[\mathbf X_K^{\ge1}\boxtimes id],
\qquad u\ge q.
\end{aligned}
\]
Additivity of the exact inner rough integral gives
\[
D^1_{I,\mathcal R}
=
-E^1_{r,q}
+
\sum_{K\in\mathcal R}E^1_{\inf K,\sup K}.
\]
Hence Lemma~\ref{lemma_simplex_local_sewing} and
superadditivity of the control imply
\begin{equation}\label{eq:simplex_first_refinement_bound}
\|D^1_{I,\mathcal R}\|_{\mathcal D^p_{\mathbf X}([q,b])}
\le
C\omega_{\mathbf X}(r,q)^\theta.
\end{equation}
Crucially, the constant is independent of the subdivision
$\mathcal R$.

For the second parameter, define analogously
\[
\begin{aligned}
D^2_{J,\mathcal R}(s)
:={}&
Z_{s,\inf J}\circ
[id\boxtimes\mathbf X_J^{\ge1}]
-
\sum_{H\in\mathcal R}
Z_{s,\inf H}\circ
[id\boxtimes\mathbf X_H^{\ge1}],
\qquad s\le\inf J.
\end{aligned}
\]
The same argument yields
\begin{equation}\label{eq:simplex_second_refinement_bound}
\|D^2_{J,\mathcal R}\|_
{\mathcal D^p_{\mathbf X}([a,\inf J])}
\le C\omega_{\mathbf X}(u_j, u_{j+1})^\theta.
\end{equation}

Insert an intermediate sum in
\eqref{eq:simplex_refinement_difference} in which only
the first parameter is refined.

The first-coordinate contribution is
\[
\sum_i
\sum_{\substack{j\\s_{i+1} \le u_j}}
D^1_{I_i,\mathcal Q|_{I_i}}(u_j)
[\mathbf X_{J_j}^{\ge1}].
\]
For fixed $i$, the admissible intervals $J_j$ form a
consecutive tail of $\mathcal P_2$. Thus the inner sum is
a one-parameter rough Riemann sum of
$D^1_{I_i,\mathcal Q|_{I_i}}$.

Using \eqref{eq:simplex_first_refinement_bound}, we obtain
\begin{align*}
\left\|\text{first-coordinate contribution}\right\|
&\le C\sum_i\omega_{\mathbf X}(s_i, s_{i+1})^\theta\\
&\le
C\omega_{\mathbf X}(a,b)
\delta_\omega(\mathcal P_1)^{\theta-1}.
\end{align*}

For the second-coordinate contribution, fix $j$ instead.
The admissible intervals $I_i$, characterized by
$s_{i+1}\le u_j$, form a consecutive prefix of
$\mathcal P_1$. Their subintervals in $\mathcal Q$
therefore form a partition of an interval contained
in $[a,u_j]$.

By \eqref{eq:simplex_second_refinement_bound}, the
corresponding rough Riemann sum is bounded by
$C\omega_{\mathbf X}(u_j, u_{j+1})^\theta$. Summing over $j$
gives
\[
\left\|\text{second-coordinate contribution}\right\|
\le
C\omega_{\mathbf X}(a,b)
\delta_\omega(\mathcal P_2)^{\theta-1}.
\]

Combining the two estimates,
\begin{align}
&\left\|
S^\Delta_{\mathcal P_1,\mathcal P_2}(Z)
-
S^\Delta_{\mathcal Q}(Z)
\right\| \le
C\omega_{\mathbf X}(a,b)
\left[
\delta_\omega(\mathcal P_1)^{\theta-1}
+
\delta_\omega(\mathcal P_2)^{\theta-1}
\right].
\label{eq:simplex_refinement_estimate}
\end{align}

Finally, since $\mathcal Q$ refines both partitions,
\(
\delta_\omega(\mathcal Q)
\le
\min\{
\delta_\omega(\mathcal P_1),
\delta_\omega(\mathcal P_2)
\}.
\)
Applying \eqref{estimate_triangular_approximation}
and combining this estimate with
\eqref{eq:simplex_refinement_estimate} proves
\eqref{eq:simplex_product_partition_estimate}.
The desired limit follows because $\theta>1$.
\end{proof}

\begin{corollary}[Identification of the simplicial integral]
\label{corollary_simplex_identification}
Under the assumptions of
Proposition~\ref{proposition_simplex_fubini},
the independent product-partition sums satisfy
\begin{align}
&\lim_{\substack{
\delta_\omega(\mathcal P_1)\to0\\
\delta_\omega(\mathcal P_2)\to0
}}
S^\Delta_{\mathcal P_1,\mathcal P_2}(Z)=
\int_a^b
\left(
\int_a^t Z_{s,t}\,d\mathbf X_s+Z_t^L
\right)d\mathbf X_t
\nonumber\\
&\qquad=
\int_a^b
\left(
\int_s^b Z_{s,u}\,d\mathbf X_u-Z_s^U
\right)d\mathbf X_s.
\label{eq:simplex_joint_iterated_identity}
\end{align}
In particular, the limit is independent of the relative
refinement rates of the two partitions.
\end{corollary}

\begin{proof}
Let $\mathcal Q=\mathcal P_1\vee\mathcal P_2$.
The refinement estimate established in
Proposition~\ref{proposition_simplex_product_partitions}
shows that
\[
S^\Delta_{\mathcal P_1,\mathcal P_2}(Z)
-
S^\Delta_{\mathcal Q}(Z)
\to0.
\]
By the single-partition representation \eqref{eq:simplex_riemann_representation},
\[
S^\Delta_{\mathcal Q}(Z)
\to
\int_a^b L_t\,d\mathbf X_t.
\]
Finally, Proposition~\ref{proposition_simplex_fubini}
identifies this integral with
$\int_a^b U_s\,d\mathbf X_s$.
\end{proof}

\begin{proposition}[Stability of the simplicial rough integral]
\label{proposition_simplex_stability}
Let $\mathbf X, \widetilde{\mathbf X} \in \Omega G_{p, \omega}([a, b], V)$
and let
\[
Z\in\mathcal D^p_{\mathbf X}
\left(\Delta_{[a,b]},\operatorname{Hom}(V\boxtimes V,F)\right),
\]
\[
\widetilde Z\in
\mathcal D^p_{\widetilde{\mathbf X}}
\left(\Delta_{[a,b]},\operatorname{Hom}(V\boxtimes V,F)\right).
\]
Assume that their rough-path and controlled-path norms are
bounded by $M$.

Then, for a constant $C_M$ depending only on these bounds,
$p$, and the control on $[a,b]$,
\begin{align}
&\left\|
\int_{\Delta_{[a,b]}}
Z_{s,u}\,d(\mathbf X_s,\mathbf X_u)
-
\int_{\Delta_{[a,b]}}
\widetilde Z_{s,u}\,
d(\widetilde{\mathbf X}_s,\widetilde{\mathbf X}_u)
\right\|\le
C_M\left(
d(Z,\widetilde Z)
+
\|\mathbf X-\widetilde{\mathbf X}\|_p
\right),
\label{eq:simplex_stability}
\end{align}
where $d$ denotes the controlled-path distance introduced
in Section~2, restricted to the simplex.
\end{proposition}

\begin{proof}
The definition of the simplex integral and
Corollary~\ref{corollary_simplex_identification} give
\[
\int_{\Delta_{[a,b]}}
Z_{s,u}\,d(\mathbf X_s,\mathbf X_u)
=
\int_a^b L_t\,d\mathbf X_t,
\]
and analogously
\[
\int_{\Delta_{[a,b]}}
\widetilde Z_{s,u}\,
d(\widetilde{\mathbf X}_s,\widetilde{\mathbf X}_u)
=
\int_a^b\widetilde L_t\,
d\widetilde{\mathbf X}_t.
\]

We first establish stability of the map $Z\to L$.

The diagonal correction
\[
Z_t^L[\eta](\cdot)
=
Z_{t,t}
[\Delta\eta\otimes(\mathbf1\boxtimes\cdot)]
\]
depends algebraically and linearly on the controlled
coefficients of $Z$. Consequently, its difference from
the corresponding correction for $\widetilde Z$ is bounded
by the controlled-path distance $d(Z,\widetilde Z)$.

For the integral term defining $L$, we apply the
stability estimate for one-parameter rough integration
in the first variable. Its dependence on the remaining
parameter is controlled by the second-variable remainder
$R^2$, whose first-variable controlled remainder is the
mixed remainder $R$.

In particular, the difference version of the local sewing
estimate in Lemma~\ref{lemma_simplex_local_sewing} gives
\begin{align*}
&\left\|
E^{1,\mathbf X}_{r,q}(Z)
-
E^{1,\widetilde{\mathbf X}}_{r,q}(\widetilde Z)
\right\|_{\mathcal D^p}\le
C_M
\left(
d(Z,\widetilde Z)
+
\|\mathbf X-\widetilde{\mathbf X}\|_p
\right)
\omega(r,q)^\theta.
\end{align*}
Indeed, this follows by applying the usual stability
estimate for the sewing map to the two local approximating
maps. The difference of their controlled remainders is
estimated using the differences of $R^1$, $R^2$, and the
mixed remainder $R$, all of which enter the distance
$d(Z,\widetilde Z)$.

Applying the same difference estimates to the remainder
decomposition in the proof of
Lemma~\ref{proposition_L_U_controlled}, including its
diagonal coproduct cancellation, yields
\[
d(L,\widetilde L)
\le
C_M\left(
d(Z,\widetilde Z)
+
\|\mathbf X-\widetilde{\mathbf X}\|_p
\right).
\]

Finally, the stability of one-parameter rough
integration gives
\begin{align*}
&\left\|
\int_a^bL_t\,d\mathbf X_t
-
\int_a^b\widetilde L_t\,
d\widetilde{\mathbf X}_t
\right\|\le
C_M\left(
d(L,\widetilde L)
+
\|\mathbf X-\widetilde{\mathbf X}\|_p
\right).
\end{align*}
Combining the preceding estimates proves
\eqref{eq:simplex_stability}.

Linearity in $Z$ for a fixed driving rough path follows
from the linearity of the inner rough integral, the
diagonal correction $Z^L$, and the outer rough integral.
The same estimates give boundedness.
\end{proof}

\begin{corollary}[Two-parameter controlled lift of the simplicial rough integral]
\label{corollary_simplex_controlled_lift}
Let
\[
Z\in\mathcal D^p_{\mathbf X}
\left(
\Delta_{[a,b]},
\operatorname{Hom}(V\boxtimes V,F)
\right).
\]
Then the simplicial rough integral admits a two-parameter
controlled lift
\(
I^\Delta_{\mathbf X}(Z)
\in
\mathcal D^p_{\mathbf X}(\Delta_{[a,b]},F),
\)
whose zeroth-order component is given by
\begin{equation}\label{eq:simplex_controlled_lift}
I^\Delta_{\mathbf X}(Z)_{s,t}
[\mathbf1,\mathbf1]
=
\int_{\Delta_{[s,t]}}
Z_{u,r}\,d(\mathbf X_u,\mathbf X_r),
\qquad a\le s\le t\le b.
\end{equation}

To describe its higher-order components, let
\[
L_r^s
=
\int_s^r Z_{u,r}\,d\mathbf X_u+Z_r^L,
\qquad
U_u^t
=
\int_u^t Z_{u,r}\,d\mathbf X_r-Z_u^U,
\]
where $Z^L$ and $Z^U$ are the diagonal corrections
introduced previously.

For $1\le j,k<\kappa$,
$\eta\in V^{\otimes(j-1)}$,
$\xi\in V^{\otimes(k-1)}$ and $v,w\in V$,
the remaining coefficients are defined by
\begin{align*}
&I^\Delta_{\mathbf X}(Z)_{s,t}
[\eta\otimes v,\mathbf1]
=
-U_s^t[\eta](v),\\
&I^\Delta_{\mathbf X}(Z)_{s,t}
[\mathbf1,\xi\otimes w]
=
L_t^s[\xi](w),\\
&I^\Delta_{\mathbf X}(Z)_{s,t}
[\eta\otimes v,\xi\otimes w]
=
-Z_{s,t}[\eta,\xi](v\boxtimes w).
\end{align*}

Furthermore, under the assumptions of
Proposition~\ref{proposition_simplex_stability},
the lift satisfies
\begin{align}
&d\left(
I^\Delta_{\mathbf X}(Z),
I^\Delta_{\widetilde{\mathbf X}}(\widetilde Z)
\right)
\le
C_M\left(
d(Z,\widetilde Z)
+
\|\mathbf X-\widetilde{\mathbf X}\|_p
\right),
\label{eq:simplex_controlled_lift_stability}
\end{align}
with 
$ \|Z\|_{\mathcal D^{p,q}_{\mathbf X}}
+\|\widetilde Z\|_{
\mathcal D^{p,q}_{\widetilde{\mathbf X}}}
+\|\mathbf X\|_p
+\|\widetilde{\mathbf X}\|_p
\le M.$
\end{corollary}

\begin{proof}
Write
\[
I_{s,t}
:=
\int_{\Delta_{[s,t]}}
Z_{u,r}\,d(\mathbf X_u,\mathbf X_r),
\qquad a\le s\le t\le b.
\]

By Corollary~\ref{corollary_simplex_identification},
applied on each subinterval $[s,t]$,
\begin{equation}\label{eq:simplex_lift_fubini}
I_{s,t}
=
\int_s^t L_r^s\,d\mathbf X_r
=
\int_s^t U_u^t\,d\mathbf X_u.
\end{equation}
The proof of Lemma~\ref{proposition_L_U_controlled}
applies equally to the endpoint-dependent paths
$L^s$ and $U^t$, with estimates uniform in the
endpoints. Consequently,
\[
I^\Delta_{\mathbf X}(Z)_{s,t}
[\eta\otimes v,\mathbf1]
=
-U_s^t[\eta](v),
\]
and
\[
I^\Delta_{\mathbf X}(Z)_{s,t}
[\mathbf1,\xi\otimes w]
=
L_t^s[\xi](w).
\]

Moreover, the diagonal corrections cancel when
taking differences with respect to the integration
endpoints. In particular,
\begin{align*}
L_t^{s'}-L_t^s
&=
-\int_s^{s'}Z_{u,t}\,d\mathbf X_u,\\
U_s^{t'}-U_s^t
&=
\int_t^{t'}Z_{s,r}\,d\mathbf X_r.
\end{align*}
The corresponding one-parameter controlled lifts
therefore identify the mixed coefficients as
\[
I^\Delta_{\mathbf X}(Z)_{s,t}
[\eta\otimes v,\xi\otimes w]
=
-Z_{s,t}[\eta,\xi](v\boxtimes w).
\]

It remains to verify the mixed remainder estimate.
Fix $s\le s'\le t\le t'$. By additivity of the
simplicial integral and its identification with
the iterated rough integrals,
\begin{align}
&I_{s',t'}-I_{s,t'}-I_{s',t}+I_{s,t}=
-\int_{[s,s']\times[t,t']}
Z_{u,r}\,d(\mathbf X_u,\mathbf X_r).
\label{eq:simplex_lift_mixed_increment}
\end{align}
Notice that the rectangle on the right-hand side
is entirely contained in the original simplex.

Together with the preceding identities for
$L^s$ and $U^t$, this identifies the mixed
controlled expansion with the negative of the
corresponding rectangular controlled integral.

We may therefore apply the rectangular integration
estimates recalled in Section~2 to the restriction
of $Z$ to $[s,s']\times[t,t']$. In particular, the
mixed remainder satisfies, for $0\le j,k<\kappa$,
\[
\left\|
\mathbf R^I
\begin{pmatrix}
s,s'\\ t,t'
\end{pmatrix}
[\eta_j,\xi_k]
\right\|
\le
C\omega(s,s')^{(\kappa-j)/p}
\omega(t,t')^{(\kappa-k)/p}
\|\eta_j\|\|\xi_k\|.
\]
The uniform one-parameter remainder estimates,
including those at the diagonal, follow from
Lemma~\ref{proposition_L_U_controlled} and
Lemma~\ref{lemma_simplex_local_sewing}.

Finally, the two representations of the second-order
remainder coincide algebraically, exactly as in
the definition of two-parameter controlled paths,
since the transport operators acting on the two
tensor factors commute.

This proves
\[
I^\Delta_{\mathbf X}(Z)
\in\mathcal D^p_{\mathbf X}(\Delta_{[a,b]},F).
\]

For stability, the difference estimates established
in the proof of
Proposition~\ref{proposition_simplex_stability}
apply uniformly to the endpoint-dependent
paths $L^s$ and $U^t$.
Combining them with the stability estimates for
one-parameter rough integration and the corresponding
rectangular controlled integration estimates gives
\[
d\left(
I^\Delta_{\mathbf X}(Z),
I^\Delta_{\widetilde{\mathbf X}}(\widetilde Z)
\right)
\le
C_M\left(
d(Z,\widetilde Z)
+
\|\mathbf X-\widetilde{\mathbf X}\|_p
\right).
\]
Linearity and boundedness follow from the same
constructions.
\end{proof}

We conclude by relating the simplicial construction to the
joint rough integral on rectangular domains.

Let
\[
\Delta^+_{[a,b]}
=
\{(s,u)\in[a,b]^2:s\le u\},
\qquad
\Delta^-_{[a,b]}
=
\{(s,u)\in[a,b]^2:s\ge u\}.
\]
For a jointly controlled field $W^-$ on the lower triangle,
define its transpose on the upper triangle by
\[
(\mathsf T W^-)_{s,u}[\eta,\xi](v,w)
:=
W^-_{u,s}[\xi,\eta](w,v),
\qquad s\le u.
\]
The transpose includes the exchange of both the controlled
tensor slots and the two integration slots.

\begin{definition}[Piecewise simplicially controlled fields]
Let
$\mathcal D_{\mathbf X}^{p,\triangle}([a,b]^2, F)$
denote the space of pairs $(W^+,W^-)$ such that
\[
W^+\in\mathcal D_{\mathbf X}^p(\Delta^+_{[a,b]}, F),
\qquad
\mathsf T W^-\in
\mathcal D_{\mathbf X}^p(\Delta^+_{[a,b]}, F).
\]

For $W=(W^+,W^-)$ in this space, define
\begin{equation}\label{eq:piecewise_square_integral}
\int_{[a,b]^2}^{\triangle}
W_{s,u}\,d(\mathbf X_s,\mathbf X_u)
:=
\int_{\Delta_{[a,b]}}
W^+_{s,u}\,d(\mathbf X_s,\mathbf X_u)
+
\int_{\Delta_{[a,b]}}
(\mathsf T W^-)_{s,u}\,
d(\mathbf X_s,\mathbf X_u).
\end{equation}
\end{definition}

\begin{proposition}[Compatibility with the rectangular rough integral]
\label{proposition_simplex_Cass_Pei_compatibility}
Suppose that $W \in \mathcal{D}^p_{\mathbf X}([a,b]^2, F)$ and let $W^+$ and $W^-$ be its restrictions to the
two triangles. Then
\begin{equation}\label{eq:simplex_CP_compatibility}
\int_{[a,b]^2}^{\triangle}
W_{s,u}\,d(\mathbf X_s,\mathbf X_u)
=
\int_{[a,b]^2}
W_{s,u}\,d(\mathbf X_s,\mathbf X_u),
\end{equation}
where the right-hand side is the Cass--Pei joint rough
integral.
\end{proposition}

\begin{proof}
Fix a partition
\[
\mathcal P=\{a=t_0<\cdots<t_N=b\}.
\]
Apply \eqref{estimate_triangular_approximation} to $W^+$
and $\mathsf T W^-$. Their two triangular approximating
sums are
\begin{align*}
S^\Delta_{\mathcal P}(W^+)
&=
\sum_{i<j}
W_{t_i,t_j}
\left[
\mathbf X_{t_i,t_{i+1}}^{\ge1}
\boxtimes
\mathbf X_{t_j,t_{j+1}}^{\ge1}
\right]
+
\sum_i
(W^+)^L_{t_i}
\left[\mathbf X_{t_i,t_{i+1}}^{\ge1}\right],
\\
S^\Delta_{\mathcal P}(\mathsf T W^-)
&=
\sum_{i>j}
W_{t_i,t_j}
\left[
\mathbf X_{t_i,t_{i+1}}^{\ge1}
\boxtimes
\mathbf X_{t_j,t_{j+1}}^{\ge1}
\right]
+
\sum_i
(\mathsf T W^-)^L_{t_i}
\left[\mathbf X_{t_i,t_{i+1}}^{\ge1}\right].
\end{align*}

Since $W$ is globally jointly controlled, transposition
interchanges the two diagonal boundary maps. Thus, the
last term in the second identity corresponds to the
diagonal correction $W^U$.

By the coproduct identity
\eqref{eq:simplex_coproduct_identity} and the estimate
\eqref{eq:simplex_diagonal_estimate}, we obtain
\[
\begin{aligned}
&S^\Delta_{\mathcal P}(W^+)
+
S^\Delta_{\mathcal P}(\mathsf T W^-)=
\sum_{i,j}
W_{t_i,t_j}
\left[
\mathbf X_{t_i,t_{i+1}}^{\ge1}
\boxtimes
\mathbf X_{t_j,t_{j+1}}^{\ge1}
\right]
+
r_{\mathcal P},
\end{aligned}
\]
where
\[
|r_{\mathcal P}|
\le
C\sum_i
\omega_{\mathbf X}(t_i,t_{i+1})^\theta
\to 0
\qquad\text{as }|\mathcal P|\to0.
\]

The double sum is precisely the Cass--Pei approximating
map on the product partition
$\mathcal P\times\mathcal P$.
By the rectangular rough integration theorem (Theorem \ref{existence_2D_theorem}), it converges
to the right-hand side of
\eqref{eq:simplex_CP_compatibility}.

On the other hand, the two triangular sums converge to
the corresponding terms in
\eqref{eq:piecewise_square_integral}.
Passing to the limit proves the result.
\end{proof}

\begin{remark}
For a triangular controlled field $Z$, the reflected object
can be interpreted as a piecewise simplicially controlled
field whose upper component is $Z$ and whose lower
component is its transpose. By construction,
\[
\int_{[a,b]^2}^{\triangle}
\widehat Z_{s,u}\,d(\mathbf X_s,\mathbf X_u)
=
2\int_{\Delta_{[a,b]}}
Z_{s,u}\,d(\mathbf X_s,\mathbf X_u).
\]
This identity does not require $\widehat Z$ to belong to the
globally jointly controlled space on $[a,b]^2$.
\end{remark}

\section{Differential equation for jointly controlled paths} \label{composition and stability}
We begin this section by introducing the last ingredient needed to give a rigorous meaning to the class of equations described by \eqref{skeleton_equation}, which encompasses the signature kernel trick and the Schwinger-Dyson kernel equation. 
Indeed, the formulation of \eqref{skeleton_equation} involves joint rough integrals whose integrands are obtained by applying a smooth function $g$ to a two-parameter controlled path. 
Thus, before the equation can be interpreted, one must verify that this nonlinear transformation preserves the two-parameter controlled structure. 
This is exactly what we prove below: for any smooth function $g$ and two-parameter controlled rough path $Z$, it is possible to construct a two-parameter controlled rough path above $g(Z^{(0,0)})$.

The construction is the natural two-parameter analogue of composition for one-parameter controlled paths. In the latter case, the higher-order components of $g(Z)$ are obtained from the truncated Taylor expansion of $g$, applied to the family
\[
Z\circ\widetilde{\Delta},\qquad
Z^{\otimes 2}\circ\widetilde{\Delta}^{\,2},\qquad
\ldots,
\]
where the shuffle coproduct distributes the rough directions among the factors appearing in each derivative of $g$. In the two-parameter setting, the same principle applies, with the product coproduct $\Delta\boxtimes\Delta$ governing the mixed components. Schematically, these are obtained from terms of the form
\[
\frac{1}{\ell!}
D^\ell g\left(Z^{(0,0)}\right)
\circ Z^{\otimes \ell}
\circ \widetilde{(\Delta\boxtimes\Delta)}^{\,\ell}.
\]
Here, if $\textrm{proj}_{\neq(0,0)}$ denotes the projection which removes the joint degree $(0,0)$, we set
\[
\widetilde{(\Delta\boxtimes\Delta)}^{\,\ell}
:=
\pi_{\neq(0,0)}^{\otimes \ell}
\circ
(\Delta\boxtimes\Delta)^\ell.
\]
In particular, components of degree $(i,0)$ or $(0,j)$ are retained, while only the joint degree $(0,0)$ is removed. The proposition below makes this construction precise and proves that it indeed defines a jointly controlled path.

Having established this stability under smooth composition, we then introduce a notion of solution to \eqref{skeleton_equation} within the class of jointly controlled paths.
In particular, we formulate \eqref{skeleton_equation} as a fixed-point problem in an appropriate controlled-path space and define solutions as jointly controlled paths satisfying the equation in this sense.

\begin{proposition}\label{stability_under_composition}
Let $S$ be a finite-dimensional Banach space. Let $\mathbf{X}\in \Omega G_p([a,b],V)$ and $\mathbf{Y}\in \Omega G_q([c,d],W)$, and set $\kappa=\lfloor p\rfloor$ and $\lambda=\lfloor q\rfloor$. Let $g\colon F\to S$ be Fr\'echet smooth, and let $Z$ be a two-parameter $(\mathbf{X},\mathbf{Y})$-controlled path.
Then the composition $g(Z)$ defines a two-parameter controlled path
\[
g(Z) \in \mathcal{D}^{p,q}_{\mathbf{X},\mathbf{Y}}([a,b]\times[c,d], S).
\]
Moreover, $g(Z)$ is characterized by the identity
\[
g(Z)_{t,v}[\eta,\xi] = g\left(Z^{(0,0)}_{t,v}\right)[\eta,\xi],
\qquad
\text{for all } \eta \in V^{\otimes 0},\ \xi \in W^{\otimes 0}.
\]

The higher level components of $g(Z)$ are given as follows. For
$\xi_j \in V^{\otimes j}$ and $\eta_k \in W^{\otimes k}$, define
\[
g(Z_{t,v})[\xi_j,\eta_k] :=
\begin{cases}
\displaystyle
\displaystyle
\sum_{m=1}^{\kappa-1} \frac{1}{m!}
D^m g\left(Z^{(0,0)}_{t,v}\right) \circ Z_{t,v}^{\otimes m}
\circ \left(\tilde{\Delta}^m \boxtimes \mathrm{id}_{T^{<\lambda(W)}}\right)
[\xi_j,\eta_k],
& j>0,\ k=0, \\[2ex]
\displaystyle
\sum_{n=1}^{\lambda-1} \frac{1}{n!}
D^n g\left(Z^{(0,0)}_{t,v}\right) \circ Z_{t,v}^{\otimes n}
\circ \left(\mathrm{id}_{T^{<\kappa(V)}} \boxtimes \tilde{\Delta}^n\right)
[\xi_j,\eta_k],
& j=0,\ k>0, \\[2ex]
\displaystyle
\sum_{\ell=1}^{\lambda+\kappa-2} \frac{1}{\ell!}
D^\ell g\left(Z^{(0,0)}_{t,v}\right) \circ Z_{t,v}^{\otimes \ell}
\circ \widetilde{(\Delta \boxtimes \Delta)}^{\ell}
[\xi_j,\eta_k],
& j>0,\ k>0 .
\end{cases}
\]
Here, $\Delta$ denotes the shuffle coproduct, $\tilde{\Delta}$
denotes the restriction $\pi_{\geq 1} \circ \Delta$ and $\Delta^m$ denotes the $m$-fold iterate of $\Delta$.
\end{proposition}
The previous result ensures that the term $g(Z)$ is again a two-parameter controlled path, so the joint integrals in \eqref{skeleton_equation} can be defined.  
For well-posedness our two examples, however, this qualitative closure is not enough:  one must control the size of $g(Z)$ in the controlled norm and compare $g(Z)$ and $g(\tilde Z)$ when either the input $Z$ or the driving rough paths are perturbed.  
The next proposition provides precisely these quantitative bounds, showing that composition defines a (locally) Lipschitz map with respect to the distance $d$.
\begin{proposition}\label{stability_proposition}
Let $g\in C^\infty$ and fix $M>0$. Let
\(
\mathbf X,\widetilde{\mathbf X}
\in\Omega G_{p,\omega_{\mathbf X}}([a,b],V),
\mathbf Y,\widetilde{\mathbf Y}
\in\Omega G_{q,\omega_{\mathbf Y}}([c,d],W),
\)
and
\(
Z\in\mathcal D^{p,q}_{\mathbf X,\mathbf Y},
\widetilde Z\in
\mathcal D^{p,q}_{\widetilde{\mathbf X},
\widetilde{\mathbf Y}}.
\)
Suppose that
\begin{align*}
&\|Z\|_{\mathcal D^{p,q}_{\mathbf X,\mathbf Y}}
+\|\widetilde Z\|_{
\mathcal D^{p,q}_{\widetilde{\mathbf X},
\widetilde{\mathbf Y}}}
+\|\mathbf X\|_p
+\|\widetilde{\mathbf X}\|_p
+\|\mathbf Y\|_q
+\|\widetilde{\mathbf Y}\|_q
\le M.
\end{align*}
     The composition of $Z$ with $g$ satisfies 
    \[
    \|g(Z)\|_{\mathcal{D}^{p,q}_{\mathbf{X}, \mathbf{Y}}} \leq C_{p,q,g,M} (g(Z)_{a,c} + \|Z\|_{\mathcal{D}^{p,q}_{\mathbf{X}, \mathbf{Y}}}),
    \]
    moreover the following stability estimate holds
    \[
    d(g(Z), g(\tilde{Z})) \leq C_{p,q,g, M}( d(Z, \tilde{Z}) + \|\mathbf{X} - \tilde{\mathbf X}\|_p + \|\mathbf{Y} - \tilde{\mathbf Y}\|_q ).
    \]

\end{proposition}

We postpone the proof of these results to the appendix in order to maintain the flow of the exposition.\newline
To formulate the Schwinger--Dyson equation studied in
Section~4.2 as an identity of jointly controlled paths,
we require a parameter-dependent extension of Corollary~16.
Indeed, its integrand involves products of the form
$Z_{s,u}Z_{u,v}$, so the initial endpoint $s$ enters both
the integrand and the integration domain $\Delta_{[s,t]}$.
The following lemma constructs the corresponding controlled
lift by combining the parameter dependence of the integrand
with the endpoint dependence already handled by
Corollary~16.

\begin{lemma}[Lift of parameter-dependent simplicial integrals]
\label{lemma_parameter_simplicial_integrals}
Let $\mathbf{X}\in \Omega G_p([a,b],V)$.
For each $r\in[a,b]$, let
\[
H^r:
T^{<\kappa}(V)
\to
\mathcal D^p_{\mathbf X}(\Delta_{[r,b]},\operatorname{Hom}(V\boxtimes V,F))
\]
be linear.

Assume that, for some $M<\infty$,
\[
\|H^r[\eta]\|_{\mathcal D^p_{\mathbf X}(\Delta_{[r,b]})}
\leq M\|\eta\|,
\]
uniformly in $r$ and $\eta\in T^{<\kappa}(V)$.
For $r\leq r'$, define, on $\Delta_{[r',b]}$,
\[
\mathcal R^H_{r,r'}[\eta]
:=
H^{r'}[\eta]
-
H^r\!\left[
\pi_{<\kappa}
\bigl(\mathbf X_{r,r'}\otimes\eta\bigr)
\right].
\]
Suppose that, for every $\eta_j\in V^{\otimes j}$,
$0\leq j<\kappa$,
\[
\|\mathcal R^H_{r,r'}[\eta_j]\|
_{\mathcal D^p_{\mathbf X}(\Delta_{[r',b]})}
\leq
M\omega(r,r')^{(\kappa-j)/p}\|\eta_j\|.
\]

Then the map
\[
(s,t)\to
\int_{\Delta_{[s,t]}}
H^s[\mathbf 1]_{u,v}\,d(\mathbf X_u,\mathbf X_v)
\]
admits a canonical jointly controlled lift
\(
I^\Delta_{\mathbf X}(H)
\in\mathcal D^p_{\mathbf X}(\Delta_{[a,b]},F).
\)

More precisely, for $r\leq s\leq t$, set
\[
A_{r;s,t}[\eta;\zeta,\xi]
:=
I^\Delta_{\mathbf X}\bigl(H^r[\eta]\bigr)_{s,t}
[\zeta,\xi].
\]
The lift is defined by
\[
I^\Delta_{\mathbf X}(H)_{s,t}
:=
A_{s;s,t}\circ
\bigl(
\Delta\boxtimes\operatorname{id}_{T^{<\kappa}(V)}
\bigr).
\]
Moreover,
\[
\|I^\Delta_{\mathbf X}(H)\|_{
\mathcal D^p_{\mathbf X}(\Delta_{[a,b]})}
\leq CM,
\]
where $C$ depends only on $p$, the rough-path bounds,
and the control on $[a,b]$.
\end{lemma}

\begin{proof}
Let $J := I^\Delta_{\mathbf X}(H)$. Apply Corollary~16 to each integrand $H^r[\eta]$.
By linearity and compatibility with restriction,
the remainder of $A$ in its external parameter is
\[
\begin{aligned}
&A_{r';s,t}[\eta;\zeta,\xi]
-
A_{r;s,t}
\!\left[
\pi_{<\kappa}
(\mathbf X_{r,r'}\otimes\eta);
\zeta,\xi
\right] =
I^\Delta_{\mathbf X}
\bigl(\mathcal R^H_{r,r'}[\eta]\bigr)_{s,t}
[\zeta,\xi],
\qquad r'\leq s\leq t.
\end{aligned}
\]
The boundedness assertion in Corollary~16 therefore gives
the required external remainder estimate in the full
controlled norm of the variables $(s,t)$.

We now restrict to the diagonal $r=s$.
All left multiplication operators below are understood
with projection onto $T^{<\kappa}(V)$.
For $s\leq s'\leq t$, inserting the intermediate value
$A_{s;s',t}$ gives
\[
\begin{aligned}
&A_{s';s',t}
-
A_{s;s,t}\circ
\bigl(
L_{\mathbf X^{<\kappa}_{s,s'}}
\boxtimes
L_{\mathbf X^{<\kappa}_{s,s'}}
\boxtimes\operatorname{id}_{T^{<\kappa}(V)}
\bigr)
\\
&=
\Bigl[
A_{s';s',t}
-
A_{s;s',t}\circ
\bigl(
L_{\mathbf X^{<\kappa}_{s,s'}}
\boxtimes\operatorname{id}_{T^{<\kappa}(V)}\boxtimes\operatorname{id}_{T^{<\kappa}(V)}
\bigr)
\Bigr]
\\
&\quad+
\Bigl[
A_{s;s',t}
-
A_{s;s,t}\circ
\bigl(
\operatorname{id}_{T^{<\kappa}(V)}
\boxtimes L_{\mathbf X^{<\kappa}_{s,s'}}
\boxtimes\operatorname{id}_{T^{<\kappa}(V)}
\bigr)
\Bigr]
\circ
\bigl(
L_{\mathbf X^{<\kappa}_{s,s'}}
\boxtimes\operatorname{id}_{T^{<\kappa}(V)}\boxtimes\operatorname{id}_{T^{<\kappa}(V)}
\bigr).
\end{aligned}
\]
The first bracket is the external-parameter remainder
estimated above. The second is the lower-endpoint
remainder supplied by Corollary~16.
At degrees $(\ell,m)$ in the first two tensor slots,
their controlled norms in the terminal variable are
bounded, respectively, by
\[
CM\omega_{\mathbf X}(s,s')^{(\kappa-\ell)/p},
\qquad
CM\omega_{\mathbf X}(s,s')^{(\kappa-m)/p}.
\]
These bounds include the terminal controlled remainders;
for the second bracket, this follows from the mixed
remainder estimate in Corollary~16.
Precomposition with left multiplication changes only
the constant in the second bound.

Since $\Delta\eta_j$ has only bidegrees $(\ell,m)$
with $\ell+m=j$, composing the preceding identity with
$\Delta\boxtimes\operatorname{id}$ bounds both remainder
terms by
\[
CM\omega_{\mathbf X}(s,s')^{(\kappa-j)/p}\|\eta_j\|
\]
in the controlled norm of the terminal variable.
Here additional powers of $\omega_{\mathbf X}(a,b)$ are absorbed
into $C$.

It remains to compare the transported coproduct with
the coproduct of the transported argument.
The group-like identity and multiplicativity of the
full coproduct show that
\[
\Bigl[
\bigl(
L_{\mathbf X^{<\kappa}_{s,s'}}
\boxtimes L_{\mathbf X^{<\kappa}_{s,s'}}
\bigr)\circ\Delta
-
\Delta\circ L_{\mathbf X^{<\kappa}_{s,s'}}
\Bigr]\eta_j
\]
has no component of total degree below $\kappa$.
A surviving component of bidegree $(\ell,m)$ contains
rough-path increments whose degrees sum to
$\ell+m-j\geq\kappa-j$.
The rough-path level bounds therefore give
\[
\left\|
\Bigl[
\bigl(
L_{\mathbf X^{<\kappa}_{s,s'}}
\boxtimes L_{\mathbf X^{<\kappa}_{s,s'}}
\bigr)\circ\Delta
-
\Delta\circ L_{\mathbf X^{<\kappa}_{s,s'}}
\Bigr]\eta_j
\right\|
\leq
C\omega_{\mathbf X}(s,s')^{(\kappa-j)/p}\|\eta_j\|.
\]
Combining this estimate with the uniform terminal
controlled bounds for $A$ yields
\[
\left\|
t\to
R^{J,1}_{s,s'}(t)[\eta_j,\cdot]
\right\|_{\mathcal D^p_{\mathbf X}([s',b],F)}
\leq
CM\omega_{\mathbf X}(s,s')^{(\kappa-j)/p}\|\eta_j\|.
\]

The terminal controlled coefficients in this estimate
are precisely the corresponding components of
$R^{J,1}_{s,s'}$, since all precomposition operators
above are independent of the terminal variable.
Its controlled remainder is therefore the mixed
remainder of $J$. Consequently,
\[
\left\|
\mathbf R^J
\begin{pmatrix}
s,s'\\ t,t'
\end{pmatrix}
[\eta_j,\xi_k]
\right\|
\leq
CM\omega_{\mathbf X}(s,s')^{(\kappa-j)/p}
\omega_{\mathbf X}(t,t')^{(\kappa-k)/p}
\|\eta_j\|\|\xi_k\|.
\]
The coefficient bounds and the first-order estimates
in the terminal variable follow directly from
Corollary~16.
Finally, the two representations of the mixed
remainder agree algebraically because the transport
operators in the two tensor slots of $J$ commute.
These estimates prove the assertion.
\end{proof}

With these composition and stability estimates in hand, all terms appearing in \eqref{skeleton_equation} are now well-defined within the two-parameter controlled framework. In particular, we can now give a precise meaning to two-parameter Rough Differential Equations.

\begin{definition}[Solution to a two-parameter rough differential equation]
\label{solution_2d_RDE}
Let $\mathbf{X},\mathbf{Y}$ be as above. We consider equations of the form
\[
Z_{t,v}
=
z
+
\int_{\mathcal D_{t,v}}
\Gamma^{g}(Z)_{s,u}\,d(\mathbf X_s,\mathbf Y_u),
\qquad (t,v)\in [a,b]\times[c,d],
\]
where $Z_{a,c}\in F$, and where $(g,\mathcal D_{t,v},\Gamma)$ is required to be of one of the following two types.

\smallskip
\noindent{Rectangular type.}
\[
g:F\to \mathrm{Hom}(V\boxtimes W,F),
\qquad
\mathcal D_{t,v}:=[a,t]\times[c,v], \qquad 
\Gamma^{g}(Z)_{s,u}:=g(Z)_{s,u}.
\]

\smallskip
\noindent{Triangular type.}
\[
g:F\boxtimes F\to \mathrm{Hom}(V\boxtimes W,F),
\qquad
\mathcal D_{t,v}:=\Delta_{[t,v]},
\qquad
\Gamma^{g}(Z)_{s,u}:=g(Z_{t, \cdot},Z)_{s,u}.
\]

We say that the equation is solved by a controlled path
\[
Z\in \mathcal D^{p,q}_{\mathbf X,\mathbf Y}\big([a,b]\times[c,d],F\big)
\]
if
\[
Z= z +I_{\mathbf X,\mathbf Y}\big(\Gamma^{g}(Z)\big).
\]
\end{definition}

\begin{remark}
    From our construction of the rough integral on triangular domain, the choice of $\mathcal{D}_{t,v} = \Delta_{[t,v]}$  implicitly forces the identity $\mathbf{Y} = \mathbf{X}$ in the above definition.
\end{remark}

We devote the remainder of the paper to the analysis of two examples of two-parameter RDEs, namely the signature kernel and the Schwinger–Dyson kernel equation, which illustrate the scope of the theory developed above.
\subsection{The signature kernel}
The first of the two applications we consider concerns the signature kernel, a kernel method for sequential data that uses the signature as a feature map. This approach was introduced in \cite{salvi2021signature} and has since been applied in a variety of settings; see, for example, \cite{lemercier2021distribution, toth2020bayesian, alden2025signature}.
In what follows, we briefly review the definition of the signature kernel and show how it can be related to the solution of a two-parameter rough differential equation.
We begin by recalling the notions of the signature and the associated signature kernel.
\newline 

Let $V=W$ be a finite-dimensional Hilbert space, and let $\mathbf{X}$ and $ \mathbf{Y} $ to be defined as in the previous sections.
Denote with $\Gamma$ the Gamma function and define
\[
\beta_p
:=
p^2\!\left(1+\sum_{i=3}^\infty
\left(\frac{2}{i-2}\right)^{\frac{\floor{p}+1}{p}}\right).
\]
Without loss of generality, from this point onwards, by replacing controls $\omega_{\mathbf{X}}$ and $ \omega_{\mathbf{Y}}$ if necessary, we may assume that 
\begin{align*}
&\left|\mathbf{X}^{(n)}_{s,t}\right| \leq \frac{\omega_{\mathbf{X}}(s,t)^{n/p}}{\beta_p \Gamma(n/p + 1)} \text{ for all } a \le s \le t \le b\text{ and } n \leq \kappa, \\
&\left|\mathbf{Y}^{(n)}_{u,v}\right| \leq \frac{\omega_{\mathbf{Y}}(u,v)^{n/q}}{\beta_q \Gamma(n/q + 1)} \text{for all } c \le u \le v \le d \text{ and } n \le \lambda.\\
\end{align*}

\begin{definition}
Let $\mathbf{X}$ be a geometric $p$-rough path on $V$, with $p \notin \mathbb{N}$.  
The signature of $\mathbf{X}$ is the unique extension of $\mathbf{X}$ to a path
\[
S(\mathbf{X}) : \Delta_{[a,b]} \to T((V)),
\]
such that $S^{(n)}(\mathbf{X}) = \mathbf{X}^{(n)}$ for every non-negative integer $n \leq \kappa$ and the following properties hold:
\begin{itemize}
\item
for every integer $n > \kappa$  the $n$-th level of the signature satisfies the estimate 
\[
\left|S^{(n)}(\mathbf{X})_{s,t}\right |
\le
\frac{\omega_{\mathbf{X}}(s,t)^{n/p}}{\beta_p\,\Gamma(n/p+1)} \quad \text{for all } a \le s \le t  \le b.
\]
\item
$S(\mathbf{X})$ is multiplicative, that is,
\[
S(\mathbf{X})_{s,s} = \mathbf{1}, \quad 
S(\mathbf{X})_{s,u} = S(\mathbf{X})_{s,t}\otimes S(\mathbf{X})_{t,u}
\quad \text{for all } a \le s \le t \le u \le b.
\]
\end{itemize}
\end{definition}

For a proof of this classical result in rough path theory we refer the reader to Theorem 3.7 in \cite{lyons2007differential}.\newline

It can be seen from the previous definition that the signatures \( S(\mathbf{X}) \) and \( S(\mathbf{Y}) \) take values in the Hilbert space
\[
F := \Big\{ \eta \in T((V)) : \|\eta\| < \infty \Big\},
\]
where
\[
\|\eta\| := \left( \sum_{k=0}^\infty \left\|\eta^{(k)}\right\|^2 \right)^{1/2},
\]
where each tensor power \(V^{\otimes k}\) is endowed with the Hilbert tensor norm induced by the fixed inner product on \(V\).

Within this framework, the signature kernel is defined as the map
\[
K^{\mathbf{X}, \mathbf{Y}} : [a,b] \times [c,d] \to \mathbb{R}, 
\qquad 
K^{\mathbf{X}, \mathbf{Y}}_{t,v} := \big\langle S(\mathbf{X})_{a,t},\, S(\mathbf{Y})_{c,v} \big\rangle,
\]
where \(\langle \cdot, \cdot \rangle\) denotes the induced inner product on $F$.

An important result linking the signature kernel to joint rough integrals appears in Proposition~6.3 of \cite{cass2023fubini}, where the signature kernel is obtained as a  two-parameter rough integral of the jointly controlled path
\begin{align*}
    & \Psi : [a,b] \times [c,d] \to \mathrm{Hom}\!\left(T^{<\kappa}(V) \boxtimes T^{<\lambda}(V), \mathrm{Hom}(V \boxtimes V, \mathbb{R})\right), \\
    & \Psi_{t,v}[\eta,\xi]
      = 
        \langle S(\mathbf{X})_{a,t} \otimes \eta \otimes \cdot,\,
                S(\mathbf{Y})_{c,v} \otimes \xi \otimes \cdot \rangle,
        \qquad \eta \in T^{<\kappa}(V), \xi \in T^{<\lambda}(V) ,
\end{align*}
via the identity
\begin{equation} \label{sig_pei_function}
K^{\mathbf{X},\mathbf{Y}}_{t,v}
=
1 + \int_{[a,t]\times[c,v]} \Psi_{s,u}\, d(\mathbf{X}_s,\mathbf{Y}_u).
\end{equation}
The identity above allows to construct a two-parameter controlled rough path above $K^{\mathbf{X}, \mathbf{Y}}$. In what follows, we use the same notation for the signature kernel viewed as a controlled path and for its trace, and clarify the intended meaning whenever necessary. \newline
Equipped with the results discussed earlier in this section, one can recognize the map $\Psi$ as arising from the composition of a smooth function with this controlled path.
To see this, consider the map
\begin{equation}\label{def_g_sig}
g : \mathbb{R} \to \mathrm{Hom}(V \boxtimes V, \mathbb{R}),
\qquad
g(z)[\eta,\xi] = z \langle \eta,\xi\rangle.
\end{equation}
When composed with the signature kernel, this yields the controlled path 
\begin{equation}\label{g_composed_sig}
\begin{aligned}
&g(K^{\mathbf{X},\mathbf{Y}}) \in \mathcal{D}^{p,q}_{\mathbf{X}, \mathbf{Y}}([a,b]\times [c,d], \mathrm{Hom}(V\boxtimes V, \mathbb{R})),\\
&g(K^{\mathbf{X},\mathbf{Y}})[\eta,\xi](\tilde{\eta},\tilde{\xi})
=
\langle S(\mathbf{X}) \otimes \eta \otimes \tilde{\eta},\,
        S(\mathbf{Y}) \otimes \xi \otimes \tilde{\xi} \rangle,
\qquad
\tilde{\eta},\tilde{\xi} \in V.
\end{aligned}
\end{equation}

In particular, this shows that the integrand $\Psi$ can be expressed as a smooth functional of the signature kernel understood as the controlled path constructed above the integral \eqref{sig_pei_function}, and therefore fits naturally within the framework of two-parameter controlled differential equations.

his naturally leads to the question of the well-posedness of the two-parameter RDE 
\[
Z_{t,v} = 1 +\int_{[a,t] \times [c, v]}   g(Z)_{s,u} d\left(\mathbf{X}_s, \mathbf{Y}_u\right),
\]
which we analyse in the next lemma and theorem.

\begin{lemma}[Local comparison for the signature-kernel equation]
\label{lemma:rectangular_local_comparison}
Choose
\(
p<p'<\kappa+1,
\) and \(
q<q'<\lambda+1.
\)

Fix $M>0$ and let
\(
I=[t_0,t_1]\subseteq[a,b],
\) and \(
J=[v_0,v_1]\subseteq[c,d].
\)
Write
\[
\rho_X:=\omega_X(t_0,t_1),
\qquad
\rho_Y:=\omega_Y(v_0,v_1),
\]
and assume $\rho_X,\rho_Y\leq1$. Throughout this lemma, $d_{I\times J}$ denotes the
controlled-path distance introduced in Section~2,
with $p,q$ replaced by $p',q'$ in the remainder seminorms.

Suppose that $Z$ and $\widetilde Z$ are two solutions
of the signature-kernel equation (47) on $I\times J$,
obtained by restricting controlled solutions of the
original equation. Assume that their controlled norms,
together with the rough-path norms of the drivers,
are bounded by $M$.

Suppose further that the two solutions have identical
incoming controlled boundary data:
\[
Z_{s,v_0}=\widetilde Z_{s,v_0},
\qquad
Z_{t_0,u}=\widetilde Z_{t_0,u},
\]
for every $s\in I$ and $u\in J$, where both equalities
are understood at every controlled level.

Then there exists $C_M>0$, independent of $I$ and $J$,
such that
\[
d_{I\times J}(Z,\widetilde Z)
\leq
C_M\rho_X^{1/p'}\rho_Y^{1/q'}
d_{I\times J}(Z,\widetilde Z).
\]
In particular, if
\[
C_M\rho_X^{1/p'}\rho_Y^{1/q'}<1,
\]
the two solutions coincide as jointly controlled paths
on $I\times J$.
\end{lemma}

\begin{proof}
All estimates are taken with respect to the restricted
controls on $I$ and $J$. Constants denoted by $C_M$
may change from line to line but are independent of
the intervals.

For a jointly controlled path $H \in \mathcal D_{\mathbf X}(\Delta_{[a,b]}, \mathbb R)$, write
\[
\mathfrak N(H):=
\sum_{j<\kappa}\sum_{k<\lambda}
\left\|
\mathbf R^{H,(j,k)}
\right\|_{
\left(\frac{p'}{\kappa-j},
      \frac{q'}{\lambda-k}\right),
(\omega_X,\omega_Y)}.
\]
Set $D:=Z-\widetilde Z$. Since its incoming controlled
boundary data vanish, the mixed-remainder identity
(11) gives
\[
\begin{aligned}
R^{D,1}_{s,t}(u)
&=\mathbf R^D
  \begin{pmatrix}s,t\\v_0,u\end{pmatrix},\\
R^{D,2}_{u,v}(s)
&=\mathbf R^D
  \begin{pmatrix}t_0,s\\u,v\end{pmatrix},\\
D_{s,u}
&=\mathbf R^D
  \begin{pmatrix}t_0,s\\v_0,u\end{pmatrix}.
\end{aligned}
\]
These identities hold at every controlled level.
Consequently, the full controlled distance reduces to
$\mathfrak N(D)$.

We first establish the one-parameter estimate used
below. Let $\mathbf B$ be a geometric $r$-rough path controlled by $\omega$,
$m=\lfloor r\rfloor$, and $r<r'<m+1$. Suppose that $\omega(h_0,h_1) \le 1$ and 
$W$ is controlled by $\mathbf B$ on
$H=[h_0,h_1]$, whose initial controlled coefficients
all vanish. Set
\[
n:=\sum_{j<m}
\|R^{W,(j)}\|_{\frac{r'}{m-j},\omega_B}.
\]
The controlled expansion and the vanishing initial
coefficients imply
\[
\sup_{h\in H}\|W_h^{(j)}\|
\leq C_M n\omega(h_0,h_1)^{(m-j)/r'},
\qquad 0\leq j<m.
\]
Thus, by the sewing lemma (Section 2 in \cite{cass2022combinatorial}) ,
\[
\left\|
\int_s^t W_u\,d\mathbf B_u-\Omega^W_{s,t}
\right\|
\leq C_M n\omega_B(s,t)^\theta.
\]
where
\[
\theta=\frac{m}{r'}+\frac1r
>\frac{m+1}{r'}>1.
\]

Let $I_{\mathbf B}(W)$ denote the canonical controlled
lift, with
\[
I_{\mathbf B}(W)^{(0)}_t
=\int_{h_0}^t W_u\,d\mathbf B_u,
\qquad
I_{\mathbf B}(W)^{(j)}=W^{(j-1)},
\quad 1\leq j<m.
\]
The preceding
estimates give
\[
\begin{aligned}
\|R^{I_{\mathbf B}(W),(0)}_{s,t}\|
&\leq C_M n
\left(
\omega(h_0,h_1)^{1/r'}\omega(s,t)^{m/r}
+\omega(s,t)^{\theta}
\right),\\
\|R^{I_{\mathbf B}(W),(j)}_{s,t}\|
&\leq C_M n
\left(
\omega(s,t)^{(m-j+1)/r'}
+\omega(h_0,h_1)^{1/r'}\omega(s,t)^{(m-j)/r}
\right),
\quad 1\leq j<m.
\end{aligned}
\]
Since $\omega(s,t)\leq\omega(h_0,h_1)\leq1$ and $r<r'$,
division by the corresponding remainder weights yields
\begin{equation}
\label{eq:local_one_parameter_lift}
\sum_{j<m}
\left\|R^{I_{\mathbf B}(W),(j)}
\right\|_{\frac{r'}{m-j},\omega_B}
\leq C_M\omega(h_0,h_1)^{1/r'}n.
\end{equation}
The integral lift also has vanishing initial
controlled coefficients.

We now apply \eqref{eq:local_one_parameter_lift}
successively in the two parameters. Let $H$ be a
jointly controlled integrand with vanishing incoming
controlled data, and set $G:=I_{\mathbf Y}(H)$.
Since the transport operators in the two tensor
factors commute,
\[
R^{G,1}_{s,t}(\cdot)
=
I_{\mathbf Y}\bigl(R^{H,1}_{s,t}(\cdot)\bigr).
\]
For each $j<\kappa$, the path
$u\to R^{H,1;(j,\cdot)}_{s,t}(u)$
is $\mathbf Y$-controlled, vanishes at $v_0$,
and has controlled remainder
$\mathbf R^{H,(j,\cdot)}$.
Applying \eqref{eq:local_one_parameter_lift}
at each controlled level and using (11), we obtain
\[
\mathfrak N(G)
\leq C_M\rho_Y^{1/q'}\mathfrak N(H).
\]
The incoming controlled data of $G$ still vanish,
so the same argument in the first parameter gives
\begin{equation}
\label{eq:local_rectangular_integral_estimate}
\mathfrak N\bigl(I_{\mathbf X}(I_{\mathbf Y}(H))\bigr)
\leq
C_M\rho_X^{1/p'}\rho_Y^{1/q'}\mathfrak N(H).
\end{equation}
By rough Fubini and the controlled-lift identities
(14), the iterated lift coincides with
$I_{\mathbf X,\mathbf Y}(H)$.

Finally, linearity of the signature-kernel
nonlinearity gives
\[
g(Z)-g(\widetilde Z)=g(D),
\qquad
\mathfrak N(g(D))\leq\|g\|\mathfrak N(D).
\]
Subtracting the two fixed-point equations and using
the common incoming controlled boundary data yields
\(
D=I_{\mathbf X,\mathbf Y}(g(D))
\)
as an identity of jointly controlled paths, with
the higher-level identities following from (14).
Hence, by \eqref{eq:local_rectangular_integral_estimate},
\[
\mathfrak N(D)
\leq
C_M\rho_X^{1/p'}\rho_Y^{1/q'}\mathfrak N(D),
\]
where $\|g\|$ has been absorbed into $C_M$.
If $C_M\rho_X^{1/p'}\rho_Y^{1/q'}<1$, then
$\mathfrak N(D)=0$, and the common incoming controlled data
imply $Z=\widetilde Z$ at every controlled level
on $I\times J$.
\end{proof}

\begin{theorem}\label{uniqueness_sig_kernel}
Let \(g\) be defined by \eqref{def_g_sig}. If \(Z\) is a solution to
\begin{equation}\label{signature_kernel_RDE}
Z_{t,v}
=
1+\int_{[a,t]\times[c,v]}
g(Z)_{s,u}\,d(\mathbf X_s,\mathbf Y_u)
\end{equation}
in the sense of Definition~\ref{solution_2d_RDE}, then
\(
Z=K^{\mathbf X,\mathbf Y}.
\)
In particular, \(K^{\mathbf X,\mathbf Y}\) is the unique solution
to \eqref{signature_kernel_RDE}.

Moreover, for every \(M>0\), there exists a constant
\(C_{p,q,M,\omega_{\mathbf X},\omega_{\mathbf Y}}>0\)
such that, whenever
\[
\widetilde{\mathbf X}
\in\Omega G_{p,\omega_{\mathbf X}}([a,b],V),
\qquad
\widetilde{\mathbf Y}
\in\Omega G_{q,\omega_{\mathbf Y}}([c,d],V),
\]
and
\[
\begin{aligned}
&\|\mathbf X\|_{p;[a,b]}
+\|\widetilde{\mathbf X}\|_{p;[a,b]} 
+\|\mathbf Y\|_{q;[c,d]}
+\|\widetilde{\mathbf Y}\|_{q;[c,d]}
\leq M,
\end{aligned}
\]
the corresponding signature kernels satisfy
\[
\bigl\|
K^{\mathbf X,\mathbf Y}
-
K^{\widetilde{\mathbf X},\widetilde{\mathbf Y}}
\bigr\|_{\infty}
\leq
C_{p,q,M,\omega_{\mathbf X},\omega_{\mathbf Y}}
\left(
\|\mathbf X-\widetilde{\mathbf X}\|_{p}
+
\|\mathbf Y-\widetilde{\mathbf Y}\|_{q}
\right),
\]
where the rough-path distances are computed using the
common controls \(\omega_{\mathbf X}\) and \(\omega_{\mathbf Y}\),
respectively.
\end{theorem}

\begin{proof}
The construction preceding the theorem shows that
$K^{\mathbf X,\mathbf Y}$ is a controlled solution
of \eqref{signature_kernel_RDE}. We prove uniqueness.

Let $Z$ and $\widetilde Z$ be two solutions driven
by the same pair $(\mathbf X,\mathbf Y)$.
Choose
\[
p<p'<\kappa+1,
\qquad
q<q'<\lambda+1,
\]
and let $M$ be a common bound for their controlled
norms and the rough-path norms.

First, we observe that the prescribed boundary data
coincide at every controlled level. Indeed, the
integral equation gives
\[
Z_{a,v}=\widetilde Z_{a,v}=1,
\qquad
Z_{t,c}=\widetilde Z_{t,c}=1.
\]
The corresponding higher-order boundary components
are uniquely determined by the controlled-lift
identities \eqref{definition_I_X_Y}. In particular,
the components with strictly positive degree in
both parameters satisfy
\[
Z^{(j,k)}=g(Z)^{(j-1,k-1)},
\qquad j,k\geq1,
\]
while the remaining components are given by the
corresponding one-parameter integral identities.
Since $g$ is linear, these identities determine
the boundary components recursively.

By continuity of the controls, we may choose finite
partitions
\[
a=t_0<\cdots<t_N=b,
\qquad
c=v_0<\cdots<v_L=d,
\]
such that, writing
$I_i=[t_i,t_{i+1}]$ and
$J_j=[v_j,v_{j+1}]$, we have
\begin{equation}
\label{eq:signature_local_partition}
C_M
\omega_{\mathbf X}(t_i,t_{i+1})^{1/p'}
\omega_{\mathbf Y}(v_j,v_{j+1})^{1/q'}
<1
\end{equation}
for every $i,j$.

By additivity of the rectangular rough integral,
the restrictions of both solutions to $I_i\times J_j$
satisfy the corresponding local integral equation,
with incoming boundary data inherited from the
global solutions.

On the first rectangle these data coincide.
Lemma~\ref{lemma:rectangular_local_comparison}
and \eqref{eq:signature_local_partition} therefore
imply that $Z=\widetilde Z$ on $I_0\times J_0$.

Proceeding by induction over $i+j$, the incoming
controlled boundary data of each subsequent rectangle
coincide, either by the original boundary conditions
or by equality on the adjacent rectangles.
Applying the same local comparison estimate gives
\[
Z=\widetilde Z
\qquad\text{on }[a,b]\times[c,d].
\]
Since $K^{\mathbf X,\mathbf Y}$ is already known to
be a solution, uniqueness follows.

It remains to prove stability. By the quantitative
extension theorem for multiplicative functionals,
\cite{lyons2007differential} Theorem 3.10, the factorial
decay estimates for the signature also hold for
differences of signatures. In particular, on bounded
sets of rough paths controlled by a common control,
\[
\|S(\mathbf X)^{(n)}_{s,t}
-S(\widetilde{\mathbf{X}})^{(n)}_{s,t}\|
\leq
C_{p,q,M,\omega_{\mathbf X},\omega_{\mathbf Y}}\|\mathbf X-\widetilde{\mathbf X}\|_p
\frac{\omega_{\mathbf X}(s,t)^{n/p}}
{\Gamma(n/p+1)},
\qquad n\geq1,
\]
with an analogous estimate for $\mathbf Y$.

Summing over the tensor levels and using factorial
decay, we obtain
\[
\begin{aligned}
\sup_{t\in[a,b]}
\|S(\mathbf X)_{a,t}\|
&\leq C_M,\\
\sup_{t\in[a,b]}
\|S(\mathbf X)_{a,t}
-S(\widetilde{\mathbf X})_{a,t}\|
&\leq
C_{p,q,M,\omega_{\mathbf X},\omega_{\mathbf Y}}\|\mathbf X-\widetilde{\mathbf X}\|_p,
\end{aligned}
\]
and similarly for $\mathbf Y$.

Recalling the signature-kernel representation and
applying the Cauchy--Schwarz inequality therefore gives
\[
\begin{aligned}
&\|K^{\mathbf X,\mathbf Y}
-K^{\widetilde{\mathbf X},\widetilde{\mathbf Y}}\|_\infty \leq
C_{p,q,M,\omega_{\mathbf X},\omega_{\mathbf Y}}\left(
\|\mathbf X-\widetilde{\mathbf X}\|_p
+
\|\mathbf Y-\widetilde{\mathbf Y}\|_q
\right),
\end{aligned}
\]
which completes the proof.
\end{proof}

\subsubsection{The smooth rough path case}
Under the additional assumption that $\mathbf{X}$ and $\mathbf{Y}$ are smooth rough paths, we show that the signature kernel admits a natural interpretation as the solution of a two-parameter RDE. In this way, the associated linear PDE, first established in \cite{lemercier2024log}, emerges directly from our framework. The argument hinges on the fact that the kernel is characterized as a fixed point of the two-parameter RDE, which is precisely the structure from which the PDE arises.

In the smooth rough path setting considered here, Theorem ~\ref{uniqueness_sig_kernel} also yields a transparent uniqueness argument for the kernel and provides a natural starting point for approximation procedures, such as replacing the drivers by smooth (or piecewise-linear) approximations.

More precisely, let $\mathbf{X}$ and $\mathbf{Y}$ be smooth rough paths and define the diagonal derivatives 
\[
\boldsymbol{x}_t := \left.\partial_t\right|_{t=s} \mathbf{X}_{s,t},
\qquad
\boldsymbol{y}_v := \left.\partial_v |_{v=u} \mathbf{Y}_{u,v}\right.
\]
From the previous section, the signature kernel admits the representation
\[
K^{\mathbf{X}, \mathbf{Y}}_{t,v}
= 1 + \int_a^t \int_c^v g\!\left(K^{\mathbf{X}, \mathbf{Y}}\right)_{s,u}
\big[ \boldsymbol{x}_s, \boldsymbol{y}_u \big]\, du\, ds,
\]
or, equivalently,
\[
\frac{\partial^2 K^{\mathbf{X}, \mathbf{Y}}_{t,v}}{\partial t\, \partial v}
= g\!\left(K^{\mathbf{X}, \mathbf{Y}}\right)_{t,v}
\big[ \boldsymbol{x}_t, \boldsymbol{y}_v \big].
\]

Fix a basis $\{e_1,\dots,e_m\}$ of $V$, let $I \neq \emptyset$ be a word in the alphabet $\{1, \dots, m\}$, and define the maps
\[
\begin{aligned}
f \colon\; & [a,b] \times [c,d] \to \mathbb{R}, 
& f(t,v) &:= (\pi_0 \boxtimes \pi_0)\!\left( K^{\mathbf{X},\mathbf{Y}}_{t,v} \right), \\[0.5em]
\phi \colon\; & [a,b] \times [c,d] \to \mathrm{Hom}\!\left(T^{<\lambda}(V), \mathbb{R}\right), 
& \phi(t,v)[e_I] &:= 
\big\langle L^*_{S(\mathbf{Y})_{c,v}} S(\mathbf{X})_{a,t}
- f(t,v)\,\mathbf{1},\, e_I \big\rangle, \\[0.5em]
\psi \colon\; & [a,b] \times [c,d] \to \mathrm{Hom}\!\left(T^{<\kappa}(V), \mathbb{R}\right), 
& \psi(t,v)[e_I] &:= 
\big\langle L^*_{S(\mathbf{X})_{a,t}} S(\mathbf{Y})_{c,v}
- f(t,v)\,\mathbf{1},\, e_I \big\rangle .
\end{aligned}
\]
Here $L^*, R^*$ denote respectively the adjoint of the left and right tensor multiplication operator on $F$. The maps $\phi$ and $\psi$ are then extended to $T^{<\lambda}(V)$ and $T^{<\kappa}(V)$ by linearity.  \newline

Within this setting, we are now in a position to provide an alternative proof of the PDE system satisfied by the signature kernel, first derived in \cite{lemercier2024log}.
In the proof, rather than expanding the two signatures into their homogeneous components, and then extracting the terms which give the PDE, we take the rough integral equation as the primary object. The explicit characterization of each level of this equation then leads directly to the PDE.

\begin{lemma}
Let $\mathbf{X}$ and $\mathbf{Y}$ be as above. Then the signature kernel satisfies the linear PDE

\begin{equation*}
\frac{\partial^2 f}{\partial t \, \partial v}
= f \langle \boldsymbol{x}_t, \boldsymbol{y}_v \rangle
+  \phi[R^*_{\boldsymbol{x}_t}  \boldsymbol{y}_v] 
+  \psi[R^*_{\boldsymbol{y}_v} \boldsymbol{x}_t],
\end{equation*}

\begin{equation*}
\frac{\partial \phi}{\partial t}
= f \boldsymbol{x}_t
+ \phi \circ R^*_{\boldsymbol{x}_t}
+ L^*_{\psi}\boldsymbol{x}_t
- \psi[\boldsymbol{x}_t] \mathbf{1} ,\end{equation*}

\begin{equation*}
\frac{\partial \psi}{\partial v}
= f \boldsymbol{y}_v
+ \psi \circ R^*_{\boldsymbol{y}_v}
+ L^*_{\phi}\boldsymbol{y}_v
- \phi[\boldsymbol{y}_v] \mathbf{1} .
\end{equation*}

\medskip

\noindent
with boundary conditions
\[
\begin{aligned}
f(a,v) &= 1, 
& \phi(a,v) &= \mathbf{0}, 
& \psi(a,v) &= \langle \mathbf{Y}_{c,v} - \mathbf{1}, \cdot \rangle, \\
f(t,c) &= 1, 
& \phi(t,c) &= \langle \mathbf{X}_{a,t} - \mathbf{1}, \cdot\rangle,
& \psi(t,c) &= \mathbf{0},
\end{aligned}
\]
where $\mathbf{0}$ denotes the corresponding zero linear map.
\end{lemma}
\begin{proof}
    We begin by determining the equation for $f$. From the result above we can write $f$ as 
    \begin{align*}
    f(t,v) &= 1 + \int_{[a,t] \times [c,v]} g(K^{\mathbf{X}, \mathbf{Y}})_{s,u} d(\mathbf{X}_s, \mathbf{Y}_u).
    \end{align*}

    From \eqref{g_composed_sig}, together with the canonical identification
\[
\mathrm{Hom}\!\left(T^{< \kappa}(V)\boxtimes T^{< \lambda}(V),\mathrm{Hom}\!\left(V\boxtimes V,\mathbb{R}\right)\right)
\cong
\mathrm{Hom}\!\left(T^{\le \kappa}(V)\boxtimes T^{\le \lambda}(V),\mathbb{R}\right),
\]
it follows that, for \(\eta \in V^{\otimes j}\) and \(\xi \in V^{\otimes k}\), with
\(j \in [1:\kappa]\) and \(k \in [1:\lambda]\),
\begin{align*}
g\!\left(K^{\mathbf X,\mathbf Y}\right)_{s,u}[\eta,\xi]
&=
\begin{cases}
\langle S(\mathbf X)_{a,s}, S(\mathbf Y)_{c,u}\rangle \, \langle \eta,\xi\rangle,
& \text{if } j=k, \\[0.5em]
\displaystyle
\left\langle S(\mathbf X)_{a,s}\otimes \eta,\, S(\mathbf Y)_{c,u} \otimes \xi \right\rangle
& \text{if } j>k, \\[0.8em]
\displaystyle
\left\langle S(\mathbf X)_{a,s}\otimes \eta,\, S(\mathbf Y)_{c,u} \otimes \xi \right\rangle
& \text{if } j<k.
\end{cases}\\
&= \begin{cases}
\langle S(\mathbf X)_{a,s}, S(\mathbf Y)_{c,u}\rangle \, \langle \eta,\xi\rangle,
& \text{if } j=k, \\[0.5em]
\displaystyle
\left\langle L^*_{S(\mathbf X)_{a,s}} S(\mathbf Y)_{c,u},\,  R^*_\xi \eta \right\rangle
& \text{if } j>k, \\[0.8em]
\displaystyle
\left\langle L^*_{S(\mathbf Y)_{c,u} } S(\mathbf X)_{a,s},\,   R^*_{\eta}\xi \right\rangle
& \text{if } j<k.
\end{cases}
\end{align*}
Notice that for every non-empty word of appropriate maximum length $I, J$, with $|I| > |J|$ we have $
\left\langle S(\mathbf X)_{a,s}\otimes e_I,\, S(\mathbf Y)_{c,u} \otimes e_J\right\rangle = \psi[R^*_{e_J} e_I]$   and $\phi[R^*_{e_I} e_J] = 0$, similarly, if $|I| < |J|$,
\[
\left\langle
S(\mathbf X)_{a,s} \otimes e_I,\,
S(\mathbf Y)_{c,u} \otimes e_J
\right\rangle
=
\phi[R^*_{e_I} e_J],
\qquad
\psi[R^*_{e_J} e_I] = 0.
\] Moreover, if $|I|=|J|$, $\psi[R^*_{e_J} e_I] = \phi[R^*_{e_I} e_J] = 0$.\newline  
This, in concert with the formula for $g\left(K^{\mathbf X,\mathbf Y}\right)$ outlined above, implies that in order to characterize $f$  we can restrict our attention to the maps
\(f, \psi\) and \(\phi\). In fact, we have
\begin{align*}
    f(t,v) &= 1 + \int_{[a,t] \times [c,v]} g(K^{\mathbf{X}, \mathbf{Y}})_{s,u} d(\mathbf{X}_s, \mathbf{Y}_u)\\
    &= 1 + \int_a^t \int_c^v g(K^{\mathbf{X}, \mathbf{Y}})_{s,u} \left[ \boldsymbol{x}_s, \boldsymbol{y}_u\right] ds\, du\\
    f(t,v)
&=
1
+
\int_a^t \int_c^v
\Big(
f_{s,u}\langle \boldsymbol{x}_s,\boldsymbol{y}_u\rangle
+
\phi_{s,u}[R^*_{\boldsymbol{x}_s}\boldsymbol{y}_u]
+
\psi_{s,u}[R^*_{\boldsymbol{y}_u}\boldsymbol{x}_s]
\Big)
\,du\,ds.
\end{align*}
This concludes the proof for $f$.\newline
To derive the equations for $\phi$ and $\psi$, it is sufficient to study the expressions $\langle S(\mathbf X)_{a,s} , S(\mathbf Y)_{c,u} \otimes \xi\rangle$ and $\langle S(\mathbf X)_{a,s} \otimes \eta, S(\mathbf Y)_{c,u}\rangle$ respectively. The calculations for these parts are almost identical to the ones in \cite{lemercier2024log} so we omit them.
\end{proof}

\begin{remark}
As seen above, the equations for $f, \psi$ and $\phi$ are enough to characterize $g(K^{\mathbf{X}, \mathbf{Y}})$, $I_\mathbf{X}(g(K^{\mathbf{X}, \mathbf{Y}}))$ and $I_{\mathbf{Y}}(g(K^{\mathbf{X}, \mathbf{Y}}))$. But this implies from the same reasoning as Theorem \ref{uniqueness_sig_kernel} that they can be used to characterize all the levels in the controlled path $K^{\mathbf{X}, \mathbf{Y}}$.
\end{remark}

\begin{comment}
    and using the definition of the equation in the previous section we can write
\begin{align*}
    f(t,v) &= 1 + \int_a^t \int_c^v \sum_{i=1}^d f_{s,u} \langle \mathbf{x}^{(i)}_s, \mathbf{y}^{(i)}_u \rangle ds du + \\
    &\quad +\int_a^t \int_c^v \sum_{i=1}^d \sum_{j=i+1}^d \varphi_{s,u} [R_{\mathbf{x}^{(i)}_s} \mathbf{y}^{(i)}_u ] ds du + \int_a^t \int_c^v \sum_{j=1}^d \sum_{i=j+1}^d \psi_{s,u} [R_{\mathbf{y}^{(i)}_s} \mathbf{x}^{(i)}_u ] ds du .
\end{align*}
\end{comment}

\begin{comment}
The definition of signature allows to write for each non empty word $I = i_1...i_n$ 
\begin{align*}
    \varphi(t,v)[e_I] = \int_a^t (\varphi(s,v)[e_{i_1...i_{n-1}}]1_{n > 1} + f(s,v) 1_{n = 1} ) \mathbf{x}^{(i_n)}_s   ds, 
\end{align*}
and similarly
\begin{align*}
    \psi(t,v)[e_I] = \int_c^v (\psi(t,u)[e_{i_1...i_{n-1}}]1_{n > 1} + f(t,u) 1_{n = 1} ) \mathbf{y}^{(i_n)}_u   du. 
\end{align*}
\end{comment}

\subsection{The Schwinger-Dyson kernel}
A second example of a two-parameter RDE is given by an extension of the Schwinger--Dyson kernel equation introduced in \cite{cass2024free}. The latter arises from the study of kernels induced by path developments in the matrix Lie group of $N \times N$ anti-Hermitian matrices $\mathfrak{u}_N$. 

More precisely, consider the unitary group-valued evolution defined by
\begin{equation*}
dZ^N_t = Z^N_t \cdot M(dX_t), \qquad Z^N_0 = Id_N,
\end{equation*}
where $V \equiv \mathbb{R}^d$, $M \in \mathrm{Hom}(V, \mathfrak{u}_N)$, $X \colon [0,T] \to V$ is a path of bounded variation and $Id_N$ is the $N \times N$ identity matrix.

Define the functional $\mathcal{U}(X,M) := Z^N_T$ and let $\langle \cdot, \cdot  \rangle_{HS,N }$ denote the Hilbert-Schmidt inner product on $\mathbb{C}^{N \times N }$. One can show, under a suitable choice of probability measure $\xi_N$ on $\mathfrak{u}_N$, the existence of the limit
\begin{equation}\label{SD_kernel_definition}
K(X^1, X^2) = \lim_{N \to \infty} \frac{1}{N} \mathbb{E}\!\left[ \left\langle \mathcal{U}(X^1, M), \mathcal{U}(X^2, M) \right\rangle_{HS, N} \right],
\end{equation}
for any couple of paths with bounded variation $X^1$ and $X^2$. In the remainder of this work, we refer to $K$ as the Schwinger-Dyson kernel.

This kernel can also be characterized in two additional ways, the first one as being the solution to the quadratic functional equation
\begin{equation}\label{SD_kernel_equation_BV}
K(s,t) = 1 - \int_s^t \int_s^r K(s,u)\, K(u,r)\, \langle dX_u, dX_r \rangle 
\end{equation}
where the path $X$ is defined as the concatenation $X := X^1 * \overleftarrow{X^2}$.\newline
And the second one as being the contraction of the signature $S(X)$ against the moment sequence $\varphi(\cdot)$ of a collection of $d$-free semicircular random variables
\begin{equation}\label{expansion_K}
K(s,t) = \sum_{|L| = 0}^\infty i^{|L|} \varphi(L) S^{(L)}(X)_{s,t},
\end{equation}
where the sum is taken over all words $L$, and $i$ denotes the imaginary unit.\newline
Leveraging this representation, we can extend the definition of the Schwinger--Dyson kernel to the case where $X$ is replaced by a  geometric $p$-rough paths of arbitrarily low regularity via an approximation argument, presented in Appendix~\ref{appendix_existence_SD}. In this case equation \eqref{SD_kernel_equation_BV} reads
\begin{equation}\label{eq:uniqueness_rough_SD_1}
    K(s,t) = 1 -  \int_s^t \int_s^r f(K(s,u), K(u,r)) d\mathbf{X}_u\,d\mathbf{X}_r.
\end{equation}
where 
\begin{align*}    
f: \mathbb R \times \mathbb R &\to \mathrm{Hom}(V \boxtimes V, \mathbb R), \quad f(z_1, z_2)  [\eta, \xi] = z_1 z_2 \langle \eta, \xi \rangle_V .
\end{align*}
It is now natural to ask whether this equation \eqref{eq:uniqueness_rough_SD_1}
admits a unique solution, the positive answer to this question is given in the following lemma and theorem.

We first specify the controlled norms and the local estimates used in
the proof. If $D$ is a simplex or a rectangle contained in
$\Delta_{[0,T]}$, write for $Q \in \mathcal D^p_{\mathbf X}(\Delta_{[a,b]}, F)$
\begin{align*}
[R^{Q,1;(j,k)}]_{\alpha_j;D}
&:=\sup
\frac{\|R^{Q,1;(j,k)}_{w,s}(t)\|}{\omega(w,s)^{\alpha_j}},\\
[R^{Q,2;(j,k)}]_{\alpha_k;D}
&:=\sup
\frac{\|R^{Q,2;(j,k)}_{z,t}(s)\|}{\omega(z,t)^{\alpha_k}},\\
[\mathbf R^{Q;(j,k)}]_{\alpha_j,\alpha_k;D}
&:=\sup
\frac{\left\|\mathbf R^{Q;(j,k)}
\begin{pmatrix}w,s\\z,t\end{pmatrix}\right\|}
{\omega(w,s)^{\alpha_j}\omega(z,t)^{\alpha_k}},
\end{align*}
where $\alpha_j = (\kappa - j) / p$.\newline
Each supremum is taken over nondegenerate increments for which all
evaluation points belong to $D$. In the last line this means
$[w,s]\times[z,t]\subseteq D$; on a simplex, in particular,
$w\le s\le z\le t$.
For $\lambda,\mu>0$, define
\begin{align}
\|Q\|_{D;\lambda,\mu}
:={}&\sum_{j,k=0}^{\kappa-1}
\Bigl(
\lambda^j\mu^k\|Q^{(j,k)}\|_{\infty;D}
+\lambda^\kappa\mu^k[R^{Q,1;(j,k)}]_{\alpha_j;D}
\nonumber\\
&\hspace{27mm}
+\lambda^j\mu^\kappa[R^{Q,2;(j,k)}]_{\alpha_k;D}
+\lambda^\kappa\mu^\kappa
[\mathbf R^{Q;(j,k)}]_{\alpha_j,\alpha_k;D}
\Bigr).
\label{eq:sdproof_scaled_norm}
\end{align}
We abbreviate $\|Q\|_D=\|Q\|_{D;1,1}$. This is a complete norm:
a Cauchy sequence converges uniformly in every controlled coefficient,
and its remainders converge in their weighted supremum norms; passing
to the limit preserves the controlled identities and the mixed-remainder
identity. The same argument applies on a simplex, since only rectangles
contained in that simplex enter the mixed seminorm.

\begin{lemma}[Bilinear estimates for the Schwinger--Dyson
integral operators]
\label{lemma:SD_bilinear_estimates}
Let $\mathbf X \in \Omega G_{p, \omega}([0, T], V)$ and define
\[
f:\mathbb R^2\to
\operatorname{Hom}(V\boxtimes V,\mathbb R),
\qquad
f(z_1,z_2)[\eta,\xi]
:=z_1z_2\langle\eta,\xi\rangle_V.
\]

Let $J=[b,c]\subseteq[0,T]$, and suppose that
\[
\omega(0,b)\leq A^p,
\qquad
\omega(b,c)\leq\varepsilon^p,
\qquad
A\geq1,\quad 0<\varepsilon\leq1.
\]
Write
\[
D_b:=[0,b]\times[b,c].
\]

For controlled paths $P,Q$ on $\Delta_J$, denote
by $\mathcal B_{\mathbf X}(P,Q)$ the canonical
controlled lift of
\[
\mathcal B_{\mathbf X}(P,Q)(s,t)
=
\int_s^t\int_s^r
f(P_{s,\cdot},Q)_{u,r}
\,d\mathbf X_u\,d\mathbf X_r.
\]

For $P$ on $\Delta_{[0,b]}$, $H$ on $D_b$, and
$Q$ on $\Delta_J$, define $\mathcal H_b(P,H)$
and $\mathcal T_b(H,Q)$ as the canonical controlled
lifts of
\begin{align*}
\mathcal H_b(P,H)(s,t)
&:=
\int_b^t\int_s^b
f(P_{s,\cdot},H)_{u,r}
\,d\mathbf X_u\,d\mathbf X_r,\\
\mathcal T_b(H,Q)(s,t)
&:=
\int_b^t\int_b^r
f(H_{s,\cdot},Q)_{u,r}
\,d\mathbf X_u\,d\mathbf X_r.
\end{align*}

Then these operators are bilinear and bounded in
the corresponding controlled-path spaces.

More precisely, there exists $C=C_{p,V}>0$,
independent of $b,c,A$ and $\varepsilon$, such that
\begin{align}
\|\mathcal B_{\mathbf X}(P,Q)\|_{
\Delta_J;\varepsilon,\varepsilon}
&\leq
C\varepsilon^2
\|P\|_{\Delta_J;\varepsilon,\varepsilon}
\|Q\|_{\Delta_J;\varepsilon,\varepsilon},
\label{eq:SD_bilinear_local}
\\
\|\mathcal H_b(P,H)\|_{D_b;A,\varepsilon}
&\leq
CA\varepsilon
\|P\|_{\Delta_{[0,b]};A,A}
\|H\|_{D_b;A,\varepsilon},
\label{eq:SD_bilinear_history}
\\
\|\mathcal T_b(H,Q)\|_{D_b;A,\varepsilon}
&\leq
C\varepsilon^2
\|H\|_{D_b;A,\varepsilon}
\|Q\|_{\Delta_J;\varepsilon,\varepsilon}.
\label{eq:SD_bilinear_tail}
\end{align}
\end{lemma}

\begin{proof}
We first establish the estimates with unscaled controlled
norms, assuming that the relevant driving controls have
total mass at most one. We then recover the factors
$A$ and $\varepsilon$ by homogeneous dilation.
Throughout, $C$ may change from line to line but
depends only on $p$, $V$, and the chosen tensor norms.

For the local operator, we apply
Lemma~\ref{lemma_parameter_simplicial_integrals}.
Take the family appearing in that lemma to be
\[
H^r[\eta_j]_{u,v}
:=
f\bigl(P_{r,u}^{(j,0)}[\eta_j],Q_{u,v}\bigr),
\qquad r\leq u\leq v,
\]
for $\eta_j\in V^{\otimes j}$, $0\leq j<\kappa$,
and extend linearly to $T^{<\kappa}(V)$.

By bilinearity of \(f\) and the shuffle coproduct rule, the controlled coefficients of \(H^r[\eta_j]\) are obtained by distributing the degrees in the shared variable \(u\) between \(P\) and \(Q\), while the degrees in \(v\) act only on \(Q\). The corresponding controlled-product estimates therefore yield
\[
\|H^r[\eta_j]\|_{\mathcal D^p_{\mathbf X}}
\leq
C\|P\|_{\Delta_J}\|Q\|_{\Delta_J}\|\eta_j\|.
\]
Moreover, bilinearity gives
\[
\mathcal R^H_{r,r'}[\eta_j]_{u,v}
=
f\bigl(
R^{P,1;(j,0)}_{r,r'}(u)[\eta_j],
Q_{u,v}
\bigr).
\]
By the mixed-remainder identity for two-parameter
controlled paths,
$R^{P,1;(j,\cdot)}_{r,r'}$ is controlled in $u$,
with norm bounded by
\(
C\|P\|_{\Delta_J}
\omega(r,r')^{(\kappa-j)/p}.
\)
Applying the controlled-product estimate again yields
\[
\begin{aligned}
&
\|\mathcal R^H_{r,r'}[\eta_j]\|_{
\mathcal D^p_{\mathbf X}(\Delta_{[r',c]})}\leq
C\|P\|_{\Delta_J}\|Q\|_{\Delta_J}
\omega(r,r')^{(\kappa-j)/p}\|\eta_j\|.
\end{aligned}
\]
Thus the hypotheses of
Lemma~\ref{lemma_parameter_simplicial_integrals}
hold with
\(
M=C\|P\|_{\Delta_J}\|Q\|_{\Delta_J}.
\)
Since $H^s[1]_{u,v}=f(P_{s,u},Q_{u,v})$,
the conclusion of that lemma gives
\begin{equation}
\label{eq:SD_bilinear_unscaled_local}
\|\mathcal B_{\mathbf X}(P,Q)\|_{\Delta_J}
\leq
C\|P\|_{\Delta_J}\|Q\|_{\Delta_J}.
\end{equation}

For the history operator, the bilinear controlled-product
estimate gives
\[
\left\|
(s,u)\mapsto f(P_{s,u},H_{u,\cdot})
\right\|_{\mathcal D^p_{\mathbf X}}
\leq
C\|P\|_{\Delta_{[0,b]}}\|H\|_{D_b},
\]
where the controlled norm on the left is taken with
values in the controlled-path space over $J$.

The dependence on the moving lower endpoint is handled
by Lemma~\ref{proposition_L_U_controlled}. Indeed,
the estimates in its proof apply coefficientwise
in the remaining variable and provide the corresponding
bounds for the full controlled lift, including its
first-order and mixed remainders. Consequently,
\[
\left\|
(s,r)\mapsto
\int_s^b
f(P_{s,u},H_{u,r})\,d\mathbf X_u
\right\|_{D_b}
\leq
C\|P\|_{\Delta_{[0,b]}}\|H\|_{D_b}.
\]
Here the controlled coefficients include the canonical
boundary corrections from the $U$-construction.

Finally, applying the one-parameter rough integration
estimate in the remaining variable yields
\begin{equation}
\label{eq:SD_bilinear_unscaled_history}
\|\mathcal H_b(P,H)\|_{D_b}
\leq
C\|P\|_{\Delta_{[0,b]}}\|H\|_{D_b}.
\end{equation}

For the history operator, the controlled-product
rule shows that
\[
(s,u)\to
f(P_{s,u},H_{u,\cdot}),
\qquad 0\leq s\leq u\leq b,
\]
is jointly controlled, with values in the
controlled-path space over $J$.
The mixed-remainder estimates for $P$ and $H$
give
\[
\left\|
(s,u)\to f(P_{s,u},H_{u,\cdot})
\right\|_{\mathcal D^p_{\mathbf X}}
\leq
C\|P\|_{\Delta_{[0,b]}}\|H\|_{D_b}.
\]

Since the two integration intervals are separated,
no additional diagonal-crossing correction is needed.

For the tail operator, we use
Corollary~\ref{corollary_simplex_controlled_lift}.
For each $s\in[0,b]$, the controlled-product rule
gives
\[
\left\|
(u,v)\to f(H_{s,u},Q_{u,v})
\right\|_{\mathcal D^p_{\mathbf X}(\Delta_J)}
\leq
C\|H\|_{D_b}\|Q\|_{\Delta_J}.
\]
Furthermore, the first-variable controlled expansion
of $H$ implies that, at degree $j$, the corresponding
parameter remainder of this integrand is the
controlled product of $R^{H,1;(j,\cdot)}_{w,s}$
with $Q$. The mixed-remainder identity and the
controlled-product estimate therefore give
\[
\begin{aligned}
&
\left\|
f\bigl(
R^{H,1;(j,\cdot)}_{w,s},Q
\bigr)
\right\|_{\mathcal D^p_{\mathbf X}(\Delta_J)}\leq
C\|H\|_{D_b}\|Q\|_{\Delta_J}
\omega(w,s)^{(\kappa-j)/p}.
\end{aligned}
\]

By Corollary~\ref{corollary_simplex_controlled_lift},
simplicial rough integration is bounded and linear
in the full two-parameter controlled norm.
Applying it coefficientwise to the preceding
first-variable expansions preserves their
remainder estimates, including those in the
terminal variable. Hence the resulting lift
satisfies
\begin{equation}
\label{eq:SD_bilinear_unscaled_tail}
\|\mathcal T_b(H,Q)\|_{D_b}
\leq
C\|H\|_{D_b}\|Q\|_{\Delta_J}.
\end{equation}

It remains to recover the dependence on
$A$ and $\varepsilon$.
Recall that the homogeneous dilation satisfies
\[
(\delta_\lambda\mathbf X)^{(m)}
=
\lambda^m\mathbf X^{(m)}.
\]
Rescaling the driving rough paths in the two
controlled variables by $\lambda^{-1}$ and
$\mu^{-1}$ transforms the controlled coefficients
according to
\[
\widehat R^{(j,k)}
=
\lambda^j\mu^kR^{(j,k)}.
\]
The first-order and mixed remainders transform
with the same tensor-degree weights.
Consequently, the unscaled norms of the rescaled
fields coincide with the scaled norms defined in
\eqref{eq:sdproof_scaled_norm}.

For the local operator, both integrations take
place in $J$. Rescaling the restricted driver
by $\varepsilon^{-1}$ normalises its control,
and the two integrations contribute
$\varepsilon^2$. Thus
\eqref{eq:SD_bilinear_unscaled_local} gives
\[
\|\mathcal B_{\mathbf X}(P,Q)\|_{
\Delta_J;\varepsilon,\varepsilon}
\leq
C\varepsilon^2
\|P\|_{\Delta_J;\varepsilon,\varepsilon}
\|Q\|_{\Delta_J;\varepsilon,\varepsilon}.
\]

For the history operator, the inner integration
takes place in $[0,b]$ and the outer integration
in $J$. Rescaling the corresponding restricted
drivers by $A^{-1}$ and $\varepsilon^{-1}$
normalises their controls, while the integrations
contribute $A\varepsilon$. Hence
\eqref{eq:SD_bilinear_unscaled_history} yields
\[
\|\mathcal H_b(P,H)\|_{D_b;A,\varepsilon}
\leq
CA\varepsilon
\|P\|_{\Delta_{[0,b]};A,A}
\|H\|_{D_b;A,\varepsilon}.
\]

Finally, both integrations in the tail operator
take place in $J$. Rescaling the new-interval
driver by $\varepsilon^{-1}$ and the first
controlled variable by $A^{-1}$ gives, from
\eqref{eq:SD_bilinear_unscaled_tail},
\[
\|\mathcal T_b(H,Q)\|_{D_b;A,\varepsilon}
\leq
C\varepsilon^2
\|H\|_{D_b;A,\varepsilon}
\|Q\|_{\Delta_J;\varepsilon,\varepsilon}.
\]

The bounded integration results cited above and
the homogeneous dilation identities concern the
full canonical controlled lifts. Thus all three
estimates include the controlled coefficients,
first-order remainders and mixed remainders.
Bilinearity follows from that of $f$ and the
linearity of rough integration.
\end{proof}

\begin{theorem}[Uniqueness and stability of the rough Schwinger--Dyson kernel]
\label{theorem:uniqueness_SD}
Let $V$ be a finite-dimensional real Hilbert space,
let $\mathbf X\in\Omega G_p([0,T],V)$, and define
\[
f:\mathbb R^2\to
\operatorname{Hom}(V\boxtimes V,\mathbb R),
\qquad
f(z_1,z_2)[\eta,\xi]
:=z_1z_2\langle\eta,\xi\rangle_V.
\]
If $Z$ is a solution to
\begin{equation}
\label{eq:uniqueness_rough_SD}
Z(s,t)
=
1-\int_s^t\int_s^r
f(Z_{s,\cdot},Z)_{u,r}\,
d\mathbf X_u\,d\mathbf X_r
\end{equation}
in the sense of Definition~\ref{solution_2d_RDE},
then $Z=K$. In particular, $K$ is the unique solution
to \eqref{eq:uniqueness_rough_SD} in
$\mathcal D^p_{\mathbf X}(\Delta_{[0,T]},\mathbb R)$.

Moreover, for every $M<\infty$, there exists a
constant $C_{p,V,M}>0$ such that, whenever
$\widetilde{\mathbf X}\in\Omega G_p([0,T],V)$ and
\[
\|\mathbf X\|_{p\text{-var};[0,T]}
+
\|\widetilde{\mathbf X}\|_{p\text{-var};[0,T]}
\leq M,
\]
the corresponding kernels satisfy
\[
\|K-\widetilde K\|_{\infty;\Delta_{[0,T]}}
\leq
C_{p,V,M}\,
\|\mathbf X - \widetilde{\mathbf X}\|_p.
\]
\end{theorem}

\begin{proof}
Existence has already been established by the
signature construction in Appendix B. We first
prove uniqueness by comparing two existing
controlled solutions.

Let $Z$ be another solution of
\eqref{eq:uniqueness_rough_SD}, and set
\(
D:=K-Z.
\)

We use the scaled controlled norms introduced in
\eqref{eq:sdproof_scaled_norm}. Choose
\[
A:=\max\{1,\omega(0,T)^{1/p}\},
\]
and
\[
M:=
A^{2\kappa}
\left(
1+\|K\|_{\Delta_{[0,T]}}
+\|Z\|_{\Delta_{[0,T]}}
\right).
\]
Since $A\geq1$, this constant bounds the relevant
scaled norms of both solutions on all subdomains
considered below, uniformly for
$0<\varepsilon\leq1$.

Let $C$ denote the constant in
Lemma~\ref{lemma:SD_bilinear_estimates}.
Choose $\varepsilon>0$ sufficiently small that
\begin{equation}
\label{eq:SD_uniqueness_smallness}
CM(A\varepsilon+\varepsilon^2)<\frac12.
\end{equation}

We establish uniqueness first on sufficiently
small simplices, and then extend it to the full
domain.

Let $J=[b,c]\subseteq[0,T]$ satisfy
\(
\omega(b,c)\leq\varepsilon^p.
\)
Restricting the two solution identities to
$\Delta_J$ and using the bilinearity of $f$, we
obtain, as an identity of full controlled lifts and Definition~\ref{solution_2d_RDE},
\begin{equation}
\label{eq:SD_local_difference}
D
=
-\mathcal B_{\mathbf X}(D,K)
-\mathcal B_{\mathbf X}(Z,D).
\end{equation}

Applying \eqref{eq:SD_bilinear_local} to
\eqref{eq:SD_local_difference}, we find
\[
\begin{aligned}
\|D\|_{\Delta_J;\varepsilon,\varepsilon}
&\leq
C\varepsilon^2
\Bigl(
\|D\|_{\Delta_J;\varepsilon,\varepsilon}
\|K\|_{\Delta_J;\varepsilon,\varepsilon} +
\|Z\|_{\Delta_J;\varepsilon,\varepsilon}
\|D\|_{\Delta_J;\varepsilon,\varepsilon}
\Bigr)\\
&\leq
CM\varepsilon^2
\|D\|_{\Delta_J;\varepsilon,\varepsilon}.
\end{aligned}
\]
By \eqref{eq:SD_uniqueness_smallness}, the
coefficient on the right-hand side is strictly
smaller than one. Consequently,
\(
D=0\) on \(\Delta_J
\)
as a full controlled field. In particular,
all controlled coefficients and remainders agree.

Suppose that, for some $b\in(0,T)$,
\(
D=0
\) on \(\Delta_{[0,b]}.
\)
Choose $c>b$ such that
$\omega(b,c)\leq\varepsilon^p$.

The preceding local uniqueness argument gives
\(
D=0\) on \(\Delta_{[b,c]}.
\)
It remains to establish equality on the crossing
region
\(
D_b:=[0,b]\times[b,c].
\)

Fix $s\leq b\leq t$. The simplicial domain admits
the decomposition
\(
\Delta_{[s,t]}
=
\Delta_{[s,b]}
\cup
\bigl([s,b]\times[b,t]\bigr)
\cup
\Delta_{[b,t]},
\)
where the common boundaries are interpreted
according to the additivity of the simplicial
rough integral.

Subtracting the two equations and recalling that
$D=K-Z$, we obtain
\[
\begin{aligned}
D(s,t)
={}&-\int_{\Delta_{[s,t]}}
\Bigl(
f(D_{s,u},K_{u,r})
+
f(Z_{s,u},D_{u,r})
\Bigr)
\,d(\mathbf X_u,\mathbf X_r).
\end{aligned}
\]

The contribution over $\Delta_{[s,b]}$ vanishes
because the two solutions coincide on the old
simplex.

On the crossing rectangle,
$D_{s,u}=0$ whenever $s\leq u\leq b$.
Thus only the second factor contributes there,
and $Z_{s,u}$ may be replaced by $K_{s,u}$.

On the new simplex $\Delta_{[b,t]}$, we have
$D_{u,r}=0$. Hence only the first factor
contributes.

It follows that
\begin{equation}
\label{eq:SD_continuation_identity}
\begin{aligned}
D(s,t)
={}&-\int_b^t\int_s^b
f(K_{s,u},D_{u,r})\,
d\mathbf X_u\,d\mathbf X_r -\int_b^t\int_b^r
f(D_{s,u},K_{u,r})\,
d\mathbf X_u\,d\mathbf X_r.
\end{aligned}
\end{equation}

By compatibility of the rectangular and simplicial
controlled lifts, the preceding decomposition
also holds at every other level in the tensor algebra. The incoming
controlled boundary data on $t=b$ coincide, so
there is no additional boundary difference.

Using the operators introduced in
Lemma~\ref{lemma:SD_bilinear_estimates}, we may
therefore write
\begin{equation}
\label{eq:SD_continuation_lift}
D
=
-\mathcal H_b(K,D)
-\mathcal T_b(D,K)
\qquad\text{on }D_b.
\end{equation}

Applying \eqref{eq:SD_bilinear_history} and
\eqref{eq:SD_bilinear_tail}, we obtain
\[
\begin{aligned}
\|D\|_{D_b;A,\varepsilon}
&\leq
CA\varepsilon
\|K\|_{\Delta_{[0,b]};A,A}
\|D\|_{D_b;A,\varepsilon} +
C\varepsilon^2
\|D\|_{D_b;A,\varepsilon}
\|K\|_{\Delta_{[b,c]};\varepsilon,\varepsilon}\\
&\leq
CM(A\varepsilon+\varepsilon^2)
\|D\|_{D_b;A,\varepsilon}.
\end{aligned}
\]

By \eqref{eq:SD_uniqueness_smallness}, this implies
\(
D=0
\) on \(D_b.
\)
Together with equality on the old and new
simplices, we conclude that
\(
D=0\) on \(\Delta_{[0,c]}.
\)

Finally, continuity of the control permits a
finite partition
\(
0=t_0<t_1<\cdots<t_N=T
\)
such that
\(
\omega(t_i,t_{i+1})\leq\varepsilon^p
\) for every i.
Starting with local uniqueness on
$\Delta_{[t_0,t_1]}$ and applying the continuation
argument successively yields
\[
Z=K
\quad\text{on }\Delta_{[0,T]}.
\]
The constants remain uniform throughout the
iteration because $M$ was fixed using the global
controlled norms of the two solutions.

This completes the uniqueness proof.

It remains to prove stability. By the quantitative
extension theorem for multiplicative functionals
used in the proof of the signature-kernel stability
estimate, factorial decay also holds for differences
of signatures. In particular, on bounded sets of
rough paths controlled by a common control $\omega$,
\[
\left\|
S^{(n)}(\mathbf X)_{s,t}
-
S^{(n)}(\widetilde{\mathbf X})_{s,t}
\right\|
\leq
\frac{C_M^n\omega(s,t)^{n/p}}
{\Gamma(n/p+1)}
\|\mathbf X-\widetilde{\mathbf X}\|_p,
\qquad n\geq1.
\]

Recalling the representation \eqref{rough_SD_rep}
and using $|\varphi(I)|\leq 2^{|I|}$, we obtain
\[
\begin{aligned}
|K(s,t)-\widetilde K(s,t)|
&\leq
\sum_{n=1}^{\infty}
\sum_{|I|=n}
|\varphi(I)|
\left|
S^{(I)}(\mathbf X)_{s,t}
-
S^{(I)}(\widetilde{\mathbf X})_{s,t}
\right|
\\
&\leq
\|\mathbf X-\widetilde{\mathbf X}\|_p
\sum_{n=1}^{\infty}
\frac{(2dC_M)^n\omega(s,t)^{n/p}}
{\Gamma(n/p+1)}.
\end{aligned}
\]
The series is uniformly bounded on $\Delta_{[0,T]}$
by factorial decay. Consequently,
\[
\|K-\widetilde K\|_{\infty;\Delta_{[0,T]}}
\leq
C_M\|\mathbf X-\widetilde{\mathbf X}\|_p,
\]
where the constant has been enlarged if necessary.
This completes the proof.
\end{proof}

\subsubsection{The smooth rough path case}
\label{sec:smooth_RP_SD}

In this subsection, we derive an integro-differential formulation
of the Schwinger--Dyson kernel equation when the driving rough path
$\mathbf X$ is smooth. The construction relies on a combinatorial
relationship between the trace component of $K$ and the higher-level
components of its controlled lift.

Recall that, under the convention of Definition~4, the components
of the controlled path $K$ admit the representation
\begin{equation}\label{gub_representation}
K^{(I,J)}(s,t)
=
(-1)^{|I|}
\sum_{|L|=0}^{\infty}
i^{|I|+|J|+|L|}
\varphi\bigl(\operatorname{rev}(I)LJ\bigr)
S^{(L)}(\mathbf X)_{s,t},
\end{equation}
where $I=i_1\cdots i_n$ and $J=j_1\cdots j_m$ are possibly
empty words over an alphabet of size $d$, and
$\operatorname{rev}(I)=i_n\cdots i_1$ denotes word reversal.
This identity is established in the proof of
Theorem~\ref{theorem_existence}.

The reversal in the initial controlled argument arises from
the antipode of the tensor algebra, whereas increments of the
terminal endpoint act through right multiplication of the
signature.

We first establish the combinatorial identities below for
bounded-variation paths. Their extension to geometric rough
paths follows by smooth approximation, using the explicit
signature representation and convergence of the associated
controlled lifts established in Appendix~\ref{appendix_existence_SD}.

Before proceeding, we introduce the set $NC(n)$ of non-crossing
partial matchings of length $n$, defined recursively by
\begin{align*}
NC(1)
&=
\bigl\{\{\{1\}\}\bigr\},\\
NC(n)
&=
\left\{
\pi\cup\{\{n\}\}:\pi\in NC(n-1)
\right\}\\
&\quad\cup
\left\{
(\pi\setminus\{\{m\}\})\cup\{\{m,n\}\}:
\pi\in NC(n-1),\
U(\pi)\neq\varnothing,\
m=\max U(\pi)
\right\}.
\end{align*}
Here $U(\pi)$ denotes the set of unmatched positions, and
$P(\pi)$ denotes the set of pairs. Whenever
$U(\pi)=\{U_1<\cdots<U_r\}$, the unmatched positions are
listed in increasing order.

We also recall from
\cite[Lemma 5.4.7, Remark 5.4.8]{anderson2010introduction}
that the moment functional $\varphi$ satisfies
\begin{align}
\varphi(IJ)
&=
\varphi(JI),
\nonumber\\
\varphi(Ii)
&=
\sum_{I=KiJ}
\varphi(K)\varphi(J).
\label{SD_prop2}
\end{align}
These are the Schwinger--Dyson identities underlying the
kernel equation.

By cyclic invariance of $\varphi$, the representation
\eqref{gub_representation} gives
\[
K^{(I,\emptyset)}(s,t)
=
(-1)^{|I|}
K^{(\emptyset,\operatorname{rev}(I))}(s,t).
\]
Consequently, it is sufficient to establish the combinatorial
representation of the controlled components in the second
parameter. The corresponding identities in the initial
parameter follow by word reversal and the appropriate sign.

In the next result, we therefore regard $K$ as a
one-parameter controlled path in its terminal variable,
writing $K^{(I)}:=K^{(\emptyset,I)}$.

\begin{proposition}\label{Gubinelli derivatives}
Let $X$ have bounded variation. For every non-empty
controlled argument word $I=i_1\cdots i_n$, with
$n\leq\kappa-1$, set
\(
W=\operatorname{rev}(I)=w_1\cdots w_n.
\)
Then
\[
K^{(I)}
=
\sum_{\pi\in NC(n)}
\left(
\prod_{(p,q)\in P(\pi)}
(-1)\delta_{w_p,w_q}
\right)
K_{w_{U_1},\ldots,w_{U_r}},
\]
with the convention $K^{(\emptyset)}=K_\emptyset=K$.

For every non-empty chronological word
$J=j_1\cdots j_m$, the auxiliary process $K_J$ is
defined by
\begin{equation}\label{index_removal}
\begin{aligned}
K_J(s,t)
={}&
(-1)^m
\int_{s<u_1<\cdots<u_m<t}
K(s,u_1)K(u_1,u_2)\cdots K(u_m,t) 
dX^{(j_1)}_{u_1}\cdots dX^{(j_m)}_{u_m}.
\end{aligned}
\end{equation}
\end{proposition}

\begin{proof}
We first establish the recursion for the controlled
derivatives of the auxiliary processes in the second parameter.

For the empty word, the Schwinger--Dyson equation gives
\[
K_\emptyset^{(j)}=K_j.
\]

For a non-empty word $J=j_1\cdots j_m$, taking the
increment in the second variable in \eqref{index_removal}
yields
\begin{equation}\label{eq:SD_aux_terminal_derivative}
K_J^{(j)}
=
K_{Jj}
-
\delta_{j_m,j}K_{j_1\cdots j_{m-1}}.
\end{equation}

Indeed, the first term is obtained from the increment
of the last kernel factor $K(u_m,t)$, while the second
is the boundary contribution at $u_m=t$.

We now proceed by induction on the length of the
controlled argument word.

For $I=\emptyset$, the assertion holds by definition.
The case $|I|=1$ follows immediately from
$K_\emptyset^{(j)}=K_j$.

Assume that the formula holds for a word $I$ of length
$n$. By the left tensor multiplication convention of
Definition~4, differentiation in a new direction $j$
prepends the new letter to the controlled argument.
Thus,
\(
K^{(jI)}
=
\bigl(K^{(I)}\bigr)^{(j)}.
\)

Since
\(
\operatorname{rev}(jI)
=
\operatorname{rev}(I)j,
\)
the recursion \eqref{eq:SD_aux_terminal_derivative}
produces precisely the two operations defining
$NC(n+1)$: the new final position is either left
unmatched or paired with the last unmatched position.

The corresponding coefficients are respectively
$1$ and $-\delta_{w_{U_r},j}$.
Summing over all $\pi\in NC(n)$ therefore yields
\[
K^{(jI)}
=
\sum_{\pi\in NC(n+1)}
\left(
\prod_{(p,q)\in P(\pi)}
(-1)\delta_{\widetilde w_p,\widetilde w_q}
\right)
K_{\widetilde w_{U_1},\ldots,\widetilde w_{U_r}},
\]
where
\(
\widetilde W
=
\operatorname{rev}(jI).
\)
This completes the induction and the proof.
\end{proof}

\begin{remark}\label{remark_smooth_numerical}
The auxiliary processes $K_I$ admit controlled lifts
in both endpoint parameters. We record here the
combinatorial representation of their terminal
controlled derivatives.

Let $I=i_1\cdots i_n$ be a chronological auxiliary
word and let $J=j_1\cdots j_m$ be an admissible
controlled argument word. Set
\[
W=I\operatorname{rev}(J)
=w_1\cdots w_{n+m}.
\]

Then
\begin{equation}\label{eq:SD_aux_controlled}
K_I^{(J)}
=
\sum_{\pi\in NC(n+m)_n}
\left(
\prod_{(p,q)\in P(\pi)}
(-1)\delta_{w_p,w_q}
\right)
K_{w_{U_1},\ldots,w_{U_r}},
\end{equation}
where $NC(n+m)_n$ denotes the collection of
non-crossing partial matchings for which no pair
joins two positions belonging to the initial
chronological word $I$.

In particular, the chronological order of $I$ is
preserved, while the controlled argument word $J$
is reversed.

For example,
\[
K_i^{(j)}
=
K_{ij}-\delta_{ij}K.
\]

Similarly, taking increments in the initial endpoint
of \eqref{index_removal} gives
\begin{equation}\label{eq:SD_aux_initial_derivative}
K_I^{s,(j)}
=
-K_{jI}
+
\delta_{j,i_1}K_{i_2\cdots i_n},
\end{equation}
for every non-empty word $I=i_1\cdots i_n$, whereas
\[
K_\emptyset^{s,(j)}=-K_j.
\]
Here the superscript \(s\) indicates the controlled expansion with respect to the initial endpoint. These identities
will be used below to construct the corresponding
controlled lifts.
\end{remark}

Following the convention introduced for the auxiliary
processes, we denote by $K_n$ the level-wise analogue
of the coordinate-wise notation $K_I$. More precisely,
\[
K_n:
\Delta_{[0,T]}
\to
\operatorname{Hom}(V^{\otimes n},\mathbb R),
\qquad
K_n(s,t)[e_I]=K_I(s,t),
\quad |I|=n.
\]

To distinguish the jointly controlled rough path $K$
from the auxiliary hierarchy, we use the notation
$K_\bullet$ for the latter.

We emphasise that $K_n$ denotes the auxiliary processes
defined by the chronological integral
\eqref{index_removal}, and not the raw homogeneous
controlled coefficients of $K$.

The preceding identities lead to the following
integral representation.

\begin{proposition}\label{numerical_scheme_rep}
For every $n\in[1:\kappa+\zeta]$, let
$\ell=\lceil n/2\rceil$. Then the auxiliary processes
satisfy
\begin{equation}\label{alternative_representation_levelwise}
\begin{aligned}
K_0(s,t)
&=
1+\int_s^t K_1(s,u)\,d\mathbf X_u,\\
K_n(s,t)
&=
-\int_s^t
\sigma_n\bigl(
K_{\ell-1}(s,\cdot),
K_{n-\ell}(\cdot,t)
\bigr)_u
\,d\mathbf X_u.
\end{aligned}
\end{equation}

For every word $I=i_1\cdots i_n$, set
\[
\widetilde I=i_1\cdots i_{\ell-1},
\qquad
\widehat I=i_{\ell+1}\cdots i_n.
\]
Given a basis $e_1,\ldots,e_d$ of $V$, define
\[
\begin{aligned}
\sigma_n:\;&
\operatorname{Hom}
(V^{\otimes(\ell-1)},\mathbb R)
\boxtimes
\operatorname{Hom}
(V^{\otimes(n-\ell)},\mathbb R)
\to
\operatorname{Hom}
\bigl(V,\operatorname{Hom}(V^{\otimes n},\mathbb R)\bigr)
\end{aligned}
\]
by
\[
\sigma_n(\eta,\xi)[e_j](e_I)
=
\eta(e_{\widetilde I})
\delta_{i_\ell j}
\xi(e_{\widehat I}),
\]
and extend linearly on $V^{\otimes n}$.
\end{proposition}

\begin{proof}
The first identity follows from the
Schwinger--Dyson equation and the definition of
$K_1$.

For the second identity, first suppose that $X$
has bounded variation. Let $J=j_1\cdots j_n$
be non-empty. By \eqref{index_removal},
\[
\begin{aligned}
K_J(s,t)
={}&
(-1)^n
\int_{s<u_1<\cdots<u_n<t}
K(s,u_1)K(u_1,u_2)\cdots K(u_n,t)
dX^{(j_1)}_{u_1}\cdots dX^{(j_n)}_{u_n}.
\end{aligned}
\]

Setting $\ell=\lceil n/2\rceil$ and separating
the integration at $u_\ell=u$, we obtain
\[
K_J(s,t)
=
-\int_s^t
K_{j_1\cdots j_{\ell-1}}(s,u)
K_{j_{\ell+1}\cdots j_n}(u,t)
\,dX_u^{(j_\ell)}.
\]

This is precisely the coordinate-wise expression
of the second identity.

The general geometric rough-path case follows
by approximation, using convergence of the
controlled lifts and continuity of one-parameter
rough integration.
\end{proof}

We now construct the controlled enhancements needed
for the integro-differential formulation.

Fix a non-negative integer $\zeta$, determining
the truncation level $\kappa+\zeta$ of the auxiliary
hierarchy. To ensure closure of the system, require
\[
\max\left\{
\left\lfloor\frac{\kappa+\zeta}{2}\right\rfloor,
1
\right\}
+\kappa-1
\leq
\kappa+\zeta.
\]
Equivalently,
\[
\zeta\geq
\begin{cases}
\kappa-3, & \kappa\geq4,\\
0, & \text{otherwise}.
\end{cases}
\]

The condition ensures that all controlled coefficients
required by the integrands in
\eqref{alternative_representation_levelwise}
can be represented within the chosen auxiliary
hierarchy.

We now construct the endpoint-controlled enhancements
required by \eqref{alternative_representation_levelwise}.
To this end, we collect the terminal and initial
controlled coefficients of the auxiliary hierarchy
into two maps, denoted by $\psi$ and $\phi$,
respectively.

Writing $N=\kappa+\zeta$ and
$m=\lceil N/2\rceil$, we consider
\[
\begin{aligned}
\psi:\;&
\Delta_{[0,T]}
\to
\operatorname{Hom}\left(
T^{\leq m-1}(V),
\operatorname{Hom}(T^{< \kappa}(V),\mathbb R)
\right),
\\
\phi:\;&
\Delta_{[0,T]}
\to
\operatorname{Hom}\left(
T^{\leq N-m}(V),
\operatorname{Hom}(T^{< \kappa}(V),\mathbb R)
\right).
\end{aligned}
\]

For chronological auxiliary words $I$ and controlled
argument words $J$ of admissible lengths, their
components are prescribed by
\begin{equation}\label{eq:SD_psi_phi}
\begin{aligned}
\psi_{s,t}(e_I)[e_J]
&:=K_I^{(J)}(s,t),\\
\phi_{s,t}(e_I)[e_J]
&:=K_I^{s,(J)}(s,t).
\end{aligned}
\end{equation}
These coefficients are determined by
\eqref{eq:SD_aux_controlled} and
\eqref{eq:SD_aux_initial_derivative}, respectively.
In particular, both lifts recover the auxiliary
processes at level zero:
\[
\psi_{s,t}(e_I)[1]
=
\phi_{s,t}(e_I)[1]
=
K_I(s,t).
\]

For fixed $s<t$, write
\[
\psi_s(u):=\psi_{s,u},
\qquad
\phi_t(u):=\phi_{u,t},
\qquad u\in[s,t].
\]
The controlled lift of the integrand
\[
u\to
\sigma_n\bigl(
K_{\ell-1}(s,u),
K_{n-\ell}(u,t)
\bigr),
\qquad
\ell=\lceil n/2\rceil,
\]
is obtained from $\psi_s$ and $\phi_t$ through the
geometric shuffle-product rule. We denote this
lift by
\[
\sigma_n(\psi_s,\phi_t)_u.
\]
The truncation condition on $\zeta$ ensures that
all coefficients required by this construction
remain within the chosen auxiliary hierarchy.

We can now obtain the integro-differential
formulation for smooth rough paths.

Recall that the diagonal derivative of a smooth
rough path $\mathbf X$ is the Lie-algebra-valued path
\[
\boldsymbol x_u
:=
\left.\partial_v\right|_{v=u}
\mathbf X_{u,v}
\in\mathcal L^{\leq\kappa}(V).
\]
In particular, smooth rough-path integration
can be expressed as ordinary integration against
the diagonal derivative:
\[
\int_s^t H_u\,d\mathbf X_u
=
\int_s^t H_u[\boldsymbol x_u]\,du,
\]
where $H$ is controlled by $\mathbf{X}$.

For the first equation of the hierarchy, we use
the same convention when writing
$K_1(s,t)[\boldsymbol x_t]$: the contraction is
taken against the full controlled extension of
$K_1$, whose homogeneous coefficients are
\[
K_1(s,t)[e_{j_1}\otimes\cdots\otimes e_{j_r}]
:=
K_{j_r}^{(j_1\cdots j_{r-1})}(s,t),
\qquad 1\leq r\leq\kappa.
\]
For $r=1$, this agrees with the original
first-level auxiliary process.

\begin{lemma}\label{lemma_SD_smooth_system}
Let $\mathbf X\in\Omega G_{p,\omega}([0,T],V)$
be a smooth rough path. Then the auxiliary
Schwinger--Dyson hierarchy satisfies
\[
\frac{\partial^2K_0}{\partial s\,\partial t}(s,t)
=
\frac{\partial}{\partial s}
K_1(s,t)[\boldsymbol x_t]
\]
and, for $n\geq1$,
\[
\begin{aligned}
\frac{\partial^2K_n}{\partial s\,\partial t}(s,t)
={}&
\left.
\frac{\partial}{\partial t}
\sigma_n(\psi_s,\phi_t)_u
\right|_{u=s}
[\boldsymbol x_s]-
\left.
\frac{\partial}{\partial s}
\sigma_n(\psi_s,\phi_t)_u
\right|_{u=t}
[\boldsymbol x_t]
\\
&-
\int_s^t
\frac{\partial^2}{\partial s\,\partial t}
\sigma_n(\psi_s,\phi_t)_u
[\boldsymbol x_u]\,du.
\end{aligned}
\]
The diagonal conditions are
\[
K_0(s,s)=1,
\qquad
K_n(s,s)=0,\quad n\geq1.
\]
\end{lemma}

\begin{proof}
By the smooth rough-path integration formula,
the first identity in
\eqref{alternative_representation_levelwise}
becomes
\[
K_0(s,t)
=
1+\int_s^t
K_1(s,u)[\boldsymbol x_u]\,du.
\]
Differentiating first with respect to $t$ and
then with respect to $s$ gives
\[
\frac{\partial^2K_0}{\partial s\,\partial t}(s,t)
=
\frac{\partial}{\partial s}
K_1(s,t)[\boldsymbol x_t].
\]

For $n\geq1$, the controlled lift constructed
above yields
\[
K_n(s,t)
=
-\int_s^t
\sigma_n(\psi_s,\phi_t)_u
[\boldsymbol x_u]\,du.
\]
Applying Leibniz's rule in the initial endpoint,
we obtain
\[
\begin{aligned}
\frac{\partial K_n}{\partial s}(s,t)
={}&
\sigma_n(\psi_s,\phi_t)_s
[\boldsymbol x_s]-
\int_s^t
\frac{\partial}{\partial s}
\sigma_n(\psi_s,\phi_t)_u
[\boldsymbol x_u]\,du.
\end{aligned}
\]
Differentiating once more with respect to $t$
gives the asserted integro-differential identity.

Finally, the diagonal conditions follow directly
from the chronological integral representation
\eqref{index_removal}.
\end{proof}

\begin{remark}
The enlargement of the auxiliary hierarchy to
degree $\kappa+\zeta$. This is sufficient to obtain a
closed system of equations. Indeed, the controlled coefficients
required by the integrands may involve auxiliary
processes of strictly higher degree than the
original components. The choice of $\zeta$ ensures
that all such coefficients remain within the
selected truncation.
\end{remark}

\begin{example}[The case $2<p<3$]
\label{example:smooth_p23}
Suppose that $2<p<3$. Then $\kappa=2$, and we
may take $\zeta=0$. The auxiliary hierarchy
therefore involves only $K_0$, $K_1$, and $K_2$.

Proposition~\ref{numerical_scheme_rep} gives
\begin{align}
K_i(s,t)
&=
-\int_s^t
K_0(s,u)K_0(u,t)\,d\mathbf X_u^{(i)},
\label{eq:p23_v}\\
K_{ij}(s,t)
&=
-\int_s^t
K_0(s,u)K_j(u,t)\,d\mathbf X_u^{(i)}.
\label{eq:p23_w}
\end{align}

The first controlled coefficients in the
terminal parameter satisfy
\[
K_0^{(j)}(s,t)=K_j(s,t),
\qquad
K_i^{(j)}(s,t)
=
K_{ij}(s,t)-\delta_{ij}K_0(s,t),
\]
whereas those in the initial parameter satisfy
\[
K_0^{s,(j)}(s,t)=-K_j(s,t),
\qquad
K_i^{s,(j)}(s,t)
=
-K_{ji}(s,t)+\delta_{ij}K_0(s,t).
\]

These identities make explicit the different
word-ordering conventions associated with the
two endpoints and illustrate the closure of the
truncated auxiliary hierarchy.
\end{example}

\section{Numerical method for Schwinger-Dyson kernel}
\label{numerical_method_section}

\subsection{Methodology}

The construction developed in the previous section, together with
Propositions~\ref{Gubinelli derivatives} and
\ref{numerical_scheme_rep}, provides a framework for constructing
a numerical scheme for the Schwinger--Dyson kernel equation.

Let
\(
\mathcal P=\{0=t_0<t_1<\cdots<t_Q=T\}
\)
be a finite partition of $[0,T]$. The scheme is defined on the
triangular grid
\(
\mathcal P^2\cap\Delta_{[0,T]},
\) and \(
\mathcal P^2:=\mathcal P\times\mathcal P.
\)

For $0\leq i\leq j\leq Q$, we use the shorthand
\[
K_\bullet(i,j):=K_\bullet(t_i,t_j).
\]

The diagonal conditions are
\[
K_\emptyset(i,i)=1,
\qquad
K_I(i,i)=0,\quad |I|\geq1,
\]
where the auxiliary hierarchy is truncated at degree
$\kappa+\zeta$, with $\zeta$ defined in
Subsection~\ref{sec:smooth_RP_SD}.

Let
\[
D:=\dim T^{\leq\kappa+\zeta}(\mathbb R^d).
\]

We retain a right-endpoint discretisation. To make the ordering
of the resulting signature terms explicit, recall that the
right-endpoint rough Taylor approximation of a controlled
integrand $H$ is obtained through the antipode.

More precisely, for a coordinate direction $e_r$, the local
approximation to its rough integral over $[a,b]$ is
\begin{equation}\label{eq:SD_right_approximation}
\sum_{|\omega|=0}^{\kappa-1}
(-1)^{|\omega|}
H_b^{(\omega)}[e_r]\,
\mathbf X_{a,b}^{(r\operatorname{rev}(\omega))}.
\end{equation}

Indeed, under the chronological signature convention,
\[
S(\mathbf X)_{u,b}^{-1}
=
\mathcal A(S(\mathbf X)_{u,b}),
\]
where the antipode satisfies
\[
\mathcal A(e_\omega)
=
(-1)^{|\omega|}
e_{\operatorname{rev}(\omega)}.
\]
The corresponding right-endpoint approximating sums converge
to the rough integral by the standard controlled rough
integration estimates introduced in Section \ref{section_review}.

We now apply this construction to the auxiliary hierarchy.

For the level-zero component, Proposition
\ref{numerical_scheme_rep} gives
\[
K_\emptyset(s,t)
=
1+\int_s^t K_1(s,u)\,d\mathbf X_u.
\]

We therefore set
\begin{equation}\label{trace_numerical}
\begin{aligned}
K_\emptyset(i,j+1)
=
1
+
\sum_{m=i+1}^{j+1}
\sum_{r=1}^{d}
\sum_{|\omega|=0}^{\kappa-1}
&(-1)^{|\omega|}
K_r^{(\omega)}(i,m)
\mathbf X_{m-1,m}^{(r\operatorname{rev}(\omega))}.
\end{aligned}
\end{equation}
In this expression, $K_r^{(\omega)}$ denotes the terminal
controlled coefficient of the auxiliary process $K_r$,
as characterised in Remark~\ref{remark_smooth_numerical}.

For higher-level components, let
\( I=i_1\cdots i_n\), with \( |I|\geq1,
\)
and set
\[
\ell=\left\lceil\frac n2\right\rceil,
\qquad
\widetilde I=i_1\cdots i_{\ell-1},
\qquad
\widehat I=i_{\ell+1}\cdots i_n.
\]

Recall that
\[
K_I(s,t)
=
-\int_s^t
K_{\widetilde I}(s,u)
K_{\widehat I}(u,t)
\,d\mathbf X_u^{(i_\ell)}.
\]

The integrand is controlled in the integration variable $u$
by the terminal controlled lift of $K_{\widetilde I}(s,u)$
and the initial controlled lift of $K_{\widehat I}(u,t)$.

These lifts are precisely the maps $\psi$ and $\phi$
introduced in \eqref{eq:SD_psi_phi}.

For a word $\omega$, let $Sh^{-1}(\omega)$ denote the
inverse shuffle, understood with shuffle multiplicities.
Equivalently, its elements and multiplicities are determined
by
\[
\Delta e_\omega
=
\sum_{(\alpha,\beta)\in Sh^{-1}(\omega)}
e_\alpha\boxtimes e_\beta,
\]
where $\Delta$ denotes the shuffle coproduct.

The shuffle-product rule for controlled paths then yields
\[
\begin{aligned}
&
\left(
K_{\widetilde I}(s,\cdot)
K_{\widehat I}(\cdot,t)
\right)^{(\omega)}_u =
\sum_{(\alpha,\beta)\in Sh^{-1}(\omega)}
\psi_{s,u}(e_{\widetilde I})[e_\alpha]\,
\phi_{u,t}(e_{\widehat I})[e_\beta].
\end{aligned}
\]

Applying the right-endpoint approximation
\eqref{eq:SD_right_approximation}, we obtain
\begin{equation}\label{eq:SD_numerical_higher}
\begin{aligned}
K_I(i,j+1)
={}&
\sum_{m=i+1}^{j+1}
\sum_{|\omega|=0}^{\kappa-1}
\sum_{(\alpha,\beta)\in Sh^{-1}(\omega)}
(-1)^{|\omega|+1}
\psi_{i,m}(e_{\widetilde I})[e_\alpha]\,
\phi_{m,j+1}(e_{\widehat I})[e_\beta]\,
\mathbf X_{m-1,m}^{(i_\ell\operatorname{rev}(\omega))}.
\end{aligned}
\end{equation}

Notice that the initial-endpoint signs are already included
in the definition of $\phi$. In particular, no additional
factor $(-1)^{|\beta|}$ is required in
\eqref{eq:SD_numerical_higher}.

We separate the final contribution, corresponding to
$m=j+1$, from the preceding terms:
\begin{align}
K_I(i,j+1)
={}&
\sum_{m=i+1}^{j}
\sum_{|\omega|=0}^{\kappa-1}
\sum_{(\alpha,\beta)\in Sh^{-1}(\omega)}
(-1)^{|\omega|+1}
\psi_{i,m}(e_{\widetilde I})[e_\alpha]\,
\phi_{m,j+1}(e_{\widehat I})[e_\beta]\,
\mathbf X_{m-1,m}^{(i_\ell\operatorname{rev}(\omega))}
\label{term_b}\\
&+
\sum_{|\omega|=0}^{\kappa-1}
\sum_{(\alpha,\beta)\in Sh^{-1}(\omega)}
(-1)^{|\omega|+1}
\psi_{i,j+1}(e_{\widetilde I})[e_\alpha]\,
\phi_{j+1,j+1}(e_{\widehat I})[e_\beta]\,
\mathbf X_{j,j+1}^{(i_\ell\operatorname{rev}(\omega))}.
\label{term_A}
\end{align}

The controlled coefficients appearing above are explicitly
expressed in terms of the auxiliary processes $K_\bullet$
through Proposition~\ref{Gubinelli derivatives},
Remark~\ref{remark_smooth_numerical}, and
\eqref{eq:SD_psi_phi}.

In particular, the diagonal coefficients
\[
\phi_{j+1,j+1}(e_{\widehat I})[e_\beta]
\]
are determined entirely by the corresponding non-crossing contraction identities and the diagonal condition
\[
K_\bullet(j+1,j+1)=\varepsilon,
\]
where $\varepsilon$ denotes the counit of the tensor algebra.

The scheme can therefore be implemented column by column.
For each fixed terminal grid index $j+1$, the initial
indices are processed in decreasing order,
\[
i=j,j-1,\ldots,0.
\]

With this ordering, the terms in \eqref{term_b} depend
only on previously computed grid values. The contribution
\eqref{term_A}, together with the final term in
\eqref{trace_numerical}, depends linearly on the unknown
vector $K_\bullet(i,j+1)$.

Consequently, the numerical update can be written as
\begin{equation}\label{eq:SD_linear_system}
(\operatorname{Id}_D-A_j)K_\bullet(i,j+1)
=
b_{i,j},
\end{equation}
where
\(
b_{i,j}\in\mathbb R^D\) and \(A_j\in\mathbb R^{D\times D}.
\)

The matrix $A_j$ depends on the rough-path increment
$\mathbf X_{j,j+1}$ and on the universal diagonal
contraction coefficients, but not on the initial grid
index $i$.

In particular, the same coefficient matrix may be
reused for every initial index in a fixed terminal
column, while the right-hand side $b_{i,j}$ is updated
using the previously computed values of the auxiliary
hierarchy.

We conclude by recalling the runtime and memory estimates
for the resulting algorithm, derived in
Appendix~\ref{appendix:algorithmic_complexity}.

Let $Q=2^N$, and denote by $\mathrm{Poly}$ and
$\mathrm{Poly}_{\emptyset}$ the numbers of retained
terms in the compiled higher-level and trace-level
polynomial rules, respectively.

As established in Appendix~\ref{appendix:algorithmic_complexity},
when the coefficient matrices are factorised once per
terminal column, there exists a constant $C>0$ such that
the runtime and memory requirements of the solver satisfy
\[
\mathrm{Cost}_{\mathrm{runtime}}
\leq C\left(
Q^3(\mathrm{Poly}+\mathrm{Poly}_{\emptyset})
+
Q^2D^2+QD^3+\mathrm{SigCompute}
\right)
\]
and
\[
\mathrm{Memory}
\leq C\left(
Q^2D+QD^2+
\mathrm{Poly}+\mathrm{Poly}_{\emptyset}
\right).
\]
Here $C$ is independent of $Q$, $D$, and the
compiled polynomial sizes.

Explicit combinatorial upper bounds for the compiled
polynomial sizes are provided in
Appendix~\ref{appendix:algorithmic_complexity}.

% Let $M>N$ be positive integers. Assume that the driving
% path is sampled on a fine dyadic grid with $2^M$ points,
% while the numerical scheme is implemented on a coarser
% dyadic grid with $2^N$ points.

% Each increment of the rough path on the coarse grid
% is obtained by combining $2^{M-N}$ consecutive fine-grid
% increments.

% Let
% \[
% \widetilde D
% =
% \dim T^{\leq\kappa}(\mathbb R^d).
% \]

% The runtime estimate is
% \[
% \begin{aligned}
% \mathrm{Cost}_{\mathrm{runtime}}
% \leq C\Big(
% &2^{3N}(2d)^{\kappa+\zeta-1}(d+1)^\kappa\\
% &+2^{2N}D^3
% +\widetilde D(2^N-1)d^{\widetilde D}2^{M-N}
% \Big),
% \end{aligned}
% \]
% while the memory requirement satisfies
% \[
% \mathrm{Memory}
% \leq C\,2^{2N}D,
% \]
% where $C>0$ is independent of the parameters appearing
% in these estimates.

\subsection{Numerical experiments}
In this part of the work, we perform numerical experiments using the methods developed above.

The first experiment investigates how different choices of \(\kappa\) affect the performance of the algorithm as the regularity of the driving signal decreases. For a given choice of \(\kappa\), we denote the corresponding method by \(SDK\kappa\).

We also include in the comparison the randomized CDE method for the Schwinger--Dyson kernel, denoted by \(SDKr\), introduced in \cite{cass2024free} for the computation of the kernel in the bounded-variation setting. The parameters used for this method are matrix dimension \(200\) and \(250\) Monte Carlo simulations.

Simulated paths of three-dimensional fractional Brownian motion on \([0,1]\) with various Hurst indices \(H\) serve as input to the different numerical methods. Since no closed-form solution is available, the discrepancies reported below should be interpreted as empirical measures of agreement between numerical schemes, rather than as errors with respect to an exact reference solution. More precisely, for each of \(50\) independent sample paths of the driving noise, each consisting of \(4096\) piecewise-interpolated increments, we compute the corresponding numerical solutions and, for each pair of methods, evaluate the terminal-time discrepancy
\[
\left|K^{(\emptyset),\text{method 1}}_{0,1}-K^{(\emptyset),\text{method 2}}_{0,1}\right|.
\]
We then report the mean absolute error and the standard deviation of this quantity over the \(50\) sample paths. For the methods \(SDK2\) and \(SDK3\), the rough path is constructed over blocks of \(32\) increments. This procedure provides a simple terminal-time diagnostic for assessing inter-method consistency and for quantifying how decreasing input regularity affects the discrepancy between schemes.

We emphasize that some of the schemes considered below are tested on driving signals whose regularity lies outside the regime in which the corresponding methods are theoretically justified. In such cases, these schemes are used only as heuristic numerical benchmarks. Therefore, the observed agreement or discrepancy between methods should be interpreted as an empirical comparison, and not as a conclusion supported by the available convergence theory.

\begin{table}[H]
\centering
\begin{tabular}{c|c|c|c}
\hline
$H$ & Method pair & MAE & STD \\
\hline
\multirow{6}{*}{0.85} 
 & SDK3 vs SDK2 & $2.02\times 10^{-11}$ & $1.76\times 10^{-11}$ \\
 & SDK3 vs SDK1 & $5.95\times 10^{-6}$  & $1.21\times 10^{-6}$ \\
 & SDK3 vs SDKr & $2.38\times 10^{-5}$  & $4.85\times 10^{-6}$ \\
 & SDK2 vs SDK1 & $5.95\times 10^{-6}$  & $1.21\times 10^{-6}$ \\
 & SDK2 vs SDKr & $2.38\times 10^{-5}$  & $4.85\times 10^{-6}$ \\
 & SDK1 vs SDKr & $1.79\times 10^{-5}$  & $3.64\times 10^{-6}$ \\
\hline
\multirow{6}{*}{0.50} 
 & SDK3 vs SDK2 & $2.68\times 10^{-9}$  & $5.97\times 10^{-10}$ \\
 & SDK3 vs SDK1 & $1.84\times 10^{-4}$  & $1.22\times 10^{-5}$ \\
 & SDK3 vs SDKr & $7.36\times 10^{-4}$  & $4.89\times 10^{-5}$ \\
 & SDK2 vs SDK1 & $1.84\times 10^{-4}$  & $1.22\times 10^{-5}$ \\
 & SDK2 vs SDKr & $7.36\times 10^{-4}$  & $4.89\times 10^{-5}$ \\
 & SDK1 vs SDKr & $5.52\times 10^{-4}$  & $3.67\times 10^{-5}$ \\
\hline
\multirow{6}{*}{0.255} 
 & SDK3 vs SDK2 & $2.91\times 10^{-7}$  & $4.41\times 10^{-8}$ \\
 & SDK3 vs SDK1 & $1.98\times 10^{-3}$  & $1.57\times 10^{-4}$ \\
 & SDK3 vs SDKr & $7.96\times 10^{-3}$  & $6.32\times 10^{-4}$ \\
 & SDK2 vs SDK1 & $1.98\times 10^{-3}$  & $1.57\times 10^{-4}$ \\
 & SDK2 vs SDKr & $7.96\times 10^{-3}$  & $6.32\times 10^{-4}$ \\
 & SDK1 vs SDKr & $5.97\times 10^{-3}$  & $4.75\times 10^{-4}$ \\
\hline
\end{tabular}
\caption{Pairwise MAE and STD between the numerical methods for different Hurst parameters \(H\).}
\label{tab:method_errors_updated}
\end{table}

From Table~\ref{tab:method_errors_updated}, we observe that, as the regularity of the input signal decreases, both the MAE and the STD between the different methods increase. We also note that the discrepancy between \(SDK3\) and \(SDK2\) is smaller than the discrepancy between either of these methods and the remaining ones.

Moreover, when the Hurst index decreases from \(0.5\) to \(0.255\), the pair \(SDK3\) vs \(SDK2\) exhibits the largest relative increase in both MAE and STD, as can be seen by comparing the corresponding ratios in Table~\ref{tab:method_errors_updated}. This behaviour is consistent with the fact that the schemes \(SDK_\kappa\) have increasing approximation order. Heuristically, this suggests that the discrepancy between \(SDK2\) and \(SDK1\) should be larger than that between \(SDK3\) and \(SDK2\).

The method $\mathrm{SDK}_\kappa$ is theoretically justified for drivers
admitting a geometric $p$-rough-path lift with
\(
p<\kappa+1.
\)
For fractional Brownian motion with Hurst parameter $H$, one may choose
any $p>1/H$; hence such a $p$ exists precisely when
\(
1/H<\kappa+1.
\)
Outside this regime, the numerical comparisons below are used only as
heuristic benchmarks.
Therefore, as the regularity decreases, the discrepancy between successive schemes is expected to increase, which is consistent with the more pronounced effect observed as \(H\) approaches \(0.255\).

Below, we also plot the runtime cost for a kernel driven by three-dimensional Brownian motion, first as the grid size increases and then as the block size increases. The numerical scheme used in these simulations is \(SDK2\).

\begin{figure}[H]
  \centering
  \begin{subfigure}[t]{0.48\textwidth}
    \centering
    \includegraphics[width=\linewidth]{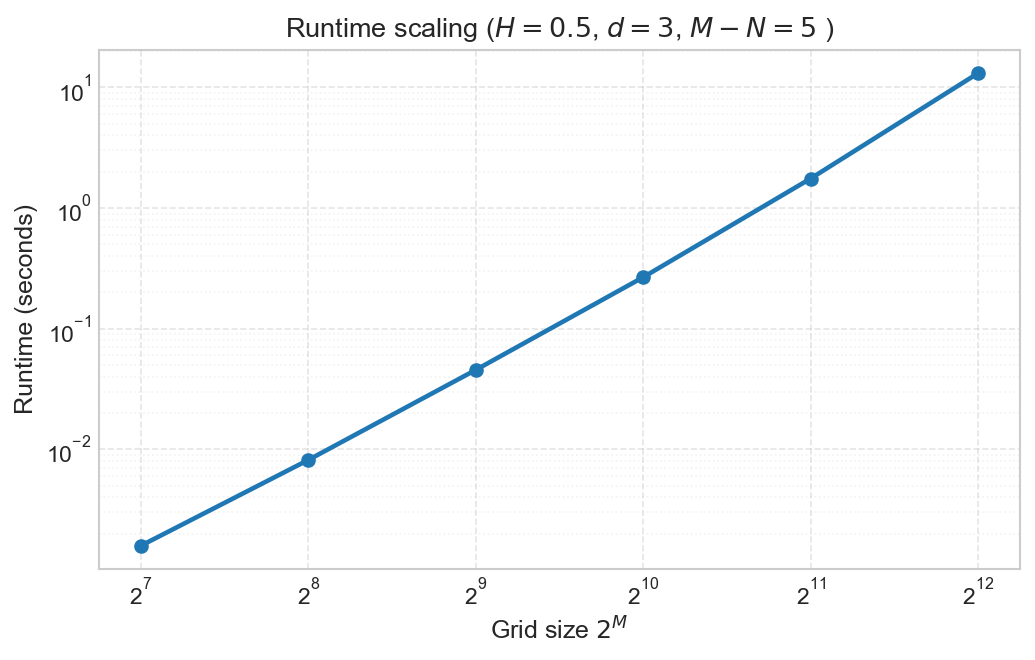}
    \caption{Runtime as a function of the grid size \(2^M\).}
    \label{fig:test1}
  \end{subfigure}%
  \hfill
  \begin{subfigure}[t]{0.48\textwidth}
    \centering
    \includegraphics[width=\linewidth]{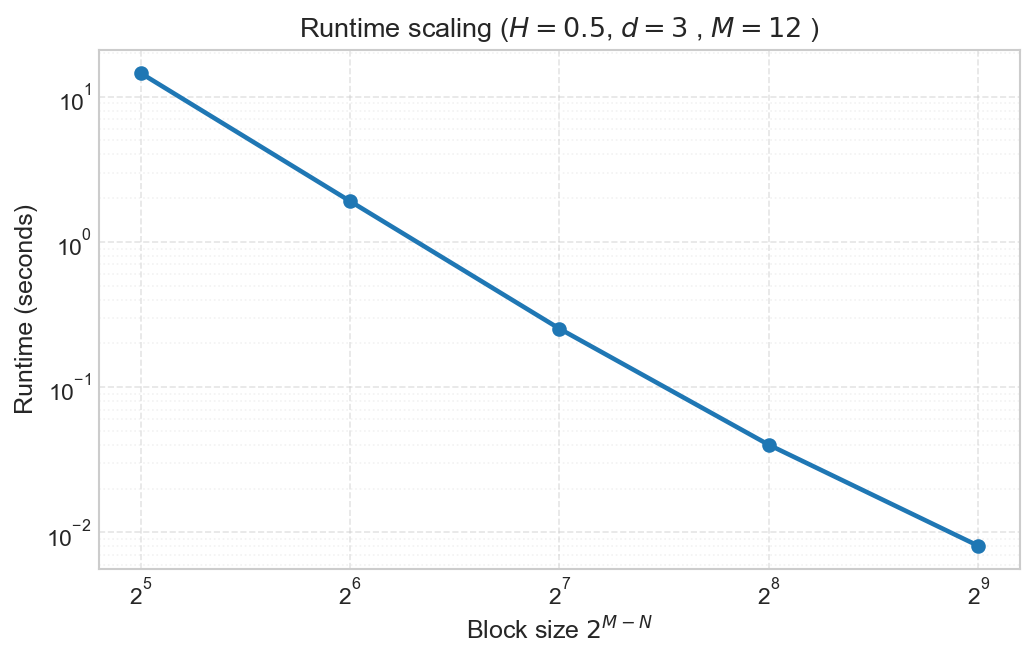}
    \caption{Runtime as a function of \(2^{M-N}\), the number of fine-grid increments used to construct each rough-path increment.}
    \label{fig:test2}
  \end{subfigure}
  \caption{Computational times for a kernel driven by three-dimensional Brownian motion, obtained by varying (a) the grid size and (b) the block size.}
  \label{fig:overall}
\end{figure}

We observe that, in accordance with our estimates, the runtime scales exponentially with both the grid size and the block size.

The next plot, in which we fix the grid size to \(M=12\) and the block size to \(M-N=5\), illustrates the polynomial scaling of the runtime complexity as the dimension \(d\) increases. The numerical scheme used in this experiment is again \(SDK2\).

\begin{figure}[H]
\centering
\includegraphics[scale=0.5]{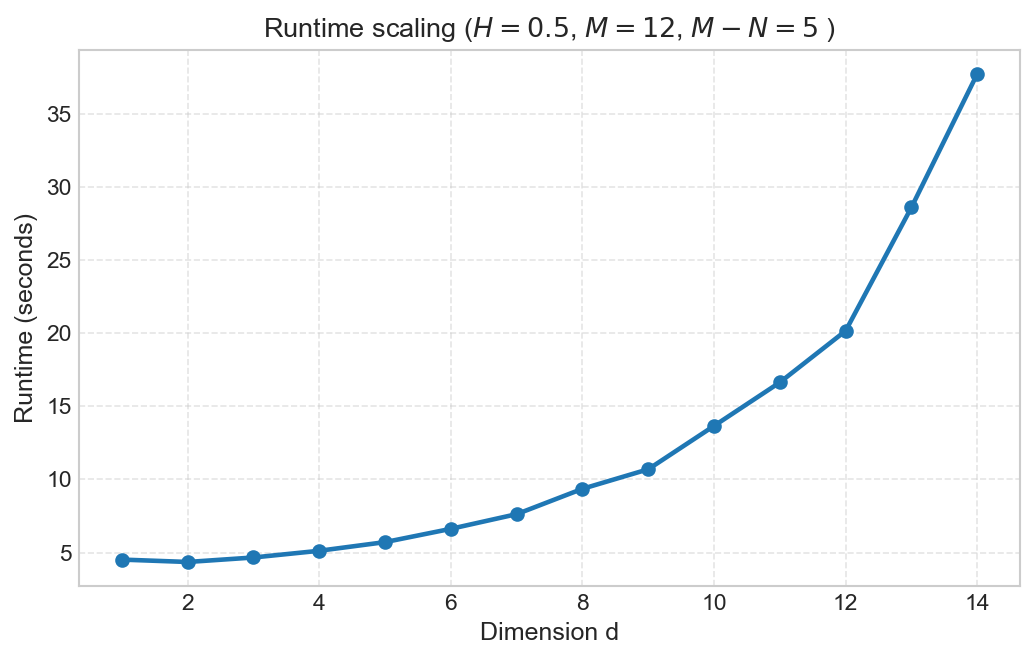} 
\caption{Runtime as a function of the dimension \(d\) for a kernel driven by Brownian motion.}
\label{fig:runtime_dimension}
\end{figure}

\appendix
\section{Technical results on two-parameter controlled paths}\label{appendix_technical}
We begin this section by recalling two of the main results in \cite{cass2023fubini}.
\begin{theorem}[Two-parameter rough integral on rectangular domains]\label{existence_2D_theorem}
Let 
\[
Z \in \mathcal{D}^{p, q}_{\mathbf{X}, \mathbf{Y}}\big([a, b] \times [c,d], \mathrm{Hom}(V \boxtimes W, F)\big).
\] 
Then the limit
\[
\int_{[a,b] \times [c,d]} Z_{s, u} \, d(\mathbf{X}, \mathbf{Y})  
   :=  \lim_{\left|\mathcal{P} \times \mathcal{D}\right| \to 0} 
   \sum_{\mathcal{P}\times \mathcal{D}} \Omega^Z
\]
exists and satisfies the estimate
\begin{align*}
\Bigg|
\int_{[a,b]\times[c,d]} Z_{s,u}\, d(\mathbf{X}_s,\mathbf{Y}_u)
- \Omega^Z\begin{pmatrix} a, b \\ c, d \end{pmatrix}
\Bigg|
&\le C_{p,q,T}\Big[
      A(\theta,T,\mathbf{R})\,
      \omega_{\mathbf{X}}(a,b)^{\theta}\,
      \omega_{\mathbf{Y}}(c,d)^{\theta} \\
&\qquad
      + B_X(c,d)\,\omega_{\mathbf{X}}(a,b)^{\theta}
      + B_Y(a,b)\,\omega_{\mathbf{Y}}(c,d)^{\theta}
\Big].
\end{align*}

where $\theta = \min(\frac{\kappa + 1}{p}, \frac{\lambda + 1}{q})$ and 
\begin{align*}
&A(p, q, \mathbf{R}) 
= \left( \sup_{j \in [0:\kappa)} \| \mathbf{X}^{(j+1)} \|_{p / j} \right)
  \left( \sup_{k \in  [0:\lambda)} \| {\mathbf{Y}}^{(k+1)} \|_{q / k} \right)
  \left( \sup_{(j,k)\in [0:\kappa) \times [0:\lambda)} \| \mathbf{R}^{(j,k)} \|_{(\frac{p}{\kappa - j},\frac{q}{q-k})} \right),\\
&B_{X}(c,d)
= 
\sup_{(j,k)\in [0:\kappa) \times [0:\lambda)}\left[ 
  \|\mathbf{X}^{(j+1)}\|_{\frac{p}{j+1}}\,
  \|\mathbf{Y}^{(k+1)}\|_{\frac{q}{k+1}}\,
  \big\| R^{1, (j,k)}
       \big\|_{\frac{p}{\kappa - j}}\,
  \omega_{\mathbf{Y}}(c,d)^{\frac{k+1}{q}}
\right],\\
&B_{Y}(a,b)
= \sup_{(j,k)\in [0:\kappa) \times [0:\lambda)}
\left[
   \|\mathbf{X}^{(j+1)}\|_{\frac{p}{j+1}}\,
  \|\mathbf{Y}^{(k+1)}\|_{\frac{q}{k+1}}\,
  \big\| R^{2,(j,k)}\!
       \big\|_{\frac{q}{\lambda - k}}\,
  \omega_{\mathbf{X}}(a,b)^{\frac{j+1}{p}}
\right].
\end{align*}

\end{theorem}

\begin{theorem}[Stability of two-parameter rough integrals]\label{stability_theorem_2D}
Let $\mathbf{X}, \tilde{\mathbf{X}} \in  \Omega G_{p, \omega_{\mathbf{X}}}([a,b], V)$ and $\mathbf{Y}, \tilde{\mathbf{Y}} \in  \Omega G_{q, \omega_\mathbf{Y}} ([c,d], W)$. Consider the controlled paths
\[
Z \in \mathcal{D}^{p, q}_{\mathbf{X}, \mathbf{Y}}\big([a, b] \times [c,d], F\big), \, \tilde{Z}\in \mathcal{D}^{p, q}_{\tilde{\mathbf{X}}, \tilde{\mathbf{Y}}}\big([a, b] \times [c,d], F\big).
\] 
Then the following estimate holds
\begin{align*}
&\left| \int_{[a,b] \times [c,d]} Z_{s, u} \, d(\mathbf{X}_s, \mathbf{Y}_u) 
      - \int_{[a,b] \times [c,d]} \tilde{Z}_{s, u} \, d(\tilde{\mathbf{X}}_s, \tilde{\mathbf{Y}}_u) \right|\\
&\qquad \leq C_{\omega_\mathbf{X}, \omega_\mathbf{Y}, p, q, Z, \tilde{Z}} 
\, \Big( D_{\omega_{\mathbf{X}}, \omega_{\mathbf{Y}}} d( Z, \tilde{Z} ) + \|\mathbf{X} - \tilde{\mathbf{X}}\|_{p} + \|\mathbf{Y} - \tilde{\mathbf{Y}}\|_{q} \Big),
\end{align*}
where $D_{\omega_{\mathbf{X}}, \omega_{\mathbf{Y}}} := \omega_{\mathbf{X}}(a,b)^{1/p} \omega_{\mathbf{Y}}(c,d)^{1/q}$.

\end{theorem}

Below we present the proof of Propositions \ref{stability_under_composition} and \ref{stability_proposition}.\newline
We only treat the case $j>0$ and $k>0$, since the remaining cases follow
directly from the classical one-parameter theory of controlled rough paths for  the first level remainders and are much simpler to prove for the second order remainders. To improve readability we replace $id_{T^{\kappa}(V)}$ and $id_{T^{\lambda}(W)}$ by $id$ whenever it is clear from the context. Moreover we introduce the notation $\Delta^{\boxtimes}_{>0} := \widetilde{\Delta \boxtimes \Delta}$ .  We also allow the remainder terms to change from line to line.
Since the controlled-path norms are bounded by $M$,
the ranges of $Z^{(0,0)}$ and
$\widetilde Z^{(0,0)}$ are contained in a bounded
set. By the smoothness of $g$, all derivatives
required in the following Taylor expansions and
difference estimates are uniformly bounded on
a sufficiently large compact ball containing
these ranges.

Throughout the proof, $C_M$ denotes a constant
depending only on these local derivative bounds,
$M$, the fixed parameters and the common controls.

\begin{proof}[Proof of Proposition 
\ref{stability_under_composition}]
Notice that, from the symmetry of $D^\ell g$, we can write
\begin{align}
&\sum_{\ell=1}^{\lambda+\kappa-2} \frac{1}{\ell!}
D^\ell g\left(Z^{(0,0)}_{t,v}\right)
\circ Z^{\otimes \ell}_{t,v}
\circ (\Delta^{\boxtimes}_{>0})^\ell \label{original_expression_expansion_g}\\
&=
\sum_{i = 1}^{\kappa -1} \sum_{m = 0}^{\lambda -1} \frac{1}{i!} \frac{1}{m!} D^{i+m}g(Z^{(0,0)}_{t,v}) \circ Z^{\otimes i + m}_{t,v} \circ \left( (\tilde{\Delta} \boxtimes \Delta)^{i} \otimes ( id \boxtimes \tilde{\Delta})^{m} \right). \nonumber
\end{align}
Expanding $(Z_{t,v})^{\otimes i+m}$ in the first parameter yields 
\begin{align}
&\sum_{i = 1}^{\kappa -1} \sum_{m = 0}^{\lambda -1} \frac{1}{i!} \frac{1}{m!} D^{i+m}g(Z^{(0,0)}_{t,v}) \circ (Z_{s,v} \circ [L_{\mathbf{X}^{<\kappa}_{s,t}} \otimes id])^{\otimes i + m} \circ \left((\tilde{\Delta} \boxtimes \Delta))^{i} \otimes ( id \boxtimes \tilde{\Delta})^{m} \right) + A_{s,t}(v),
\end{align}
where the term $A_{s,t}(v)$ absorbs all the terms in the expansion that contain at least a term $R^1_{s,t}(v)$.\newline
We now expand the derivatives of $g$ around $Z^{(0,0)}_{s,v}$ in the first term above.
This yields
\begin{align*}
&\sum_{i = 1}^{\kappa -1} \sum_{n = i}^{\kappa -1} \sum_{m = 0}^{\lambda -1} \frac{1}{n!} \frac{1}{m!} \binom{n}{i} D^{n+m}g(Z^{(0,0)}_{s,v})\\
& \qquad \circ (Z_{s,v} \circ [L_{\mathbf{X}^{<\kappa}_{s,t}} \otimes id])^{\otimes i + m} \circ \left((\tilde{\Delta} \boxtimes \Delta))^{i} \otimes ( id \boxtimes \tilde{\Delta})^{m} \otimes (Z^{(0,0)}_{t,v} - Z^{(0,0)}_{s,v})^{\otimes n-i}\right) + B_{s,t}(v).
\end{align*}
Here $B_{s,t}(v)$ collects the Taylor expansion remainder for the expansion above.\newline
Expanding the difference $Z^{(0,0)}_{t,v} - Z^{(0,0)}_{s,v}$, we obtain
\begin{align*}
&\sum_{i = 1}^{\kappa -1} \sum_{n = i}^{\kappa -1} \sum_{m = 0}^{\lambda -1} \frac{1}{n!} \frac{1}{m!} \binom{n}{i} D^{n+m}g(Z^{(0,0)}_{s,v})\\
& \qquad \circ (Z_{s,v})^{n+m} \circ \left([L_{\mathbf{X}^{<\kappa}_{s,t}} \otimes id]^{\otimes i + m} \circ \left((\tilde{\Delta} \boxtimes \Delta))^{i} \otimes ( id \boxtimes \tilde{\Delta})^{m} \right) \otimes \left[\mathbf{X}_{s,t} \boxtimes \mathbf{1} \right]^{\otimes n-i}\right) + C_{s,t}(v).
\end{align*}
Where the term $C$ groups the terms in which at least a term $R^{1, (0,0)}$ appears.\newline
From the estimate
\begin{align*}
&\Bigg\| \sum_{n = 1}^{\kappa -1}  \frac{1}{n!}  D^{n+m}g(Z^{(0,0)}_{s,v})\circ (Z_{s,v})^{n} \circ\\
&\qquad \left(\sum_{i = 1}^{n}\binom{n}{i} [L_{\mathbf{X}^{<\kappa}_{s,t}} \otimes id]^{\otimes i } \circ (\tilde{\Delta} \boxtimes \Delta)^{i}  \otimes \left[\mathbf{X}_{s,t} \boxtimes \mathbf{1} \right]^{\otimes n-i} - (\tilde{\Delta} \boxtimes \Delta)^{n}  \otimes [L_{\mathbf{X}^{<\kappa}_{s,t}} \otimes id]^{\otimes n } \right)[\eta_j, \xi_k]\Bigg\|\\
&\leq C \omega_\mathbf{X}(s,t)^{(\kappa - j)/p} \|\eta_j\| \|\xi_k \|,
\end{align*}
valid for every $m \leq \lambda -1$, $\eta_j \in V^{\otimes j}$, $\xi_k \in W^{\otimes k}$ and exploiting the symmetry of the derivatives of $g$, the previous identity reduces to
\begin{align}
&\sum_{n = 1}^{\kappa -1} \sum_{m = 0}^{\lambda -1} \frac{1}{n!} \frac{1}{m!} D^{n+m}g(Z^{(0,0)}_{s,v}) \circ (Z_{s,v} \circ [L_{\mathbf{X}^{<\kappa}_{s,t}} \otimes id])^{\otimes n + m} \circ \left((\tilde{\Delta} \boxtimes \Delta))^{n} \otimes ( id \boxtimes \tilde{\Delta})^{m} \right) + E_{s,t}(v).
\end{align}
So that we can rewrite the original expression \eqref{original_expression_expansion_g} as 
\[
\sum_{\ell=1}^{\lambda+\kappa-2} \frac{1}{\ell!}
D^\ell g\left(Z^{(0,0)}_{s,v}\right)
\circ Z^{\otimes \ell}_{s,v}
\circ (\Delta^{\boxtimes}_{>0})^\ell \circ \left[L_{\mathbf{X}^{< \kappa}_{s,t}} \boxtimes id \right] + R^{g, 1}_{s,t}(v)\\
\]
The remainder $R^{g, 1}_{s,t}(v)$ collects all higher-order terms arising from:
(i) the remainder in the controlled expansion of $Z$,
(ii) the Taylor remainder of $g$, (iii) the remainder in the expansion for $Z^{(0,0)}_{t,v} - Z^{(0,0)}_{s,v}$
and (iv) the truncation error in the shuffle expansion. Its explicit form is
given by
\begin{align*}
R^{g, 1}_{s,t}(v)
&= A_{s,t}(v) + B_{s,t}(v) + C_{s,t}(v) + E_{s,t}(v),
\end{align*}
where
\begin{align*}
A_{s,t}(v)
&:=
\sum_{\ell=1}^{\lambda+\kappa-2}
\sum_{i=1}^{\ell}
\frac{1}{\ell!}
D^\ell g\left(Z^{(0,0)}_{t,v}\right)
\circ
\left(
Z^{\otimes(i-1)}_{s,v} \circ \left[L_{\mathbf{X}^{< \kappa}_{s,t}} \boxtimes id \right]
\otimes
R^{1}_{s,t}(v)
\otimes
Z^{\otimes(\ell-i)}_{t,v}
\right)
\circ
(\Delta^{\boxtimes}_{>0})^\ell,
\\[1ex]
B_{s,t}(v)
&:=
\sum_{i=1}^{\kappa-1}
\sum_{m=0}^{\lambda-1}
\frac{1}{(\kappa-1)!}
\frac{1}{m!}
\binom{\kappa-1}{i}
\int_0^1
(1-r)^{\kappa-1-i}
D^{\kappa+m}g
\left(
(1-r)Z^{(0,0)}_{s,v}
+rZ^{(0,0)}_{t,v}
\right)\,dr
\\
&\qquad\circ
\left(
\left(
Z_{s,v}\circ
[L_{\mathbf X^{<\kappa}_{s,t}}\otimes id]
\right)^{\otimes(i+m)}
\circ
\left(
(\tilde{\Delta}\boxtimes\Delta)^i
\otimes
(id\boxtimes\tilde{\Delta})^m
\right)
\otimes
\left(
Z^{(0,0)}_{t,v}
-
Z^{(0,0)}_{s,v}
\right)^{\otimes(\kappa-i)}
\right),
\\[1ex]
C_{s,t}(v) &:= \sum_{i = 1}^{\kappa -1} \sum_{n = i}^{\kappa -1} \sum_{m = 0}^{\lambda -1} \sum_{p=1}^{n-i} \frac{1}{n!} \frac{1}{m!} \binom{n}{i} D^{n+m}g(Z^{(0,0)}_{s,v})\\
& \qquad \circ (Z_{s,v} \circ [L_{\mathbf{X}^{<\kappa}_{s,t}} \otimes id])^{\otimes i + m} \\
&\qquad \circ \left((\tilde{\Delta} \boxtimes \Delta))^{i} \otimes ( id \boxtimes \tilde{\Delta})^{m} \otimes \left(Z_{s,v} \circ \left[\mathbf{X}^{\geq 1 }_{s,t} \boxtimes id \right]\right)^{\otimes p-1} \otimes R^{1, (0,0)}_{s,t}(v) \otimes \left(Z^{(0,0)}_{t,v} - Z^{(0,0)}_{s,v}\right)^{\otimes n-i-p} \right) \\[1ex]
E_{s,t}(v)
&:=
\sum_{n=1}^{\kappa-1}
\sum_{m=0}^{\lambda-1}
\frac{1}{n!m!}
D^{n+m}g(Z^{(0,0)}_{s,v})
\circ Z^{\otimes(n+m)}_{s,v}
\\
&\quad\circ
\Bigg(
[id\boxtimes\tilde{\Delta}]^{\otimes m}
\otimes
\Bigg[
\sum_{r=1}^{n}
\binom{n}{r}
\left(
[L_{\mathbf X^{<\kappa}_{s,t}}\otimes id]^{\otimes r}
\circ
(\tilde{\Delta}\boxtimes\Delta)^r
\right)
\otimes
[\mathbf X_{s,t}\boxtimes\mathbf 1]^{\otimes(n-r)}
\\
&\hspace{35mm}
-
[L_{\mathbf X^{<\kappa}_{s,t}}\otimes id]^{\otimes n}
\circ
(\tilde{\Delta}\boxtimes\Delta)^n
\Bigg]
\Bigg).
\end{align*}
It is not hard to see that the each term of this remainder satisfies the regularity constraints in equation \eqref{regularity_condition_r1}, moreover, the expressions above satisfy for a $C_M > 0$
\[
\big\| R^{g,1, (j,k)}\!\left(v\right)
       \big\|_{\frac{p}{\kappa - j}; \omega_\mathbf{X}} \leq C_M \|Z\|_{\mathcal{D}^{p,q}_{\mathbf{X}, \mathbf{Y}}}
\]
so that we can conclude for the expansion in the first parameter.\newline
The verification of the second constraint in \eqref{constraints} follows by
identical arguments and is therefore omitted.\newline
.

It remains to verify that the second-order remainder satisfies
\eqref{remainder_remainder_norm}. Recall that the first-order
remainder obtained above admits the decomposition
\[
R^{g,1}_{s,t}(v)
=
A_{s,t}(v)+B_{s,t}(v)+C_{s,t}(v)+E_{s,t}(v).
\]
By symmetry, it is sufficient to show that each term in this
decomposition is controlled by $\mathbf Y$ in its remaining
parameter, with a remainder satisfying the required mixed
regularity estimate.

Starting with $A$, fix $u\le v$ and expand the factors in
the second parameter around $u$. This yields
\begin{align*}
A_{s,t}(v)
&=
\sum_{\ell=1}^{\lambda+\kappa-2}
\sum_{i=1}^{\ell}
\frac{1}{\ell!}
D^\ell g\left(Z^{(0,0)}_{t,v}\right)
\\
&\quad\circ
\left(
Z^{\otimes(i-1)}_{s,u}
\circ
[L_{\mathbf X_{s,t}^{<\kappa}}\boxtimes id]
\otimes
R^1_{s,t}(u)
\otimes
Z^{\otimes(\ell-i)}_{t,u}
\right)
\\
&\quad\circ
\left[
id\boxtimes L_{\mathbf Y_{u,v}^{<\lambda}}
\right]^{\otimes\ell}
\circ
(\Delta^{\boxtimes}_{>0})^\ell
+
\widehat{\mathbf R}^{A}
\begin{pmatrix}s,t\\u,v\end{pmatrix},
\end{align*}
where $\widehat{\mathbf R}^{A}$ collects the remaining
contributions arising from the controlled expansions of
$Z_{s,v}$, $Z_{t,v}$ and $R^1_{s,t}(v)$.

In particular, the latter expansion is
\[
R^1_{s,t}(v)
=
R^1_{s,t}(u)\circ
\left[
id_{T^{<\kappa}(V)}
\boxtimes L_{\mathbf Y_{u,v}^{<\lambda}}
\right]
+
\mathbf R
\begin{pmatrix}s,t\\u,v\end{pmatrix}.
\]
The mixed estimate \eqref{remainder_remainder_norm}
therefore applies to the contributions containing
$\mathbf R$.

We now expand $D^\ell g$ around $Z^{(0,0)}_{t,u}$.
For every $\ell$, Taylor's formula gives
\begin{align*}
D^\ell g\left(Z^{(0,0)}_{t,v}\right)
&=
\sum_{m=\ell}^{\lambda+\kappa-2}
\frac{1}{(m-\ell)!}
D^m g\left(Z^{(0,0)}_{t,u}\right)
\\
&\quad\circ
\left(
Z^{(0,0)}_{t,v}
-
Z^{(0,0)}_{t,u}
\right)^{\otimes(m-\ell)}
+
\operatorname{Rem}_{\ell}(u,v).
\end{align*}
The Taylor remainder is absorbed into
$\widehat{\mathbf R}^{A}$.

Substituting this expansion and reindexing the sums
by fixing $m$, we obtain
\begin{align*}
A_{s,t}(v)
&=
\sum_{i=1}^{\lambda+\kappa-2}
\sum_{m=i}^{\lambda+\kappa-2}
\sum_{\ell=i}^{m}
\frac{1}{m!}
\binom{m}{m-\ell}
D^m g\left(Z^{(0,0)}_{t,u}\right)
\\
&\quad\circ
\left(
Z^{\otimes(i-1)}_{s,u}
\otimes
R^1_{s,t}(u)
\otimes
Z^{\otimes(\ell-i)}_{t,u}
\otimes
\left(
Z^{(0,0)}_{t,v}
-
Z^{(0,0)}_{t,u}
\right)^{\otimes(m-\ell)}
\right)
\\
&\quad\circ
\left[
id\boxtimes L_{\mathbf Y_{u,v}^{<\lambda}}
\right]^{\otimes\ell}
\circ
(\Delta^{\boxtimes}_{>0})^\ell
+
\widehat{\mathbf R}^{A}
\begin{pmatrix}s,t\\u,v\end{pmatrix}.
\end{align*}

Next, we expand
$Z^{(0,0)}_{t,v}-Z^{(0,0)}_{t,u}$ using
\eqref{constraints}. Contributions containing the
second-parameter remainder $R^2$ are included in
$\widehat{\mathbf R}^{A}$.

Using the symmetry of the Fréchet derivatives and the
multiplicativity of the shuffle coproduct, the principal
terms can be reorganised into the transported expression
for $A_{s,t}(u)$. The difference arising from the
truncation of the shuffle expansion is included in the
second-order remainder. Consequently,
\begin{align*}
A_{s,t}(v)
&=
A_{s,t}(u)\circ
\left[
id\boxtimes L_{\mathbf Y_{u,v}^{<\lambda}}
\right]
+
\mathbf R^A
\begin{pmatrix}s,t\\u,v\end{pmatrix}.
\end{align*}

The estimates on $R^1$, $R^2$ and $\mathbf R$, together
with the Taylor remainder estimates and the graded
shuffle cancellations, give
\begin{align}
\left\|
\mathbf R^A
\begin{pmatrix}s,t\\u,v\end{pmatrix}
[\eta_j\boxtimes\xi_k]
\right\|
&\le
C_M
\omega_{\mathbf X}(s,t)^{(\kappa-j)/p}
\omega_{\mathbf Y}(u,v)^{(\lambda-k)/q}
\|\eta_j\|\|\xi_k\|.
\label{eq:composition_A_mixed_bound}
\end{align}

We now turn to the terms $B$ and $C$. Recall their
representations
\begin{align*}
B_{s,t}(v)
&=
\sum_{i=1}^{\kappa-1}
\sum_{m=0}^{\lambda-1}
\frac{1}{(\kappa-1)!m!}
\binom{\kappa-1}{i}
\int_0^1
(1-r)^{\kappa-1-i}
\\
&\quad\cdot
D^{\kappa+m}g
\left(
(1-r)Z^{(0,0)}_{s,v}
+rZ^{(0,0)}_{t,v}
\right)\,dr
\\
&\quad\circ
\left(
\left(
Z_{s,v}\circ
[L_{\mathbf X_{s,t}^{<\kappa}}\boxtimes id]
\right)^{\otimes(i+m)}
\circ
\left(
(\widetilde\Delta\boxtimes\Delta)^i
\otimes
(id\boxtimes\widetilde\Delta)^m
\right)
\right.
\\
&\hspace{35mm}\left.
\otimes
\left(
Z^{(0,0)}_{t,v}
-
Z^{(0,0)}_{s,v}
\right)^{\otimes(\kappa-i)}
\right),
\end{align*}
and
\begin{align*}
C_{s,t}(v)
&=
\sum_{i=1}^{\kappa-1}
\sum_{n=i}^{\kappa-1}
\sum_{m=0}^{\lambda-1}
\sum_{r=1}^{n-i}
\frac{1}{n!m!}
\binom{n}{i}
D^{n+m}g(Z^{(0,0)}_{s,v})
\\
&\quad\circ
\left(
Z_{s,v}\circ
[L_{\mathbf X_{s,t}^{<\kappa}}\boxtimes id]
\right)^{\otimes(i+m)}
\\
&\quad\circ
\left(
(\widetilde\Delta\boxtimes\Delta)^i
\otimes
(id\boxtimes\widetilde\Delta)^m
\right)
\\
&\quad\otimes
\left(
Z_{s,v}\circ
[\mathbf X_{s,t}^{\ge1}\boxtimes id]
\right)^{\otimes(r-1)}
\otimes
R^{1,(0,0)}_{s,t}(v)
\\
&\quad\otimes
\left(
Z^{(0,0)}_{t,v}
-
Z^{(0,0)}_{s,v}
\right)^{\otimes(n-i-r)}.
\end{align*}

For $B$, we expand the controlled coefficients and the
derivatives of $g$ in the second parameter, using
Taylor's formula in integral form. The first-parameter
Taylor remainder already provides the required factor
$\omega_{\mathbf X}(s,t)^{(\kappa-j)/p}$, while the
second-parameter expansion yields the additional factor
$\omega_{\mathbf Y}(u,v)^{(\lambda-k)/q}$.

For $C$, the same procedure is applied to each factor
in its finite sum. In particular, the expansion of
$R^1_{s,t}(v)$ produces the mixed remainder
$\mathbf R$, while the other contributions contain
controlled remainders or higher-order terms arising
from the Taylor expansion.

Collecting the transported principal terms, we obtain
\begin{align*}
B_{s,t}(v)
&=
B_{s,t}(u)\circ
\left[
id\boxtimes L_{\mathbf Y_{u,v}^{<\lambda}}
\right]
+
\mathbf R^B
\begin{pmatrix}s,t\\u,v\end{pmatrix},
\\
C_{s,t}(v)
&=
C_{s,t}(u)\circ
\left[
id\boxtimes L_{\mathbf Y_{u,v}^{<\lambda}}
\right]
+
\mathbf R^C
\begin{pmatrix}s,t\\u,v\end{pmatrix},
\end{align*}
where both $\mathbf R^B$ and $\mathbf R^C$ satisfy
the bound \eqref{remainder_remainder_norm}, with
a constant depending on the prescribed bound $M$.

We finally consider the shuffle-truncation term $E$.
Recall that
\begin{align*}
E_{s,t}(v)
&=
\sum_{n=1}^{\kappa-1}
\sum_{m=0}^{\lambda-1}
\frac{1}{n!m!}
D^{n+m}g(Z^{(0,0)}_{s,v})
\circ
Z_{s,v}^{\otimes(n+m)}
\\
&\quad\circ
\left(
[id\boxtimes\widetilde\Delta]^{\otimes m}
\otimes
G_n(\mathbf X_{s,t}^{\ge1})
\right),
\end{align*}
where $G_n(\mathbf X_{s,t}^{\ge1})$ denotes the
shuffle-truncation defect introduced in the preceding
calculation.

Expanding $Z$ in the second parameter and collecting
all terms containing at least one $R^2$ in the remainder,
we obtain
\begin{align*}
E_{s,t}(v)
&=
\sum_{n=1}^{\kappa-1}
\sum_{m=0}^{\lambda-1}
\frac{1}{n!m!}
D^{n+m}g(Z^{(0,0)}_{s,v})
\circ
Z_{s,u}^{\otimes(n+m)}
\\
&\quad\circ
\left(
\left(
[id\boxtimes L_{\mathbf Y_{u,v}^{<\lambda}}]^{\otimes m}
\circ
[id\boxtimes\widetilde\Delta]^{\otimes m}
\right)
\right.
\\
&\hspace{30mm}\left.
\otimes
\left(
[id\boxtimes L_{\mathbf Y_{u,v}^{<\lambda}}]^{\otimes n}
\circ
G_n(\mathbf X_{s,t}^{\ge1})
\right)
\right)
+
\widehat{\mathbf R}^{E}
\begin{pmatrix}s,t\\u,v\end{pmatrix}.
\end{align*}

The terms collected in $\widehat{\mathbf R}^{E}$
satisfy \eqref{remainder_remainder_norm}, since
$G_n$ has the required first-parameter truncation
order and every such term contains a second-parameter
controlled remainder.

We subsequently expand $D^{n+m}g$ around
$Z^{(0,0)}_{s,u}$ and absorb the Taylor remainder into
$\widehat{\mathbf R}^{E}$. This yields
\begin{align*}
E_{s,t}(v)
&=
\sum_{n=1}^{\kappa-1}
\sum_{m=0}^{\lambda-1}
\sum_{\ell=m}^{\lambda-1}
\frac{1}{n!\ell!}
\binom{\ell}{m}
D^{n+\ell}g(Z^{(0,0)}_{s,u})
\\
&\quad\circ
Z_{s,u}^{\otimes(n+\ell)}
\\
&\quad\circ
\left(
\left(
[id\boxtimes L_{\mathbf Y_{u,v}^{<\lambda}}]
\circ
[id\boxtimes\widetilde\Delta]
\right)^{\otimes m}
\right.
\\
&\hspace{20mm}\left.
\otimes
\left(
[id\boxtimes L_{\mathbf Y_{u,v}^{<\lambda}}]^{\otimes n}
\circ
G_n(\mathbf X_{s,t}^{\ge1})
\right)
\otimes
[\mathbf1\boxtimes\mathbf Y_{u,v}^{\ge1}]^{\otimes(\ell-m)}
\right)
\\
&\quad+
\widehat{\mathbf R}^{E}
\begin{pmatrix}s,t\\u,v\end{pmatrix}.
\end{align*}

The newly absorbed terms again satisfy
\eqref{remainder_remainder_norm}. Reorganising the
principal terms through the shuffle identities used
in the first-order calculation, we recover
\[
E_{s,t}(v)
=
E_{s,t}(u)\circ
\left[
id\boxtimes L_{\mathbf Y_{u,v}^{<\lambda}}
\right]
+
\mathbf R^E
\begin{pmatrix}s,t\\u,v\end{pmatrix},
\]
where $\mathbf R^E$ consists of the preceding
second-order remainder contributions and satisfies
the required mixed regularity estimate.

Combining the four calculations, we conclude that
\begin{align*}
R^{g,1}_{s,t}(v)
&=
R^{g,1}_{s,t}(u)\circ
\left[
id_{T^{<\kappa}(V)}
\boxtimes L_{\mathbf Y_{u,v}^{<\lambda}}
\right]
+
\mathbf R^g
\begin{pmatrix}s,t\\u,v\end{pmatrix},
\end{align*}
where
\[
\mathbf R^g
\begin{pmatrix}s,t\\u,v\end{pmatrix}
=
\sum_{F\in\{A,B,C,E\}}
\mathbf R^F
\begin{pmatrix}s,t\\u,v\end{pmatrix}.
\]

In particular,
\begin{align}
\left\|
\mathbf R^g
\begin{pmatrix}s,t\\u,v\end{pmatrix}
[\eta_j\boxtimes\xi_k]
\right\|
&\le
C_M
\omega_{\mathbf X}(s,t)^{(\kappa-j)/p}
\omega_{\mathbf Y}(u,v)^{(\lambda-k)/q}
\|\eta_j\|\|\xi_k\|.
\label{eq:composition_mixed_remainder_bound}
\end{align}

It remains to verify the compatibility of the two
representations in \eqref{eq:second_order_remainder}.
By substituting the definitions of the first-order
remainders of $g(Z)$, we obtain
\begin{align*}
& R^{g,1}_{s,t}(v)
-
R^{g,1}_{s,t}(u)\circ
\left[
id\boxtimes L_{\mathbf Y_{u,v}^{<\lambda}}
\right]
\\
&=
g(Z)_{t,v}
-
g(Z)_{s,v}\circ
\left[
L_{\mathbf X_{s,t}^{<\kappa}}\boxtimes id
\right]
-
g(Z)_{t,u}\circ
\left[
id\boxtimes L_{\mathbf Y_{u,v}^{<\lambda}}
\right]
\\
&\quad+
g(Z)_{s,u}\circ
\left[
L_{\mathbf X_{s,t}^{<\kappa}}
\boxtimes L_{\mathbf Y_{u,v}^{<\lambda}}
\right]
\\
&=
R^{g,2}_{u,v}(t)
-
R^{g,2}_{u,v}(s)\circ
\left[
L_{\mathbf X_{s,t}^{<\kappa}}\boxtimes id
\right],
\end{align*}
where we used the commutativity of the transport
operators acting on the two tensor factors.

Hence the second-order remainder is well-defined and
satisfies \eqref{remainder_remainder_norm}. Together
with the first-order estimates established above, this
proves that $g(Z)$ belongs to the required space of
two-parameter controlled paths.
\end{proof}

\begin{proof}[Proof of Proposition~\ref{stability_proposition}]
We now establish the stability estimate. The controlled
structure of the composition, including the mixed
regularity of its second-order remainder, has already
been established in
Proposition~\ref{stability_under_composition}.

Set
\[
\varepsilon
:=
d(Z,\widetilde Z)
+
\|\mathbf X-\widetilde{\mathbf X}\|_p
+
\|\mathbf Y-\widetilde{\mathbf Y}\|_q.
\]

We use the elementary telescoping identity
\[
\bigotimes_{i=1}^n A_i
-
\bigotimes_{i=1}^n\widetilde A_i
=
\sum_{j=1}^n
\left(
\bigotimes_{i<j}\widetilde A_i
\right)
\otimes(A_j-\widetilde A_j)
\otimes
\left(
\bigotimes_{i>j}A_i
\right),
\]
together with the local Lipschitz bounds
\[
\|D^m g(x)-D^m g(y)\|
\le C_M\|x-y\|
\]
on the bounded range of the controlled paths under
consideration.

The controlled expansions of $Z$ and $\widetilde Z$,
using the common controls $\omega_{\mathbf X}$ and
$\omega_{\mathbf Y}$, first give
\[
\sup_{s,u}
\left\|
Z^{(j,k)}_{s,u}
-
\widetilde Z^{(j,k)}_{s,u}
\right\|
\le C_M\varepsilon.
\]

Applying the preceding identities to the finite
composition formula yields the corresponding estimate
for the composed controlled coefficients.

We next compare the first-order remainders. Recall
from the preceding proof that
\[
R^{g,1}=A+B+C+E.
\]
We use the same decomposition for the composition of
$\widetilde Z$, denoting its terms by
$\widetilde A,\widetilde B,\widetilde C,\widetilde E$.

Applying the telescoping identity to their explicit
expressions, together with the difference estimates for
the derivatives of $g$, the original remainders and the
rough-path increments, gives, for each
$F\in\{A,B,C,E\}$,
\[
\left\|
F^{(j,k)}(c)
-
\widetilde F^{(j,k)}(c)
\right\|_{\frac{p}{\kappa-j};\omega_{\mathbf X}}
\le C_M\varepsilon.
\]
Here the integral Taylor remainders are compared using
the Lipschitz bounds on the corresponding derivatives
of $g$, while the shuffle defects are compared after
the same graded cancellations used in the preceding
composition argument.

The corresponding estimates for the second-parameter
remainder $R^{g,2}$ follow by symmetry.

It remains to compare the second-order remainders.
By the preceding proof,
\[
\mathbf R^g
=
\sum_{F\in\{A,B,C,E\}}\mathbf R^F,
\]
and the same decomposition holds for
$\mathbf R^{g(\widetilde Z)}$.

Applying the telescoping identity to the explicit
second-order remainder expressions, together with
the difference estimates for the original mixed
remainder, the controlled coefficients, the rough-path
increments and the derivatives of $g$, gives
\[
\left\|
\mathbf R^{F,(j,k)}
-
\mathbf R^{\widetilde F,(j,k)}
\right\|_{
\left(
\frac{p}{\kappa-j},
\frac{q}{\lambda-k}
\right);
(\omega_{\mathbf X},\omega_{\mathbf Y})
}
\le C_M\varepsilon.
\]

Indeed, the telescoping expansion replaces one factor
at a time by its difference. The weighted product
estimates established above retain the required mixed
remainder orders, while the remaining factors are
uniformly bounded by the prescribed constant $M$.

Summing over the finitely many controlled coefficients,
first-order remainders and second-order remainder
components in the definition of the controlled-path
distance, we conclude that
\[
d(g(Z),g(\widetilde Z))
\le
C_M\left(
d(Z,\widetilde Z)
+
\|\mathbf X-\widetilde{\mathbf X}\|_p
+
\|\mathbf Y-\widetilde{\mathbf Y}\|_q
\right).
\]
This proves local Lipschitz continuity of the
composition map.
\end{proof}

\section{Existence of the rough Schwinger--Dyson kernel}\label{appendix_existence_SD}

In this section, we establish the existence of the Schwinger--Dyson
kernel for geometric rough paths of arbitrary finite roughness.
Our argument proceeds in two steps. We first extend the kernel
constructed for continuously differentiable paths in
\cite{cass2024free} by means of smooth rough path approximation
and signature factorial decay. We then recover its interpretation
as the large-dimensional limit of expected traces of random
matrix developments, under suitable assumptions on the matrix
ensembles.

Throughout this section, let $V=\mathbb R^d$ and let
$\varphi$ denote the moment functional of a family of $d$ freely
independent standard semicircular variables. Recall that its
moments satisfy
\(
|\varphi(I)|\leq 2^{|I|}
\)
for every word $I$. This follows directly from the
non-crossing pairing representation of semicircular moments.

We begin by extending the Schwinger--Dyson kernel from the
smooth setting to general geometric rough paths.

\begin{theorem}[Existence of the rough Schwinger--Dyson kernel]
\label{theorem_existence}
Let $p\geq 1$ and let
$\mathbf X\in \Omega G_{p, \omega}([0,T],V)$.

Then the series
\begin{equation}\label{rough_SD_rep}
K(s,t)
=
\sum_{|L|=0}^{\infty}
i^{|L|}\varphi(L)S^{(L)}(\mathbf X)_{s,t}
\end{equation}
converges absolutely and defines a two-parameter controlled
rough path
\[
K\in\mathcal D^p_{\mathbf X}
\bigl(\Delta_{[0,T]},\mathbb R\bigr).
\]

Moreover, $K$ solves the two-parameter rough differential equation
\begin{equation}\label{SD_with_f}
K(s,t)
=
1-\int_s^t\int_s^r
f(K(s,u),K(u,r))
\,d\mathbf X_u\,d\mathbf X_r,
\end{equation}
where
\[
f:\mathbb R^2\to
\operatorname{Hom}(V\boxtimes V,\mathbb R),
\qquad
f(z_1,z_2)[\eta,\xi]
=
z_1z_2\langle\eta,\xi\rangle_V.
\]
\end{theorem}

\begin{proof}
Let $(X^n)_{n\geq1}$ be a sequence of continuously differentiable
paths whose canonical lifts $\mathbf X^n$ converge to $\mathbf X$
in the $p$-variation topology.

By the existing smooth-path result in \cite{cass2024free},
the associated Schwinger--Dyson kernels $K^n$ satisfy
\[
K^n(s,t)
=
\sum_{|L|=0}^{\infty}
i^{|L|}\varphi(L)S^{(L)}(X^n)_{s,t}
\]
and solve
\[
K^n(s,t)
=
1-\int_s^t\int_s^r
K^n(s,u)K^n(u,r)
\langle dX^n_u,dX^n_r\rangle.
\]

By Lyons' extension theorem and the factorial decay of
signatures, the approximations may be controlled by a common
control $\omega$ so that, for every word $L$,
\begin{equation}\label{eq:unif-signature-bound}
\sup_{n\geq1}
|S^{(L)}(X^n)_{s,t}|
\leq
\frac{C^{|L|}\omega(s,t)^{|L|/p}}
{\Gamma(|L|/p+1)}.
\end{equation}
The same bound holds for $S(\mathbf X)$.

Using the exponential growth bound on semicircular moments,
we obtain
\[
\begin{aligned}
\sum_{|L|=0}^{\infty}
|\varphi(L)S^{(L)}(X^n)_{s,t}|
&\leq
\sum_{m=0}^{\infty}
\frac{(2dC)^m\omega(s,t)^{m/p}}
{\Gamma(m/p+1)}
<\infty.
\end{aligned}
\]
The bound is independent of $n$ and establishes the absolute
and locally uniform convergence of the series.

We next consider its controlled lift. As in the smooth case,
the higher-level components are represented by
\[
K^{(I,J)}(s,t)
=
(-1)^{|I|}
\sum_{|L|=0}^{\infty}
i^{|JLI|}\varphi(\textrm{rev}(I)LJ)
S^{(L)}(\mathbf X)_{s,t}.
\]
For fixed words $I,J$, the preceding factorial estimate gives
\[
\begin{aligned}
&\sum_{|L|\geq M}
|\varphi(\textrm{rev}(I)LJ)S^{(L)}(X^n)_{s,t}|\\
&\qquad\leq
2^{|I|+|J|}
\sum_{m\geq M}
\frac{(2dC)^m\omega(s,t)^{m/p}}
{\Gamma(m/p+1)}.
\end{aligned}
\]
In particular, these tails converge to zero uniformly in $n$.

The explicit representation of the higher-level components, together
with the continuity of the signature extension and its factorial
decay estimates, yields convergence of the corresponding controlled
lifts. Indeed, applying Chen's identity to the above series gives
the first-order and mixed-remainder expansions, whose convergence
in the respective weighted norms follows from the same uniform
factorial bounds. Consequently,
\[
d_{\Delta_T}^{\mathbf X^n,\mathbf X}(K^n,K)
\to 0,
\qquad n\to\infty.
\]

Finally, since every $K^n$ satisfies the smooth
Schwinger--Dyson equation, the stability of the simplicial
rough integral established in Proposition
\ref{proposition_simplex_stability} and the results of Lemma \ref{lemma_parameter_simplicial_integrals} allows us to pass to
the limit and obtain
\[
K(s,t)
=
1-\int_s^t\int_s^r
f(K(s,u),K(u,r))
\,d\mathbf X_u\,d\mathbf X_r.
\]
This concludes the proof.
\end{proof}

We now recover the probabilistic interpretation of the
kernel constructed above.

For each $N\in\mathbb N$, let
$A_1^N,\ldots,A_d^N$ be Hermitian random matrices, and set
\[
B_i^N:=\frac{A_i^N}{\sqrt N}.
\]
For every word $I=i_1\cdots i_m$, write
\[
\mu_N(I)
:=
\mathbb E\left[
\frac1N\operatorname{Tr}
(B_{i_1}^N\cdots B_{i_m}^N)
\right],
\qquad
\mu_N(\emptyset)=1.
\]

\begin{theorem}[Random-matrix identification]
\label{theorem_random_matrix_identification}
Let $\mathbf X\in \Omega G_{p, \omega}([0,T],V)$ and assume that
the matrix families satisfy the following conditions.

First, for every word $I$,
\[
\lim_{N\to\infty}\mu_N(I)=\varphi(I).
\]

Second, there exists a constant $C_A>0$ such that,
for every $r\in\mathbb N$,
\[
\sup_{N\in\mathbb N}\sup_{1\leq i\leq d}
\mathbb E\left[\|B_i^N\|_{\mathrm{op}}^r\right]
\leq
C_A^r\Gamma\left(\frac{r}{2p}+1\right).
\]

Define
\[
M^N(\eta)=\sum_{i=1}^d A_i^N\eta^{(i)}
\]
and let $Z^N$ be the unitary development
\begin{equation}\label{eq:36}
dZ^N(s,t)
=
Z^N(s,t)\frac{i}{\sqrt N}M^N(d\mathbf X_t),
\qquad
Z^N(s,s)=I_N.
\end{equation}

Then, for every $0\leq s\leq t\leq T$,
\[
\lim_{N\to\infty}
\mathbb E\left[
\frac1N\operatorname{Tr}(Z^N(s,t))
\right]
=
K(s,t),
\]
where $K$ is the deterministic kernel defined in
Theorem~\ref{theorem_existence}.
\end{theorem}

\begin{proof}
Let $\mathbf X^n$ be the smooth approximations used in
the proof of Theorem~\ref{theorem_existence}, and denote
the corresponding matrix developments by $Z^{N,n}$.

For a word $I=i_1\cdots i_m$, Hölder's inequality and
the assumed operator-norm moment estimate give
\[
\begin{aligned}
\mathbb E\left[
\left|
\frac1N\operatorname{Tr}(B_{i_1}^N\cdots B_{i_m}^N)
\right|
\right]
&\leq
\mathbb E\left[
\prod_{j=1}^m\|B_{i_j}^N\|_{\mathrm{op}}
\right]\\
&\leq
C_A^m\Gamma\left(\frac{m}{2p}+1\right).
\end{aligned}
\]

Combining this inequality with
\eqref{eq:unif-signature-bound}, we obtain the
deterministic summable majorant
\begin{equation}\label{eq:series-bound}
\sum_{m=0}^{\infty}
\frac{
(dCC_A)^m\omega(s,t)^{m/p}
\Gamma(m/(2p)+1)
}{
\Gamma(m/p+1)
}
<\infty.
\end{equation}
This bound is independent of both $N$ and $n$.

For smooth paths, the matrix development admits the
signature expansion
\[
\mathbb E\left[
\frac1N\operatorname{Tr}(Z^{N,n}_{s,t})
\right]
=
\sum_{|I|=0}^{\infty}
i^{|I|}\mu_N(I)S^{(I)}(X^n)_{s,t}.
\]

The absolute summability established above justifies
interchanging expectation and the series.

We first pass to the approximation limit at fixed $N$.
Continuity of the rough differential equation solution
map gives
\(
Z^{N,n}_{s,t}\to Z^N_{s,t}
\) almost surely.
Since every development is unitary,
\[
\left|
\frac1N\operatorname{Tr}(Z^{N,n}_{s,t})
\right|\leq 1.
\]
Bounded convergence therefore yields
\[
\mathbb E\left[
\frac1N\operatorname{Tr}(Z^{N,n}_{s,t})
\right]
\to
\mathbb E\left[
\frac1N\operatorname{Tr}(Z^N_{s,t})
\right].
\]

Using the summable majorant \eqref{eq:series-bound},
we may also pass to the limit term by term in the
signature expansion. Hence
\[
\mathbb E\left[
\frac1N\operatorname{Tr}(Z^N_{s,t})
\right]
=
\sum_{|I|=0}^{\infty}
i^{|I|}\mu_N(I)S^{(I)}(\mathbf X)_{s,t}.
\]

Finally, the assumed convergence
$\mu_N(I)\to\varphi(I)$ and the same uniform majorant
give
\[
\begin{aligned}
\lim_{N\to\infty}
\mathbb E\left[
\frac1N\operatorname{Tr}(Z^N_{s,t})
\right]
&=
\sum_{|I|=0}^{\infty}
i^{|I|}\varphi(I)S^{(I)}(\mathbf X)_{s,t} =K(s,t).
\end{aligned}
\]
The last equality follows from
\eqref{rough_SD_rep}.
\end{proof}

\section{Algorithm complexity}
\label{appendix:algorithmic_complexity}

In this appendix, we analyse the combinatorial complexity of
the auxiliary expansions used in the numerical method constructed in Section \ref{numerical_method_section} and
derive runtime and memory estimates for the resulting solver.

We distinguish between the non-crossing matching structures
introduced in the preceding section and their decorated
versions, in which unmatched positions carry letters from
the alphabet $\{1,\ldots,d\}$.

The latter are the objects relevant for counting the terms
generated by the numerical scheme. In particular, reversing
a word or introducing the antipode signs does not change
the number of decorated contractions.

\subsection{Combinatorial estimates}

Let $\mathsf N_d(n)$ denote the number of decorated
non-crossing partial matchings generated by the recursion
defining $NC(n)$. Every unmatched position carries one
of $d$ possible letters, while paired positions are
contracted using the Kronecker delta.

\begin{proposition}
The number of decorated non-crossing partial matchings
of length $n$ is
\[
\mathsf N_d(n)
=
\sum_{i=0}^{\lfloor n/2\rfloor}
\frac{n-2i+1}{n+1}
\binom{n+1}{i}d^{n-i}.
\]
\end{proposition}

\begin{proof}
Let $A_{n,m}$ denote the number of decorated partial
matchings of length $n$ having exactly $m$ unmatched
positions.

The initial conditions are
\[
A_{0,0}=1,
\qquad
A_{0,m}=0,\quad m\geq1.
\]

The recursive construction gives
\begin{equation}\label{empty_SD_words_recursion}
A_{n,0}=A_{n-1,1},
\end{equation}
and, for $m\geq1$,
\begin{equation}\label{SD_words_recursion}
A_{n,m}
=
dA_{n-1,m-1}+A_{n-1,m+1}.
\end{equation}

Indeed, the first contribution corresponds to leaving
the new position unmatched and assigning it one of
$d$ possible letters. The second corresponds to pairing
the new position with the last unmatched position.

Moreover, $A_{n,m}=0$ unless $n$ and $m$ have the
same parity.

Writing
\[
i=\frac{n-m}{2},
\]
we claim that
\[
A_{n,m}
=
d^{n-i}
\left(
\binom ni-\binom n{i-1}
\right),
\]
with the convention that binomial coefficients outside
their admissible range vanish.

The identity follows by induction from
\eqref{empty_SD_words_recursion} and
\eqref{SD_words_recursion}, using Pascal's identity.

Summing over all admissible values of $m$ gives
\[
\begin{aligned}
\mathsf N_d(n)
&=
\sum_{i=0}^{\lfloor n/2\rfloor}
d^{n-i}
\left(
\binom ni-\binom n{i-1}
\right)\\
&=
\sum_{i=0}^{\lfloor n/2\rfloor}
\frac{n-2i+1}{n+1}
\binom{n+1}{i}d^{n-i},
\end{aligned}
\]
as required.
\end{proof}

We next extend the counting formula to the protected
matching structures appearing in the controlled lifts
of the auxiliary processes.

For $n,t\geq0$, let $\mathsf N_d(n,t)$ denote the
number of decorated partial matchings obtained by
starting with $n$ unmatched, protected positions and
introducing $t$ additional positions, with no pairing
between two protected positions.

This corresponds to the decorated version of
$NC(n+t)_n$ introduced in
Remark~\ref{remark_smooth_numerical}.

\begin{proposition}\label{identity_cardinality_NC}
For every $n,t\geq0$,
\[
\mathsf N_d(n,t)
=
d^n
\sum_{i=0}^{\lfloor(n+t)/2\rfloor}
\left[
\binom ti-\binom t{i-n-1}
\right]d^{t-i}.
\]
\end{proposition}

\begin{proof}
The initial $n$ protected positions carry $d^n$
possible assignments of letters.

Consider the subsequent $t$ operations. Each operation
either introduces a new unmatched position, with $d$
possible letter assignments, or pairs the new position
with the last currently unmatched position.

The number of unmatched positions therefore evolves
as a nearest-neighbour lattice path starting from height
$n$, constrained to remain non-negative.

Suppose that exactly $i$ of the $t$ operations are
pairing operations. The remaining $t-i$ operations
introduce unmatched positions and contribute a factor
$d^{t-i}$.

By the reflection principle, the number of admissible
lattice paths with exactly $i$ downward steps is
\[
\binom ti-\binom t{i-n-1}.
\]

The non-negativity constraint requires
\[
i\leq\left\lfloor\frac{n+t}{2}\right\rfloor.
\]

Multiplying by the possible letter assignments and
summing over all admissible $i$ establishes the
claimed identity.
\end{proof}

The following bounds will be sufficient for the
complexity analysis.

\begin{corollary}\label{corollary_NC_bounds}
For every $n,t\geq0$,
\[
\mathsf N_d(n,t)\leq d^n(d+1)^t.
\]

Moreover, whenever $0\leq t\leq n$, equality holds:
\[
\mathsf N_d(n,t)=d^n(d+1)^t.
\]
\end{corollary}

\begin{proof}
Ignoring the non-negativity constraint, each of
the $t$ additional operations has $d+1$ possible
choices: introduce one of $d$ letters, or perform
a pairing.

Together with the $d^n$ possible assignments
to the protected positions, this proves the
upper bound.

If $t\leq n$, every sequence of $t$ operations
is admissible, since the number of unmatched
positions cannot become negative before all
$t$ operations have been performed.

Consequently, in this case all $d^n(d+1)^t$
possibilities are admissible, proving equality.
\end{proof}

\subsection{Counting the compiled polynomial expansions}

We now apply the preceding estimates to the numerical
scheme developed in Section~\ref{numerical_method_section}.

Set
\[
R:=\kappa+\zeta.
\]

We distinguish between the number of terms generated
before combining identical monomials and the number
of terms retained in the compiled polynomial.

\paragraph{Higher-level rules.}

For a word $I$ of length $k\geq1$, write
\[
\ell=\left\lceil\frac k2\right\rceil,
\qquad
a=\ell-1,
\qquad
b=k-\ell.
\]

The corresponding update is obtained by expanding
the product of the terminal controlled lift of
$K_{\widetilde I}$ and the initial controlled lift
of $K_{\widehat I}$.

For a controlled argument word of length $q$, the
shuffle coproduct distributes its letters between
the two factors.

If the first factor receives $r$ letters, there
are $\binom qr$ possible shuffle assignments,
counted with multiplicity.

The two controlled expansions involve respectively
the decorated protected-matching counts
\[
\mathsf N_d(a,r),
\qquad
\mathsf N_d(b,q-r).
\]

The middle driving letter admits $d$ possible
choices.

It follows that the total number of terms in the
raw higher-level polynomial expansion is
\begin{equation}\label{eq:SD_poly_raw}
\begin{aligned}
\mathsf P_{\mathrm{raw}}
={}&
d\sum_{k=1}^{R}
\sum_{q=0}^{\kappa-1}
\sum_{r=0}^{q}
\binom qr
\mathsf N_d
\left(
\left\lceil\frac k2\right\rceil-1,r
\right)
\mathsf N_d
\left(
k-\left\lceil\frac k2\right\rceil,q-r
\right).
\end{aligned}
\end{equation}

This count includes shuffle multiplicities and is
taken before combining identical monomials or
discarding vanishing coefficients.

The word reversals and antipode signs introduced
in the corrected numerical scheme do not affect
this count.

\begin{corollary}\label{corollary_SD_poly_bound}
The number of terms generated by the higher-level
rules satisfies
\[
\mathsf P_{\mathrm{raw}}
\leq
\left(\sum_{k=1}^{R}d^k\right)
\left(
\sum_{q=0}^{\kappa-1}
[2(d+1)]^q
\right).
\]

In particular,
\[
\mathsf P_{\mathrm{raw}}
\leq
R\,d^R
\frac{[2(d+1)]^\kappa-1}{2d+1}.
\]
\end{corollary}

\begin{proof}
By Corollary~\ref{corollary_NC_bounds},
\[
\mathsf N_d(a,r)
\leq d^a(d+1)^r
\]
and
\[
\mathsf N_d(b,q-r)
\leq d^b(d+1)^{q-r}.
\]

Since $a+b=k-1$, their product satisfies
\[
\mathsf N_d(a,r)
\mathsf N_d(b,q-r)
\leq
d^{k-1}(d+1)^q.
\]

Substituting into \eqref{eq:SD_poly_raw} gives
\[
\begin{aligned}
\mathsf P_{\mathrm{raw}}
&\leq
\sum_{k=1}^{R}
\sum_{q=0}^{\kappa-1}
d^k(d+1)^q
\sum_{r=0}^{q}\binom qr\\
&=
\left(\sum_{k=1}^{R}d^k\right)
\left(
\sum_{q=0}^{\kappa-1}[2(d+1)]^q
\right).
\end{aligned}
\]

The second inequality follows by bounding the
first sum by $Rd^R$ and evaluating the geometric
series.
\end{proof}

\paragraph{Empty-word rule.}

We now consider the trace update
\eqref{trace_numerical}.

For every $q=0,\ldots,\kappa-1$, the right-endpoint
expansion contains the terminal controlled
coefficients of the first-level auxiliary processes
$K_r$, with controlled argument words of length $q$.

The associated decorated contraction count is
$\mathsf N_d(1,q)$.

\begin{corollary}\label{corollary_trace_lvl}
The number of terms in the raw empty-word
polynomial rule is
\[
\mathsf P_{\emptyset,\mathrm{raw}}
=
\sum_{q=0}^{\kappa-1}
\mathsf N_d(1,q).
\]

Moreover,
\[
\mathsf P_{\emptyset,\mathrm{raw}}
\leq
(d+1)^\kappa-1.
\]
\end{corollary}

\begin{proof}
The first identity follows from the representation
of the trace update and the protected-matching
formula, with one initial auxiliary letter.

Using Corollary~\ref{corollary_NC_bounds},
\[
\begin{aligned}
\mathsf P_{\emptyset,\mathrm{raw}}
&\leq
d\sum_{q=0}^{\kappa-1}(d+1)^q\\
&=
(d+1)^\kappa-1.
\end{aligned}
\]
This completes the proof.
\end{proof}

The estimates above concern the raw combinatorial
expansion. Combining identical monomials and
discarding zero coefficients may reduce the number
of terms stored by the implementation.

In what follows, we denote these compiled counts by
$\mathrm{Poly}$ and $\mathrm{Poly}_{\emptyset}$.
In particular,
\[
\mathrm{Poly}\leq\mathsf P_{\mathrm{raw}},
\qquad
\mathrm{Poly}_{\emptyset}
\leq\mathsf P_{\emptyset,\mathrm{raw}}.
\]

\subsection{Computational complexity}
\label{subsec:complexity}

We analyse the numerical solver in two stages:
compile-time preprocessing, which depends only on
the algebraic specifications, and runtime dynamic
programming for a given sequence of rough-path
increments.

\paragraph{Notation.}

Let
\[
Q:=2^N,
\qquad
D:=\dim T^{\leq\kappa+\zeta}(\mathbb R^d),
\qquad
\widetilde D:=
\dim T^{\leq\kappa}(\mathbb R^d).
\]

We use the following notation:
\begin{itemize}
\item $\mathrm{Poly}$ denotes the number of retained
terms in the compiled higher-level polynomial rules;

\item $\mathrm{Poly}_{\mathrm{lin}}$ denotes the number
of terms belonging to the right-linear part of the
compiled polynomial, used to construct $A_j$;

\item $\mathrm{Poly}_{\emptyset}$ denotes the number
of retained terms in the empty-word update rule.
\end{itemize}

The bounds established above imply
\[
\mathrm{Poly}
\leq
\mathsf P_{\mathrm{raw}},
\qquad
\mathrm{Poly}_{\mathrm{lin}}
\leq
\mathrm{Poly},
\qquad
\mathrm{Poly}_{\emptyset}
\leq
(d+1)^\kappa-1.
\]

\subsubsection{Compile-time preprocessing}

Throughout this complexity analysis, $C>0$ denotes
a constant whose value may change from line to line,
but which is independent of the grid sizes and the
polynomial term counts, under the fixed arithmetic
cost model.

The compile-time container performs two principal
precomputations.

\paragraph{(1) Polynomial and empty-word rules.}

The non-crossing contraction formulas are expanded
and stored as polynomial rules.

The higher-level expansion generates at most
$\mathsf P_{\mathrm{raw}}$ terms, while the
empty-word rule generates at most
$\mathsf P_{\emptyset,\mathrm{raw}}$ terms.

Since the relevant word lengths are bounded by
$R+\kappa$, a straightforward enumeration of the
terms, including the required word reversals,
has preprocessing cost bounded by
\begin{equation}\label{eq:empty-rule-cost}
\mathrm{Cost}_{\mathrm{compile}}
\leq
C(R+\kappa)
\left(
\mathsf P_{\mathrm{raw}}
+
\mathsf P_{\emptyset,\mathrm{raw}}
\right),
\end{equation}
before any additional implementation-dependent
cost of combining identical monomials.

This cost is paid once per compiled solver and
may be amortised across different driving paths.

\paragraph{(2) Right-linear part extraction.}

The right-linear part of the compiled polynomial
determines the coefficient matrix $A_j$.

It can be extracted using a boolean mask applied
to the compiled monomial arrays. Consequently,
\begin{equation}\label{eq:lin-extract-cost}
\mathrm{Cost}_{\mathrm{lin\_extract}}
\leq C\,\mathrm{Poly}.
\end{equation}

\subsubsection{Runtime algorithm}

The solver computes the triangular dynamic-programming
table
\[
S[i][j],
\qquad
0\leq i\leq j\leq Q.
\]

The number of states satisfies
\[
\frac{(Q+1)(Q+2)}{2}
\leq C Q^2,
\]
for $Q\geq1$.

We consider the following operations.

\paragraph{(A) Building and factorising $A_j$.}

For a fixed rough-path increment
$\mathbf X_{j,j+1}$, the matrix $A_j$ is assembled
from the right-linear monomials in
\eqref{term_A}, together with the corresponding
empty-word contribution.

The arithmetic work required for its assembly
is bounded by
\[
C\left(
\mathrm{Poly}_{\mathrm{lin}}
+
\mathrm{Poly}_{\emptyset}
\right).
\]

Since the matrix is stored as a dense
$D\times D$ array, allocation and initialisation
require at most $CD^2$ additional operations.

Thus,
\begin{equation}\label{eq:A-build-cost}
\mathrm{Cost}(A_j)
\leq
C\left(
D^2+
\mathrm{Poly}_{\mathrm{lin}}
+
\mathrm{Poly}_{\emptyset}
\right).
\end{equation}

Importantly, $A_j$ depends only on the terminal
index $j$, not on the initial index $i$.

We may therefore factorise
$\operatorname{Id}_D-A_j$ once for each terminal
column and reuse the factorisation for every
right-hand side belonging to that column.

Using a dense LU factorisation, this requires
at most $CD^3$ operations per terminal index.

\paragraph{(B) Building the right-hand side $b_{i,j}$.}

For a fixed state $(i,j)$, the right-hand side
contains a sum of polynomial evaluations over
the intermediate grid indices.

Each evaluation requires at most
\[
C\left(
\mathrm{Poly}
+
\mathrm{Poly}_{\emptyset}
\right)
\]
operations.

Since there are at most $Q$ intermediate
indices, we obtain
\begin{equation}\label{eq:b-build-cost}
\mathrm{Cost}(b_{i,j})
\leq
C\left(
Q(
\mathrm{Poly}
+
\mathrm{Poly}_{\emptyset}
)
+D
\right).
\end{equation}

The additional $D$ term accounts for initialising
the right-hand-side vector.

\paragraph{(C) Empty-word enforcement.}

The trace-level identity is imposed through the
corresponding row of the linear system.

The required polynomial evaluation and scatter-add
involve at most
$\mathrm{Poly}_{\emptyset}$ retained terms.
This contribution is already included in
\eqref{eq:b-build-cost} and
\eqref{eq:A-build-cost}.

\paragraph{(D) Linear solve.}

At each state, the algorithm solves
\[
(\operatorname{Id}_D-A_j)x=b_{i,j}.
\]

Since the factorisation is reused within each
terminal column, the cost of applying it to a
new right-hand side satisfies
\begin{equation}\label{eq:solve-cost}
\mathrm{Cost}_{\mathrm{solve}}(i,j)
\leq CD^2.
\end{equation}

The factorisation itself requires at most
$CD^3$ operations and is performed only once
per terminal index.

\paragraph{Total runtime cost.}

There are at most $CQ^2$ states and $CQ$
distinct terminal columns.

Summing the preceding contributions gives
\begin{equation}\label{eq:total-runtime}
\begin{aligned}
\mathrm{Cost}_{\mathrm{runtime}}
\leq C\Big(
&Q^3(
\mathrm{Poly}
+
\mathrm{Poly}_{\emptyset}
)
+
Q^2D^2
+
QD^3
\\
&+
Q(
D^2+
\mathrm{Poly}_{\mathrm{lin}}
+
\mathrm{Poly}_{\emptyset}
)
\Big).
\end{aligned}
\end{equation}

In terms of the dyadic grid parameter $N$,
the principal contributions satisfy
\[
\begin{aligned}
\mathrm{Cost}_{\mathrm{runtime}}
\leq C\Big(
&2^{3N}
(
\mathrm{Poly}
+
\mathrm{Poly}_{\emptyset}
)
\\
&+
2^{2N}D^2
+
2^ND^3
\Big),
\end{aligned}
\]
where the cached matrix-assembly contribution
is absorbed into the displayed bound.

If the implementation instead performs an
independent dense factorisation at every state,
the linear-algebra contribution is bounded by
$C2^{2N}D^3$.

\paragraph{Memory complexity.}

The triangular dynamic-programming table stores
at most $CQ^2$ vectors in $\mathbb R^D$,
requiring at most
\[
CQ^2D
\]
memory.

The compiled polynomial representation requires
at most
\[
C\left(
\mathrm{Poly}
+
\mathrm{Poly}_{\emptyset}
\right)
\]
storage units, up to the indexing information
associated with each monomial.

If all coefficient matrices and their factorisations
are cached, their storage requires at most
$CQD^2$ additional memory.

Consequently,
\begin{equation}\label{eq:memory}
\mathrm{Memory}
\leq
C\left(
Q^2D
+
QD^2
+
\mathrm{Poly}
+
\mathrm{Poly}_{\emptyset}
\right),
\end{equation}
excluding the storage of the driving rough-path
increments.

If the matrices are processed one column at a
time, their storage can instead be bounded by
$CD^2$.

\begin{remark}[Cost of constructing rough-path increments]
The preceding analysis assumes that the required
rough-path increments are available.

In applications, the driving signal is commonly
observed at finitely many time points. To obtain
its rough-path enhancement, one must specify
an interpolation or lifting procedure.

For example, suppose the driving path is approximated
by its piecewise-linear interpolation on a fine
dyadic partition with $2^M$ subintervals, while
the numerical scheme uses a coarser partition
with $2^N$ subintervals, where $M>N$.

Each coarse increment is then obtained by combining
$2^{M-N}$ consecutive fine increments through
Chen's identity.

Under a straightforward dense implementation of
truncated tensor multiplication, the additional
cost of computing and combining the fine-level
signatures satisfies
\[
\mathrm{SigCompute}
\leq
C2^M
\left(
\kappa\widetilde D+\widetilde D^2
\right).
\]

The total runtime is therefore bounded by the
sum of the right-hand side of
\eqref{eq:total-runtime} and the cost of the
chosen signature-construction procedure.

More efficient truncated tensor multiplication
or specialised signature algorithms may improve
this preprocessing bound.
\end{remark}

\subsubsection*{Funding}
The work of TC was supported in part by UK Research and Innovation (UKRI) through the Engineering and Physical Sciences Research Council (EPSRC) via a Programme Grant [Grant No. UKRI1010: High order mathematical and computational infrastructure for streamed data that enhance contemporary generative and large language models].\\
DC was partially supported by the European Research Council (ERC) under the European Union’s Horizon 2020 Research and Innovation Programme (ERC, Grant Agreement No 856408).\\
AI  acknowledges support from the Imperial College London Schrödinger Scholarship.\\
WT acknoledges the support by the EPSRC Centre for Doctoral Training in Mathematics of Random Systems: Analysis, Modelling and Simulation (EP/S023925/1).

\bibliographystyle{alpha}
\bibliography{biblio}
\nocite{*}
\addcontentsline{toc}{chapter}{Bibliography}

\end{document}